\documentclass[a4paper,10pt]{amsart}
\usepackage[]{amsmath,amsfonts,amssymb,amsthm,amscd,amsxtra,amstext,latexsym,syntonly,tabularx}
\usepackage{graphicx}
\usepackage[all,dvips]{xy}
\usepackage[pagebackref]{hyperref}

\theoremstyle{plain}
\newtheorem{xx}{xx}[section]
\newtheorem{thm}[xx]{Theorem}
\newtheorem*{thm*}{Theorem}
\newtheorem{prop}[xx]{Proposition}
\newtheorem{cor}[xx]{Corollary}
\newtheorem{lem}[xx]{Lemma}
\newtheorem*{lem*}{Lemma}

\theoremstyle{definition}
\newtheorem{defn}[xx]{Definition}
\newtheorem{ex}[xx]{Example}

\theoremstyle{remark}
\newtheorem{rem}[xx]{Remark}

\numberwithin{equation}{xx}

\DeclareMathOperator{\Add}{Add}
\DeclareMathOperator{\at}{at}

\DeclareMathOperator{\Arr}{Arr}

\DeclareMathOperator{\coker}{coker}
\DeclareMathOperator{\Def}{Def}

\DeclareMathOperator{\depth}{depth}
\DeclareMathOperator{\Der}{Der}
\DeclareMathOperator{\cE}{E}

\DeclareMathOperator{\End}{End}

\DeclareMathOperator{\ev}{ev}
\DeclareMathOperator{\Ext}{Ext}

\DeclareMathOperator{\uExt}{\ul{\Ext}}

\DeclareMathOperator{\grade}{grade}

\DeclareMathOperator{\cH}{H}

\DeclareMathOperator{\Hom}{Hom}
\DeclareMathOperator{\uHom}{\ul{\Hom}}
\DeclareMathOperator{\sHom}{\mathcal{H}\!\mathit{om}}
\DeclareMathOperator{\id}{id}
\DeclareMathOperator{\im}{im}
\DeclareMathOperator{\Injdim}{inj{.}dim}

\DeclareMathOperator{\mSpec}{\fr{m}-Spec}

\DeclareMathOperator{\ob}{ob}
\DeclareMathOperator{\pdim}{pdim}

\DeclareMathOperator{\Quot}{Quot}

\DeclareMathOperator{\Resdim}{res{.}dim}
\DeclareMathOperator{\Rk}{rk}
\DeclareMathOperator{\Sets}{\cat{Sets}}

\DeclareMathOperator{\Spec}{Spec}

\DeclareMathOperator{\Supp}{Supp}

\DeclareMathOperator{\Syz}{Syz}
\DeclareMathOperator{\Tor}{Tor}

\DeclareMathOperator{\To}{\mathit{T}^{\mspace{-0.5mu}\mathrm{o}}}
\DeclareMathOperator{\bTo}{\mathit{\bar{T}}^{\mspace{-0.5mu}\mathrm{o}}}

\newcommand{\ra}{\rightarrow}
\newcommand{\la}{\leftarrow}
\newcommand{\lra}{\longrightarrow}

\newcommand{\thr}{\twoheadrightarrow}

\newcommand{\hra}{\hookrightarrow}

\newcommand{\co}{\colon}

\newcommand{\ot}{{\otimes}}
\newcommand{\vare}{\varepsilon}

\newcommand{\vG}{\varGamma}
\newcommand{\vL}{\varLambda}

\newcommand{\La}{\Leftarrow}

\newcommand{\Ra}{\Rightarrow}

\newcommand{\Llra}{\Longleftrightarrow}
\newcommand{\sbeq}{\subseteq}

\newcommand{\limproj}{\varprojlim}
\newcommand{\liminj}{\varinjlim}

\newcommand{\QH}{{{}_{\vL}\cat{H}}}
\newcommand{\QR}{{{}_{\vL}\cat{R}}}
\newcommand{\QA}{{{}_{\vL}\cat{A}}}
\newcommand{\kH}{{{}_{k}\cat{H}}}

\newcommand{\Df}{\hat{\cat{D}}{}^{\textnormal{fl}}}
\newcommand{\cDf}{\check{\cat{D}}{}^{\textnormal{fl}}}
\newcommand{\uDf}{\underline{\hat{\cat{D}}}{}^{\textnormal{fl}}}
\newcommand{\Pf}{\hat{\cat{P}}{}^{\textnormal{fl}}}
\newcommand{\MCMf}{\widehat{\cat{MCM}}{}^{\textnormal{fl}}}

\newcommand{\modf}{\cat{mod}{}^{\textnormal{fl}}}
\newcommand{\umod}{\underline{\cat{mod}}}
\newcommand{\umodf}{\underline{\cat{mod}}{}^{\textnormal{fl}}}
\newcommand{\Xf}{\hat{\cat{X}}{}^{\textnormal{fl}}}

\newcommand{\hot}{{\tilde{\otimes}}}
\newcommand{\resdim}[2]{{#1}\textnormal{-}\Resdim{#2}}
\newcommand{\injdim}[2]{{#1}\textnormal{-}\Injdim{#2}}
\newcommand{\rk}[2]{{#1}\textnormal{-}\Rk({#2})}

\renewcommand{\phi}{\varphi}
\renewcommand{\geq}{\geqslant}
\renewcommand{\leq}{\leqslant}

\newcommand{\BB}[1]{\mathbb{{#1}}}
\newcommand{\bs}[1]{\boldsymbol{{#1}}}
\newcommand{\cat}[1]{\mathsf{{#1}}}

\newcommand{\gDer}[1]{{}_{#1}{\Der}}
\newcommand{\df}[2]{{\Def}_{#2}^{#1}}

\newcommand{\fr}[1]{\mathfrak{{#1}}}

\newcommand{\gH}[1]{{}_{#1}{\cH}}
\newcommand{\hm}[4]{{\Hom}_{#2}^{#1}({#3},{#4})}
\newcommand{\uhm}[4]{{\uHom}_{#2}^{#1}({#3},{#4})}

\newcommand{\mc}[1]{\mathcal{#1}}

\newcommand{\mr}[1]{\mathrm{{#1}}}
\newcommand{\nd}[3]{{\End} _{#2}^{#1}({#3})}

\newcommand{\ol}[1]{\overline{{#1}}}

\newcommand{\Q}{\mathcal{O}}
\newcommand{\quot}[2]{\Quot^{#1}_{#2}}

\newcommand{\shm}[4]{{\sHom} _{#2}^{#1}({#3},{#4})}

\newcommand{\tn}{\textnormal}
\newcommand{\tor}[4]{{\Tor}_{#2}^{#1}({#3},{#4})}
\newcommand{\ul}[1]{\underline{{#1}}}
\newcommand{\xla}[1]{\xleftarrow{{#1}}}
\newcommand{\xra}[1]{\xrightarrow{{#1}}}
\newcommand{\xt}[4]{{\Ext} _{#2}^{#1}({#3},{#4})}
\newcommand{\uxt}[4]{{\uExt} _{#2}^{#1}({#3},{#4})}
\newcommand{\syz}[2]{{\Syz}_{#2}^{#1}}

\begin{document}
\title[CM approximation in fibred categories]
{Cohen-Macaulay approximation\\ in fibred categories}
\author{Runar Ile}
\address{Department of 
Mathematics \\University of Bergen \\
5008 Bergen \\ NORWAY}
\email{runar.ile@math.uib.no} 
\thanks{\textit{Acknowledgement}. An essential part of this work was done during a visit to Institut Mittag-Leffler for which the author is most grateful.} 
\keywords{Cohen-Macaulay map, quotient category, algebraic ring, versal deformation, Andr{\'e}-Quillen cohomology, Artin's Approximation Theorem, Kodaira-Spencer class} 
\subjclass[2010]
{Primary 13C60, 14B12; Secondary 13D02, 18D30}
\begin{abstract}
We extend the Auslander-Buchweitz axioms and prove Cohen-Macaulay approximation results for fibred categories. We show that these axioms apply for the fibred category of pairs consisting of a finite type flat family of Cohen-Macaulay rings and modules. In particular such a pair admits an approximation with a flat family of maximal Cohen-Macaulay modules and a hull with a flat family of modules with finite injective dimension. 
The existence of minimal approximations and hulls in the local, flat case implies extension of upper semi-continuous invariants. As an example of MCM approximation we define a relative version of Auslander's fundamental module.

In the second part we study the induced maps of deformation functors and deduce properties like smoothness and injectivity under general, mainly cohomological conditions on the module. We also provide deformation theory for pairs (algebra, module), e.g.\ a cohomology for such pairs, a long exact sequence linking this cohomology to the Andr{\'e}-Quillen cohomology of the algebra and the \(\Ext\) cohomology of the module, Kodaira-Spencer classes and maps including a secondary Kodaira-Spencer class, and existence of a versal family for pairs with isolated singularity. 
\end{abstract}
\maketitle
\tableofcontents
\section{Introduction}\label{sec.intro}
Axiomatic Cohen-Macaulay approximation was introduced by M. Auslander and R.-O.\ Buchweitz in \cite{aus/buc:89}. We define this theory in terms of fibred categories and obtain approximation results for various classes of flat families of modules.

Let \(A\) be a Cohen-Macaulay ring of finite Krull dimension with a canonical module \(\omega_{A}\). Let \(\cat{MCM}_{A}\) and \(\cat{FID}_{A}\) denote the categories of maximal Cohen-Macaulay modules and of finite modules with finite injective dimension, respectively. M.\ Auslander and R.-O.\ Buchweitz proved in \cite{aus/buc:89} that for any finite \(A\)-module \(N\) there exist short exact sequences
\begin{equation}\label{eq.MCMseq}
0\ra L\lra M\lra N\ra 0\qquad\text{and}\qquad 0\ra N\lra L'\lra M'\ra 0
\end{equation}
with \(M\) and \(M'\) in \(\cat{MCM}_{A}\) and \(L\) and \(L'\) in \(\cat{FID}_{A}\). The maps \(M\ra N\) and \(N\ra L'\) in \eqref{eq.MCMseq} are called a maximal Cohen-Macaulay approximation and a hull of finite injective dimension, respectively, of the module \(N\).
The association \(N\mapsto X\) for \(X\) equal to \(M,M',L\) and \(L'\) defines functors of corresponding stable categories. 
In this article we study the continuous properties of these functors.

Linear representations provided by (sheaves) of modules and the associated homological algebra play an important role in algebra and algebraic geometry, e.g.\ as a means for classification by providing invariants. 
Finite complexes have particular properties as seen in the Buchsbaum-Eisenbud acyclicity criterion and the intersection theorems of Peskine, Szpiro and Roberts. However, for a non-regular local ring \(A\), the standard homological invariants are given by the (generally) infinite minimal \(A\)-free resolutions, of which very little is known. To stay within finite complexes one can enlarge or change the category of resolving objects and Cohen-Macaulay approximation is a structured way of doing this. 

Let \(\cat{D}_{A}\) denote the subcategory \(\Add\{\omega_{A}\}\) of modules \(D\) isomorphic to direct summands of the \(\omega_{A}^{\oplus r}\). A part of the approximation result says that all the modules in \(\cat{FID}_{A}\) have finite resolutions by objects in \(\cat{D}_{A}\). In particular the MCM approximation in \eqref{eq.MCMseq} can be extended to a finite resolution
\begin{equation}\label{eq.Dres}
0\ra D^{-n}\lra D^{-n+1}\lra\dots \lra D^{-1}\lra M\lra N\ra 0
\end{equation}
with the \(D^{i}\) in \(\cat{D}_{A}\). In the case \(A\) is Gorenstein, \(\cat{D}_{A}\) equals the category of finite projective modules \(\cat{P}_{\!A}\). This generalises: By a result of R.\ Y.\ Sharp \cite{sha:75b} the functor \(\hm{}{A}{\omega_{A}}{-}\) gives an exact equivalence \(\cat{D}_{A}\simeq\cat{P}_{\!A}\), hence a finite projective resolution is associated to \(N\). In the case \(A\) is local, the  approximations and the complex can be chosen to be minimal and unique (with \(D^{i}\cong \omega_{A}^{\oplus d^{i}}\)) and in particular the \(d^{i}\) are invariants of \(N\). 

The developments since Auslander and Buchweitz' fundamental work \cite{aus/buc:89} have included studies of invariants defined by Cohen-Macaulay approximation; \cite{din:92, aus/din/sol:93, has/shi:97} among several, `injectivity' and `surjectivity' properties of the approximation maps; \cite{kat:99, yos/iso:00, kat:07}, and characterisations of quasi-homogeneous isolated singularities; cf.\ \cite{her/mar:93, mar:00b}, all exclusively in the Gorenstein case. Noteworthy is \cite{sim/str:02} where A.-M.\ Simon and J.\ R.\ Strooker related some of these invariants with Hochster's Canonical Element Conjecture and the Monomial Conjecture. In particular these conjectures are equivalent to the vanishing of the \(\delta\)-invariant of certain cyclic modules over all Gorenstein rings.
S.\ P.\ Dutta applied the existence of a FID hull to prove a relationship between two of the Serre conjectures on intersection numbers: Failure of vanishing implies failure of higher non-negativity in the Gorenstein case under certain conditions; see \cite{dut:04}.

Buchweitz' unpublished manuscript \cite{buc:86}, a precursor to \cite{aus/buc:89}, contains homological ideas which have influenced subsequent developments (e.g.\ \cite{kra:05}). Auslander and I.\ Reiten elaborated in \cite{aus/rei:91} on \cite{aus/buc:89}, mainly with a view towards artin algebras, instigating several generalisations and analogies to Cohen-Macaulay approximation.

However, the `relative' and continuous aspects have received surprisingly little attention.
M.\ Hashimoto has given several new examples of Cohen-Macaulay approximation \cite{has:00}. 
In \cite[IV 1.4.12]{has:00} an affine algebraic group \(G\) acts on a positively graded Cohen-Macaulay ring \(T\) which is flat over a regular base ring \(R\). Hashimoto considers graded maximal Cohen-Macaulay \(T\)-modules (which automatically are \(R\)-flat) and graded modules locally of finite injective dimension (not \(R\)-flat in general), all with \(G\)-action. His result (with trivial group) is hence different from our Theorem \ref{thm.flatCMapprox}. We also note some explicit \(1\)-parameter families of indecomposable finite length modules \(N_{t}\) (for many Gorenstein rings) such that the minimal MCM approximation module \(M_{t}\) is without free summands; see \cite{yos:99}.

A central part of the classification problem is to prove the existence of objects with certain properties and to estimate `how many' such objects there are.  
A natural question is thus whether there is Cohen-Macaulay approximation for flat families of modules. In Theorem \ref{thm.flatCMapprox} we give a positive answer to this question. For a Cohen-Macaulay (CM) map \(h:S\ra T\) and an \(S\)-flat and finite \(T\)-module \(\mc{N}\) there are short exact sequences of \(S\)-flat and finite \(T\)-modules
\begin{equation}\label{eq.flatMCMseq}
0\ra \mc{L}\lra \mc{M}\lra \mc{N}\ra 0\qquad\text{and}\qquad 0\ra \mc{N}\lra \mc{L}'\lra \mc{M}'\ra 0
\end{equation}
such that the fibres of these sequences give `absolute'  approximations and hulls as in the two sequences \eqref{eq.MCMseq}. Note that \(T\) in general is \emph{not} a Cohen-Macaulay ring although the fibres of \(h\) are. 
We consider a category \(\modf\) of pairs \(\xi=(h:S\ra T,\mc{N})\) and subcategories \(\cat{MCM}\), \(\cat{FID}\)  and \(\cat{D}\).  
They are fibred over the category \(\cat{CM}\) of CM maps and also fibred over the base category of noetherian rings. The approximation and the hull \eqref{eq.flatMCMseq} induce functors of certain quotient categories fibred in additive categories over \(\cat{CM}\)
\begin{equation} 
\modf/\cat{D}\ra \cat{MCM}/\cat{D}\quad \text{and}\quad\modf/\cat{D}\ra \cat{FID}/\cat{D}
\end{equation}
with analogous properties to the absolute case. If \(h:S\ra T\) is a \emph{local} CM map, there is an approximation result with minimal (and hence unique) choices of the two sequences in \eqref{eq.flatMCMseq}; see Corollaries \ref{cor.locCMapprox} and \ref{cor.minapprox}. 

A major consequence of these results is that any numerical and additive upper semi-continuous invariant of MCM or FID modules by the minimal approximations and hulls induces upper semi-continuous invariants for all finite modules; see Theorem \ref{thm.semicont}. Examples of such invariants are given by the \(\omega_{A}\)-ranks in the minimal \emph{representing complex} \(D^{*}(N)\) which is an (infinite) extension to the right of the \(\cat{D}_{A}\)-complex in \eqref{eq.Dres}.

Auslander's fundamental module \(E_{A}\) for a normal \(2\)-dimensional singularity \(\Spec A\) is given by the MCM approximation of the maximal ideal;
\begin{equation}
0\ra \omega_{A}\lra E_{A}\lra \fr{m}_{A}\ra 0
\end{equation}
which in a certain sense generates all almost split sequences for \(A\); see \cite{aus:86}. 
As a general example of flat Cohen-Macaulay approximation we define the fundamental module for any finite type CM map of pure relative dimension \(\geq 2\); see Corollary \ref{cor.Fmod}, and more generally a `fundamental' functor of projective modules in Proposition \ref{prop.Fmod}. 

An attractive feature of Auslander and Buchweitz' theory is its axiomatic formulation with several applications besides the classical case described in the first paragraph, e.g.\ coherent rings with a cotilting module, the graded case, approximation with modules of Gorenstein dimension \(0\), and coherent sheaves on a projectively embedded Cohen-Macaulay scheme. See \cite{aus/buc:89} and \cite{has:00} for more examples. We formulate a relative Cohen-Macaulay approximation theory axiomatically in terms of categories \(\cat{D}\sbeq\cat{X}\sbeq\cat{A}\) fibred in abelian and additive subcategories over a base category \(\cat{C}\). In addition to the Auslander-Buchweitz axioms (AB1-AB4) for the fibre categories we formulate two axioms (BC1-BC2) regarding base change properties of the fibred categories. AB1-AB2 and BC1-BC2 imply the existence of an approximation and a hull which are preserved by any base change; see Theorem \ref{thm.cofapprox}. If AB3 holds too, we get functoriality and adjointness properties in suitable stable categories fibred in additive categories; see Theorem \ref{thm.cofmain}. In the case described above \(\cat{C}=\cat{CM}\), \(\cat{A}\) is the category \(\cat{mod}\) of pairs \((h:S\ra T,\mc{N})\) where \(\mc{N}\) is a finite \(T\)-module (no \(S\)-flatness) and \(\cat{X}=\cat{MCM}\). Another application of this theory is given in \cite{ile:11xb}.

In the second half of the article we proceed to study properties of continuous families of MCM approximations and FID hulls by homological methods.
As a consequence of the existence of minimal approximations and hulls of local flat families there are induced natural maps of deformation functors of pairs (algebra, module): 
\begin{equation}
\df{}{(A,N)}\lra\df{}{(A,X)}\quad\text{for}\quad X=M, M', L\,\,\text{and}\,\, L'
\end{equation}
There are corresponding maps \(\df{A}{N}\ra\df{A}{X}\) of deformation functors of the modules where \(A\) only deforms trivially. Rather weak conditions on \(N\), e.g.\ \(\grade N\geq 1\), respectively \(\grade N\geq 2\), imply the injectivity and formal smoothness of these maps for \(X=L'\). If, in addition, there is a versal family in \(\df{}{(A,N)}\) (or \(\df{A}{N}\)) then the maps are smooth for the appropriate category of henselian rings; see Theorem \ref{thm.defgrade} and Corollary \ref{cor.defgrade}. As a consequence each CM algebraic \(k\)-algebra \(A\) with \(A/\fr{m}_{A}\cong k\) and \(\dim A\geq 2\) has a finite \(A\)-module \(Q'\) of finite \emph{projective} dimension with a \emph{universal} deformation in \(\df{A}{Q'}(A)\); see Corollary \ref{cor.defapprox}. There are analogous general results for \(X=M\); see Theorem \ref{thm.defgrade2} and Corollary \ref{cor.defgrade2}, with applications in Corollary \ref{cor.depth} and \ref{cor.2dim}. E.g.\ if there is a closed subscheme \(Z\) in \(\Spec A\) containing the singular locus and with complement \(U\) such that \(\tilde{N}_{\vert U}=0\) and \(\depth_{Z}N\geq 2\) then \(\sigma_{M}:\df{}{(A,N)}\ra\df{}{(A,M)}\) is formally smooth.
Or if \(\Spec A\) is a \(2\)-dimensional normal Gorenstein singularity and \(N\) is torsion-free then the map \(\sigma_{M}\) is smooth. In this case both functors have versal elements by Theorem \ref{thm.ExVers}.

Consider a quotient ring \(B=A/I\) defined by a regular sequence \(I=(f_{1},\dots,f_{n})\) and an MCM \(B\)-module \(N\). Then \(N\) is also an \(A\)-module with an MCM approximation \(M\ra N\). If \(N\) has a \emph{lifting} to \(A/I^{2}\), then the composition of natural maps \(\df{B}{N}\ra \df{A}{N}\ra\df{A}{M}\) is injective; see Theorem \ref{thm.defMCM}. It turns out that the lifting condition is equivalent to the splitting of \(B\ot_{A}M\ra N\) (this generalises \cite[4.5]{aus/din/sol:93}).

The second part of the article also contains some general deformation theory of a pair \((h:S\ra T,\mc{N})\) of an algebra and a \(T\)-module. We define the graded algebra \(\vG:=T\oplus \mc{N}\) and consider the graded Andr{\'e}-Quillen cohomology \(\gH{0}^{*}(S,\vG, J)\) which govern the obstruction theory of the pair. In the case the graded \(\vG\)-module \(J\) is concentrated in degree \(0\) and \(1\) there is by Proposition \ref{prop.lang} a natural long-exact sequence which in the case \(J=\vG\) (with \(\gH{0}^{*}(S,\vG)=\gH{0}^{*}(S,\vG,\vG)\)) gives the suggestive
\begin{equation}
\begin{split}
0 &{} \ra \nd{}{T}{\mc{N}}\ra\gDer{0}_{S}(\vG)\ra\Der_{S}(T)\ra\xt{1}{T}{\mc{N}}{\mc{N}}\ra\gH{0}^{1}(S,\vG)\ra \\
 &{}\cH^{1}(S,T) \ra\xt{2}{T}{\mc{N}}{\mc{N}}\ra\gH{0}^{2}(S,\vG)\ra\cH^{2}(S,T)\ra\dots
\end{split}
\end{equation}
It relates the cohomology of the pair with the cohomology groups governing the obstruction theory of the algebra \(T\) and of the module \(\mc{N}\).
The sequence is used in the proof of the existence of a versal element in \(\df{}{(A,N)}\) where \(\Spec A\) is an isolated equidimensional singularity and \(N\) is locally free on the smooth locus; see Theorem \ref{thm.ExVers}. It is also used to define and study the Kodaira-Spencer class \(\kappa(\vG/S/\vL)\) in \(\gH{0}^{1}(S,\vG,\varOmega_{S/\vL}\ot\vG)\) (where \(\vL\ra S\) is a another ring homomorphism) which maps to the ungraded Kodaira-Spencer class \(\kappa(T/S/\vL)\). In the case the latter is zero we define a `secondary' Kodaira-Spencer class \(\kappa(\sigma,\mc{N})\) in \(\xt{1}{T}{\mc{N}}{\varOmega_{S/\vL}\ot \mc{N}}\) which depends on a choice of an \(S\)-algebra splitting \(\sigma\). This enables us to define `global' Kodaira-Spencer maps
\begin{equation}
g^{\vG}:\Der_{\vL}(S)\ra\gH{0}^{1}(S,\vG,\vG)\quad \text{and}\quad g^{(\sigma,\,\mc{N})}:\Der_{\vL}(S)\ra\xt{1}{T}{\mc{N}}{\mc{N}}\,.
\end{equation}
We also describe how classes and maps are related to the Atiyah class \(\at_{T/\vL}(\mc{N})\) in \(\xt{1}{T}{\mc{N}}{\varOmega_{T/\vL}\ot\mc{N}}\); see Proposition \ref{prop.at}. These results might have a certain independent interest. The arguments should be extendable to the setting of L.\ Illusie's \cite{ill:71}.

Injectivity of the corresponding \emph{local} Kodaira-Spencer maps gives a criterion for a global flat family to be non-trivial. The Kodaira-Spencer maps commute with Cohen-Macaulay approximation and this is applied to show that injectivity of the Kodaira-Spencer map is preserved by Cohen-Macaulay approximation under conditions as in Theorems \ref{thm.defgrade}, \ref{thm.defgrade2} and \ref{thm.defMCM}.

To make the text more reader friendly we have included some background material, e.g.\ on Cohen-Macaulay approximation and Kodaira-Spencer maps, and some of the central technical tools such as a general `cohomology and base change' result and some language of fibred categories. Many results have analogous parts with similar arguments and the policy has been to give a fairly detailed proof of one case and leave the other cases to the reader.
\section{Preliminaries}\label{sec.CMapprox}
All rings are commutative. If \(A\) is a ring, \(\cat{Mod}_{A}\) denotes the category of \(A\)-modules and \(\cat{mod}_{A}\) denotes the full subcategory of finite \(A\)-modules. If \(A\) is local then \(\fr{m}_{A}\) denotes the maximal ideal. Subcategories are usually full and essential.
\subsection{Axiomatic Cohen-Macaulay approximation}\label{subs.CCM}
We briefly recall some of the main features of Cohen-Macaulay approximation as introduced by Auslander and Buchweitz in \cite{aus/buc:89}. 
In this section let \(\cat{A}\) be an abelian category and \(\cat{D}\sbeq\cat{X}\sbeq\cat{A}\) additive subcategories.
Let \(\hat{\cat{X}}\) denote the subcategory of \(\cat{A}\) of objects \(N\) which have finite resolutions \(0\ra M_{n}\ra\dots\ra M_{0}\ra N\ra 0\) with the \(M_{i}\) in \(\cat{X}\). If \(n\) is the smallest such number, then \(\resdim{\cat{X}}{N} = n\). Let \(\injdim{\cat{X}}{N}\) be the minimal \(n\) (possibly \(\infty\)) such that \(\xt{i}{\cat{A}}{M}{N}= 0\) for all \(i>n\) and all \(M\) in \(\cat{X}\). Let \(\cat{X}^{\perp}\) denote the subcategory of objects \(L\) in \(\cat{A}\) with \(\injdim{\cat{X}}{L}=0\); the right complement of \(\cat{X}\). The left complement \({}^{\perp}\cat{X}\) is defined analogously. 

Let \(N\) be an object in \(\cat{A}\). An \emph{\(\cat{X}\)-approximation} and a \emph{\(\hat{\cat{D}}\)-hull}
of \(N\) are exact sequences as in \eqref{eq.MCMseq} with \(L\), \(L'\) in \(\hat{\cat{D}}\) and \(M\), \(M'\) in \(\cat{X}\). 

In general any \(f:M\ra N\) in \(\cat{A}\) is called a \emph{right \(\cat{X}\)-approximation of \(N\)} if \(M\) is in \(\cat{X}\) and any \(f':M'\ra N\) with \(M'\) in \(\cat{X}\)  factorises through \(f\). Dually, \(g:N\ra L\) is called a \emph{left \(\cat{X}\)-approximation of \(N\)} if \(L\) is in \(\cat{X}\) and any \(g':N\ra L'\) with \(L'\) in \(\cat{X}\) factorises through \(g\).

Consider the following conditions on the triple of categories \((\cat{A},\cat{X},\cat{D})\).
\begin{enumerate}
\item[(AB1)] \(\cat{X}\) is exact in \(\cat{A}\) (\(\cat{X}\) is closed under direct summands and extensions).
\item[(AB2)] \(\cat{D}\) is a \emph{cogenerator} for \(\cat{X}\), i.e.\ for each object \(M\) in \(\cat{X}\) there is an object
\(D\) in \(\cat{D}\) and a short exact sequence \(M\ra D\ra M'\) with \(M'\) in \(\cat{X}\).
\item[(AB3)] \(\cat{D}\) is \(\cat{X}\)-\emph{injective}, i.e.\ \(\cat{D}\sbeq\cat{X}^{\perp}\).
\item[(AB4)] \(\cat{A}\)-epimorphisms in \(\cat{X}\) are admissible (i.e.\ their kernels are contained in \(\cat{X}\)).
\end{enumerate}
If AB1 and AB2, there exist \(\cat{X}\)-approximations and \(\hat{\cat{D}}\)-hulls for all objects in \(\hat{\cat{X}}\) \cite[1.1]{aus/buc:89}. Assume AB1-AB3. Then any \(\cat{X}\)-approximation is a right \(\cat{X}\)-approximation and any \(\hat{\cat{D}}\)-hull is a left \(\hat{\cat{D}}\)-approximation. An \(\cat{X}\)-approximation determines a \(\hat{\cat{D}}\)-hull and vice versa through the following diagram of short exact sequences; the upper horizontal and right vertical being an \(\cat{X}\)-approximation and a \(\hat{\cat{D}}\)-hull of \(N\), \(D\) is in \(\cat{D}\). The boxed square is (co)cartesian (see \cite[1.4]{aus/buc:89}):
\begin{equation}\label{eq.pullpush}
\xymatrix@C-0pt@R-12pt@H-30pt{
L \ar[r]\ar@{=}[d] & M \ar[r]\ar[d]\ar@{}[dr]|{\Box} & N \ar[d]  \\   
L \ar[r] & D \ar[r]\ar[d] & L' \ar[d]  \\
& M' \ar@{=}[r] & M'
}
\end{equation}
Moreover, the category \(\cat{D}\) is determined by \(\cat{X}\subset\cat{A}\). Indeed \(\cat{D}=\cat{X}\cap\cat{X}^{\perp}\). By \cite[3.9]{aus/buc:89} monomorphisms in \(\hat{\cat{D}}\) are admissible and \(\hat{\cat{D}}=\hat{\cat{X}}\cap\cat{X}^{\perp}\). Also \(\cat{X}={}^{\perp}\hat{\cat{D}}\cap\hat{\cat{X}}={}^{\perp}\cat{D}\cap\hat{\cat{X}}\). If \(\cat{X}/\cat{D}\) denotes the quotient category, the \(\cat{X}\)-approximation induces a right adjoint to the inclusion functor \(\cat{X}/\cat{D}\sbeq \hat{\cat{X}}/\cat{D}\) and the \(\hat{\cat{D}}\)-hull induces a left adjoint to the inclusion functor \(\hat{\cat{D}}/\cat{D}\sbeq\hat{\cat{X}}/\cat{D}\); see \cite[2.8]{aus/buc:89}.

A morphism \(f: M\ra N\) in \(\cat{A}\) is called \emph{right minimal} if for any \(g:M\ra M\) with \(fg=f\) it follows that \(g\) is an automorphism. Dually, \(f\) is called \emph{left minimal} if for any \(h:N\ra N\) with \(hf=f\) it follows that \(h\) is an automorphism. Note that if \(f:M\ra N\) and \(f: M'\ra N\) both are right minimal then there exists an isomorphism \(g:M\ra M'\) with \(f=f'g\), and similarly for left minimal morphisms.

We will simply call an \(\cat{X}\)-approximation (a \(\hat{\cat{D}}\)-hull) for minimal if it is right (left) minimal. 
\begin{ex}\label{ex.MCMapprox}
Suppose \(A\) is a Cohen-Macaulay ring which possesses a canonical module \(\omega_{A}\) in the sense that any localisation in a maximal ideal gives a maximal Cohen-Macaulay module of finite injective dimension and Cohen-Macaulay type \(1\); cf.\ \cite[3.3.16]{bru/her:98}. Let \(\cat{MCM}_{A}\) denote the category of maximal Cohen-Macaulay (MCM) \(A\)-modules and put \(\cat{D}_{A}:=\Add\{\omega_{A}\}\). Then the triple \((\cat{A},\cat{X},\cat{D})=(\cat{mod}_{A},\cat{MCM}_{A},\cat{D}_{A})\) satisfies properties AB1-AB4; cf.\ \cite[I 4.10.11]{has:00} and \(\hat{\cat{X}}=\cat{mod}_{A}\). If \(A\) in addition is a local ring, then the \(\cat{MCM}_{A}\)-approximation and the \(\hat{\cat{D}}_{A}\)-hull can be chosen to be minimal; cf.\ \cite[Section 3]{sim/str:02} or Corollary \ref{cor.minapprox}. 

Let \(\cat{FID}_{A}^{l}\) denote the subcategory of finite \(A\)-modules \(E\) which have locally finite injective dimension, i.e.\ \(\Injdim_{A_{\fr{p}}} E_{\fr{p}}<\infty\) for all \(\fr{p}\in\Spec A\). The approximation result implies that \(\cat{FID}_{A}^{l}=\hat{\cat{D}}_{A}\): Let \(L\) be in \(\hat{\cat{D}}_{A}\). By induction on \(\resdim{\cat{D}_{A}}{L}\) \(L\) is in \(\cat{FID}_{A}^{l}\). Conversely let \(E\) be in \(\cat{FID}_{A}^{l}\). If \(L\ra M\ra E\) is an \(\cat{MCM}_{A}\)-approximation of \(E\) then \(M\) also has locally finite injective dimension. Let \(M^{\vee}\) denote \(\hm{}{A}{M}{\omega_{A}}\) and choose a surjection \(A^{{\oplus}n}\ra M^{\vee}\). Both \(M^{\vee}\) and the kernel \(M_{1}\) are MCM. Applying \(\hm{}{A}{-}{\omega_{A}}\) gives (by duality theory) the short exact sequence \(M\xra{i} \omega_{A}^{{\oplus}n}\ra M_{1}^{\vee}\). But \(i\) splits since \(\xt{1}{A}{M_{1}^{\vee}}{M}=0\) by \cite[3.3.3]{bru/her:98} and so \(M\) is in \(\cat{D}_{A}\) and \(E\) is in \(\hat{\cat{D}}_{A}\). 
\end{ex}
\subsection{The representing complex}\label{subsec.cplx}
Consider an abelian category \(\cat{A}\) and additive subcategories \(\cat{D}\sbeq\cat{X}\sbeq \cat{A}\). A \emph{\(\cat{DX}\)-resolution} of an object \(N\) in \(\cat{A}\) is a finite resolution \({}^{-}C^{*}\twoheadrightarrow N\) with \({}^{-}C^{i}\in\cat{D}\) for \(i<0\) and \({}^{-}C^{0}\in \cat{X}\). If \(L:=\coker (d^{-2}:{}^{-}C^{-2}\ra {}^{-}C^{-1})\), then the short exact sequence \(L\ra {}^{-}C^{0}\ra N\) is an \(\cat{X}\)-approximation. A \emph{\(\hat{\cat{D}}\cat{D}\)-coresolution} of \(N\) is a coresolution \(N\rightarrowtail {}^{+}C^{*}(N)\) such that \({}^{+}C^{0}\in\hat{\cat{D}}\), \({}^{+}C^{i}\in\cat{D}\) and \(\ker d^{i}\in \cat{X}\) for \(i>0\). If \(M':=\ker d^{1}\) then the short exact sequence \(N\ra {}^{+}C^{0}\ra M'\) is a \(\hat{\cat{D}}\)-hull. 
Given AB1 and AB2,  each \(N\) in \(\hat{\cat{X}}\) has a \(\cat{DX}\)-resolution and a \(\hat{\cat{D}}\cat{D}\)-coresolution.
Finally, a bounded below \(\cat{D}\)-complex \(D^{*}(N):\dots\ra D^{-1}\ra D^{0}\ra D^{1}\ra\dots\) with \(\ker d^{i}\in\cat{X}\) for all \(i\geq 0\) and its only non-trivial cohomology in degree zero with \(\cH^{0}(D^{*})\cong N\) is called a \emph{\(\cat{D}\)-complex representing \(N\)}. A representing complex splits into (and is (re)constructed from) a \(\cat{DX}\)-resolution given by \(\dots\ra D^{-1}\ra \ker d^{0}\thr \cH^{0}(D^{*})=N\) and a \(\hat{\cat{D}}\cat{D}\)-coresolution \(N\rightarrowtail \coker d^{-1}\ra D^{1}\ra\dots\) where \(N\rightarrowtail \coker d^{-1}\) is induced by \(\ker d^{0}\rightarrowtail D^{0}\).
\begin{lem}\label{lem.res}
Assume \(\xt{1}{\cat{A}}{\cat{X}}{\hat{\cat{D}}}=0\).
Suppose \(f:N_{1}\ra N_{2}\) is in \(\hat{\cat{X}}\)\textup{.} Assume \(F^{*}(N_{i})\) exists for \(i=1,2\) where \(F^{*}(N_{i})\) denotes one of the complexes \({}^{-}C^{*}(N_{i})\), \({}^{+}C^{*}(N_{i})\) or \(D^{*}(N_{i})\)\textup{.} Then \(f\) can be extended to an arrow of chain complexes \(f^{*}:F^{*}(N_{1})\ra F^{*}(N_{2})\) which is uniquely defined up to homotopy\textup{.}

Assume \textup{AB1-AB3} for the triple of categories \((\cat{A},\cat{X},\cat{D})\)\textup{.} Then \(N\mapsto {}^{-}C^{*}(N)\)\textup{,} \(N\mapsto {}^{+}C^{*}(N)\) and \(N\mapsto D^{*}(N)\) induce functors to the homotopy categories of chain complexes as follows\textup{:}
\begin{equation*}
{}^{-}C^{*}:\hat{\cat{X}}\ra \cat{K}^{\textnormal{b}}(\cat{X})\quad\quad
{}^{+}C^{*}:\hat{\cat{X}}\ra \cat{K}^{+}(\hat{\cat{D}})\quad\quad
D^{*}:\hat{\cat{X}}/\cat{D}\ra \cat{K}^{+}(\cat{D})
\end{equation*}
\end{lem}
\begin{proof}
The proof for \({}^{-}C^{*}(N)\) and \({}^{+}C^{*}(N)\) follows standard lines for constructing chain maps and homotopies. The assumption \(\xt{1}{\cat{A}}{\cat{X}}{\hat{\cat{D}}}=0\) is used every time a lifting or extension of an arrow is required.

Let \((D^{*}_{i}, d_{i}^{*})=D^{*}(N_{i})\) and let \(M_{i}=\ker d_{i}^{0}\) and \(L_{i}=\im d_{i}^{-1}\). Then there are short exact sequences \(L_{i}\ra M_{i}\ra N_{i}\) which by assumption are \(\cat{X}\)-approximations. Since \(\xt{1}{\cat{A}}{M_{1}}{L_{2}}=0\), the arrow \(N_{1}\ra N_{2}\) extends to the \(\cat{X}\)-approximation and further on to the negative part of the complexes. If \(M'_{i}=\ker d^{1}_{i}\) then the \(M'_{i}\) are in \(\cat{X}\) by assumption and there are short exact sequences \(M_{i}\ra D^{0}_{i}\ra M'_{i}\). There is an extension of \(M_{1}\ra D^{0}_{2}\) to \(D^{0}_{1}\ra D^{0}_{2}\) and an induced arrow \(M'_{1}\ra M'_{2}\) which again extends and so on to a chain map \(f^{*}:D^{*}_{1}\ra D^{*}_{2}\). 

Let \(g^{*}:D^{*}_{1}\ra D^{*}_{2}\) be a chain map, put \(g=\cH^{0}(g^{*})\), \(s=f{-}g\) and \(s^{*}=f^{*}{-}g^{*}\). Suppose \(s\) factors through \(D\) in \(\cat{D}\); \(s=ab\) with \(a:D\ra N_{2}\). Since \(\xt{1}{\cat{A}}{D}{L_{2}}=0\) there exists a lifting \(\tilde{a}:D\ra M_{2}\) of \(a\). Put \(h_{N}=\tilde{a}b\) and continue similarly to construct a homotopy \(h\) for the extended negative part:
\begin{equation*}
\xymatrix@C+6pt@R-0pt@H+6pt{
& \dots \ar[r] & D_{1}^{-1} \ar[dl]_{h^{-1}}\ar[r]\ar[d]|-{s^{-1}} & M_{1} \ar[dl]|-{h_{M}}\ar[r]\ar[d]|-{\mr{Z}^{0}\!(s^{*})} & N_{1} \ar[dl]|-{h_{N}}\ar[d]^{s}\ar[r] & 0 \\   
\dots \ar[r] & D_{2}^{-2}\ar[r] & D_{2}^{-1} \ar[r] & M_{2} \ar[r] & N_{2} \ar[r] & 0 
}
\end{equation*}
In particular \(h_{M}:M_{1}\ra D^{-1}_{2}\) can be extended to an \(h^{0}:D^{0}_{1}\ra D^{-1}_{2}\) with \(s^{-1}=h^{0}d^{-1}_{1}{+}d^{-2}_{2}h^{-1}\). The construction of the \(h^{i}\) for \(i>0\) is standard.
\end{proof}
\begin{lem}\label{lem.xres}
Assume \textup{AB1-AB3} for the triple of categories \((\cat{A},\cat{X},\cat{D})\)\textup{.} Given an exact sequence \(\vare:0\ra N_{1}\ra N_{2}\ra N_{3}\ra 0\) with objects in \(\hat{\cat{X}}\)\textup{.} Then there are exact sequences of complexes where \(\vare\) equals the cohomology\textup{:}
\begin{enumerate}
\item[(i)] \(0\ra {}^{-}C^{*}(N_{1})\lra {}^{-}C^{*}(N_{2})\lra {}^{-}C^{*}(N_{3})\ra 0\) 
\item[(ii)] \(0\ra {}^{+}C^{*}(N_{1})\lra {}^{+}C^{*}(N_{2})\lra {}^{+}C^{*}(N_{3})\ra 0\)
\item[(iii)] \(0\ra D^{*}(N_{1})\lra D^{*}(N_{2})\lra D^{*}(N_{3})\ra 0\) \textup{(}termwise split exact\textup{)}
\end{enumerate}
\end{lem}
\begin{proof}
Choose \(\cat{X}\)-approximations \(L_{i}\ra M_{i}\ra N_{i}\) for \(i=1,3\). There is a \(3{\times}3\) commutative diagram of \(6\) short exact sequences which extends the ``horseshoe'' diagram; cf.\ \cite[1.12.11]{has:00}. One obtains an \(\cat{X}\)-approximation of \(N_{2}\) and short exact sequences \(m:M_{1}\ra M_{2}\ra M_{3}\) and \(L_{1}\ra L_{2}\ra L_{3}\) in \(\cat{X}\) and \(\hat{\cat{D}}\) respectively since both categories are closed by extensions (by AB1 and \cite[3.8]{aus/buc:89}). If \(D^{*}_{i}[1]\xra{\eta_{i}} L_{i}\) are finite \(\cat{D}\)-resolutions then since \(\xt{1}{\cat{A}}{D^{-1}_{3}}{L_{1}}=0\) there is a lifting \(\tilde{\eta}_{3}:D^{-1}_{3}\ra L_{2}\) of \(\eta_{3}\) which combined with \(\eta_{1}\) gives \(\eta_{2}:D^{-1}_{1}\amalg D^{-1}_{3}\ra L_{2}\). The kernels of the resulting arrows between short exact sequences give a short exact sequence of objects in \(\hat{\cat{D}}(S)\). The argument is repeated. Splicing with \(m\) in degree zero the short exact sequence of \({}^{-}C^{*}\)-resolutions in (i) is obtained. 

Choose short exact sequences \(M_{i}\ra D^{0}_{i}\ra M_{i}'\) for \(i=1,3\) as in AB2. Since \(\xt{1}{\cat{A}}{M_{3}}{D^{0}_{1}}=0\) there is an extension to an arrow of short exact sequences from \(m\) to \(D^{0}_{1}\ra D^{0}_{2}\ra D^{0}_{3}\) with \(D^{0}_{2}=D^{0}_{1}\amalg D^{0}_{3}\) and \(M_{2}':=\coker(M_{2}\ra D^{0}_{2})\in \cat{X}\) by AB1. Repeated application of this argument gives a short exact sequence of \(\cat{D}\)-coresolutions and splicing with the sequences in (i) gives (iii). Pushout of \(M_{i}\ra D^{0}_{i}\ra M_{i}'\) along \(M_{i}\ra N_{i}\) gives a short exact sequence of \(\hat{\cat{D}}\)-hulls and splicing with \(D^{1}_{i}\ra D^{2}_{i}\ra\dots\) gives (ii).
\end{proof}
\subsection{Base change}
The main tool for reducing properties to the fibres in a flat family will be the base change theorem.
We follow the quite elementary and general approach of A.\ Ogus and G.\ Bergman \cite{ogu/ber:72}.
\begin{defn}
Let \(h:S\ra T\) be a ring homomorphism and \(I\) an \(S\)-module.
Let \(F\) be an \(S\)-linear functor of some additive subcategory of \(\cat{Mod}_{S}\) to \(\cat{Mod}_{T}\). Then the \emph{exchange map \(e_{I}\) for \(F\)} is defined as the \(T\)-linear map \(e_{I}:F(S)\ot_{S}I\ra F(I)\) given by \(\xi\ot u\mapsto F(u)(\xi)\) where we consider \(u\) as the multiplication map \(u:S\ra I\). Let \(\mSpec T\) denote the set of closed points in \(\Spec T\).
\end{defn}
\begin{prop}\label{prop.nakayama}
Let \(h:S\ra T\) be a ring homomorphism with \(S\) noetherian\textup{.} Suppose \(\{F^{q}:\cat{mod}_{S}\ra \cat{mod}_{T}\}_{q\geq 0}\) is an \(h\)-linear cohomological \(\delta\)-functor\textup{.} 
\begin{enumerate}
\item[(i)] If the exchange map \(e_{S/\fr{n}}^{q}:F^{q}(S)\ot_{S}S/\fr{n}\ra F^{q}(S/\fr{n})\) is surjective for all \(\fr{n}\) in \(Z=\im\{\mSpec T\ra\Spec S\}\)\textup{,} then \(e_{I}^{q}:F^{q}(S)\ot_{S}I\ra F^{q}(I)\) is an isomorphism for all \(I\) in \(\cat{mod}_{S}\)\textup{.} 
\item[(ii)] If \(e_{S/\fr{n}}^{q}\) is surjective for all \(\fr{n}\) in \(Z\)\textup{,} then \(e_{I}^{q-1}\) is an isomorphism for all \(I\) in \(\cat{mod}_{S}\) if and only if \(F^{q}(S)\) is \(S\)-flat\textup{.}
\end{enumerate}
\end{prop}
Note that if the \(F^{q}\) in addition extend to functors of all \(S\)-modules \(F^{q}:\cat{Mod}_{S}\ra\cat{Mod}_{T}\) which commute with direct limits, then the conclusions are valied for all \(I\) in \(\cat{Mod}_{S}\).
\begin{ex}\label{ex.nakcplx}
Suppose \(S\) and \(T\) are noetherian. Let \(K^{*}:\, K^{0}\ra K^{1}\ra\dots\) be a complex of \(S\)-flat and finite \(T\)-modules. Define \(F^{q}:\cat{mod}_{S}\ra\cat{mod}_{T}\) by \(F^{q}(I)=\cH^{q}(K^{*}\ot_{S}I)\). Then \(\{F^{q}\}_{q\geq 0}\) is an \(h\)-linear cohomological \(\delta\)-functor which extends to all \(S\)-modules and commutes with direct limits.
\end{ex}
\begin{ex}\label{ex.nakayama}
Suppose \(S\) and \(T\) are noetherian. Let \(M\) and \(N\) be finite \(T\)-modules with \(N\) \(S\)-flat. Then the functors \(F^{q}:\cat{mod}_{S}\ra\cat{mod}_{T}\) defined by \(F^{q}(I)=\xt{q}{T}{M}{N\ot_{S}I}\) for \(q\geq 0\) give an \(h\)-linear cohomological \(\delta\)-functor which extends to all \(S\)-modules and commutes with direct limits.
\end{ex}
Let \(S\ra T\) and \(S\ra S'\) be ring homomorphisms, \(M\) a \(T\)-module, \(T'=T\ot_{S}S'\) and \(N'\) a \(T'\)-module. Then there is a change of rings spectral sequence 
\begin{equation}\label{eq.ss}
\cE_{2}^{p,q}=\xt{q}{T'}{\tor{S}{p}{M}{S'}}{N'}\,\Ra\,\xt{p+q}{T}{M}{N'}
\end{equation} 
which, in addition to the isomorphism \(\hm{}{T'}{M\ot_{S}S'}{N'}\cong\hm{}{T}{M}{N'}\), gives edge maps \(\xt{q}{T'}{M\ot_{S}S'}{N'}\ra\xt{q}{T}{M}{N'}\) for \(q>0\) which are isomorphisms too if \(M\) (or \(S'\)) is \(S\)-flat. If \(I'\) is an \(S'\)-module we can compose the exchange map \(e^{q}_{I'}\) (regarding \(I'\) as \(S\)-module) with the inverse of this edge map for \(N'=N\ot_{S}I'\) and obtain a map \(c^{q}_{I'}\) of \(T'\)-modules 
\begin{equation}\label{eq.basechange}
c^{q}_{I'}:\xt{q}{T}{M}{N}\ot_{S}I'\ra\xt{q}{T'}{M\ot_{S}S'}{N\ot_{S}I'}\,.
\end{equation}
\begin{rem}
This is the base change map (in the affine case) considered by A.\ Altman and S.\ Kleiman, their conditions are slightly different; see \cite[1.9]{alt/kle:80}.
\end{rem}
We will use the following geometric notation. Suppose \(h:S\ra T\) is a ring homomorphism, \(M\) is a \(T\)-module and \(s\) is a point in \(\Spec S\) with residue field \(k(s)\). Then \(M_{s}\) denotes the fibre \(M\ot_{S}k(s)\) of \(M\) at \(s\) with its natural \(T_{s}=T\ot_{S}k(s)\)-module structure. Now Proposition \ref{prop.nakayama} implies the following:
\begin{cor}\label{cor.xtdef}
Suppose \(S\ra T\) and \(S\ra S'\) are homomorphisms of noetherian rings\textup{,} \(M\) and \(N\) are finite \(T\)-modules\textup{,} \(Z=\im\{\mSpec T\ra\Spec S\}\) and \(q\) is an integer\textup{.} Assume that \(M\) and \(N\) are \(S\)-flat\textup{.}
\begin{enumerate}
\item[(i)] If\, \(\xt{q+1}{T_{s}}{M_{s}}{N_{s}}=0\) for all \(s\) in \(Z\)\textup{,} then \(c^{q}_{I'}\) in \eqref{eq.basechange} is an isomorphism for all \(S'\)-modules \(I'\)\textup{.}
\item[(ii)] If in addition \(\xt{q-1}{T_{s}}{M_{s}}{N_{s}}=0\) for all \(s\in Z\)\textup{,} then \(\xt{q}{T}{M}{N}\) is \(S\)-flat\textup{.}
\end{enumerate}
\end{cor}
\section{Categories fibred in additive categories}\label{sec.cof}
We will phrase our results in the language of fibred categories\footnote{We have chosen to work with rings instead of (affine) schemes. Our definition of a fibred category \(p:\cat{F}\ra\cat{C}\) reflects this choice and is equivalent to the functor of opposite categories \(p^{\text{op}}:\cat{F}^{\text{op}}\ra\cat{C}^{\text{op}}\) being a fibred category as defined in \cite{FAG}.}. We therefore briefly recall some of the basic notions, taken mainly from A.\ Vistoli's article in \cite{FAG}. Then we define quotients of categories fibred in additive categories.

Consider a category \(\cat{C}\). Given a category over \(\cat{C}\), i.e.\ a functor \(p:\cat{F}\ra\cat{C}\). 
To an object \(T\) in \(\cat{C}\), let \(\cat{F}(T)\); the \emph{fiber of \(\cat{F}\) over \(T\)}, denote the subcategory of arrows \(\phi\) in \(\cat{F}\) such that \(p(\phi)=\id_{T}\).
An arrow \(\phi_{1}:\xi\ra\xi_{1}\) in \(\cat{F}\) is \emph{cocartesian} if for any arrow \(\phi_{2}:\xi\ra\xi_{2}\) in \(\cat{F}\) and any arrow \(f_{21}:p(\xi_{1})\ra p(\xi_{2})\) in \(\cat{C}\) with \(f_{21} p(\phi_{1})=p(\phi_{2})\) there exists a unique arrow \(\phi_{21}:\xi_{1}\ra\xi_{2}\) with \(p(\phi_{21})=f_{21}\) and \(\phi_{21}\phi_{1}=\phi_{2}\). If for any arrow \(f:T\ra T'\) in \(\cat{C}\) and any object \(\xi\) in \(\cat{F}\) with \(p(\xi)=T\) there exists a cocartesian arrow \(\phi:\xi\ra \xi'\) for some \(\xi'\) with \(p(\phi)=f\), then \(\cat{F}\) (or rather \(p:\cat{F}\ra\cat{C}\)) is a \emph{fibred category}. Moreover, \(\xi'\) will be called a \emph{base change} of \(\xi\) by \(f\). If \(\xi''\) is another base change of \(\xi\) by \(f\) then \(\xi'\) and \(\xi''\) are isomorphic over \(T'\) by a unique isomorphism. We shall also say that a property \(P\) of objects in the fibres of \(\cat{F}\) is \emph{preserved by base change} if \(P(\xi)\) implies \(P(\xi')\) for any base change \(\xi'\) of \(\xi\).
A morphism of fibred categories is a functor \(F:\cat{F}_{1}\ra\cat{F}_{2}\) with \(p_{2}F=p_{1}\) such that \(\phi\) cocartesian implies \(F(\phi)\) cocartesian. If \(F\) in addition is an inclusion of categories, \(\cat{F}_{1}\) is a fibred subcategory of \(\cat{F}_{2}\).
A category with all arrows being isomorphisms is a groupoid. A fibred category \(\cat{F}\) over \(\cat{C}\) is called a \emph{category fibred in groupoids} (often abbreviated to groupoid) if all fibres \(\cat{F}(T)\) are groupoids. Then all arrows in \(\cat{F}\) are cocartesian. If all fibres \(\cat{F}(T)\) only contain identities, then \(\cat{F}\) is called a \emph{category fibred in sets}.
\begin{lem}\label{lem.gpoid}
Given functors \(F:\cat{F}\ra\cat{G}\) and \(q:\cat{G}\ra\cat{C}\) and suppose \(q\) is fibred in sets\textup{.} Then \(F\) is fibred \textup{(}in groupoids\textup{/}sets\textup{)} if and only if \(qF\) is fibred \textup{(}in groupoids\textup{/}sets\textup{).}
\end{lem}
If \(T\) is an object in a category \(\cat{C}\) let \(\cat{C}/T\) denote the comma category of arrows to \(T\). Then the forgetful functor \(\cat{C}/T\ra\cat{C}\) is fibred in sets. If \(p:\cat{F}\ra\cat{C}\) is fibred (in groupoids/sets), \(\xi\) is an object in \(\cat{F}\) and \(T=p(\xi)\), then there is a natural functor \(p_{\xi}:\cat{F}/\xi\ra\cat{C}/T\). The composition \(\cat{F}/\xi\ra\cat{F}\ra\cat{C}\) is clearly fibred (in groupoids/sets) and hence \(\cat{F}/\xi\ra \cat{C}/T\) is fibred (in groupoids/sets) by Lemma \ref{lem.gpoid}.
If \(p:\cat{F}\ra\cat{C}\) is a functor and \(\cat{C}'\) is a subcategory of \(\cat{C}\) we can define the \emph{restriction} \(p':\cat{F}_{\vert \cat{C}'}\ra\cat{C}'\) of \(\cat{F}\) to \(\cat{C}'\) by picking for \(\cat{F}_{\vert \cat{C}'}\) the objects and morphisms in \(\cat{F}\) that \(p\) takes into \(\cat{C}'\). It follows that \(\cat{F}_{\vert \cat{C}'}\) is fibred (in groupoids/sets) if \(\cat{F}\) is.

The composition of two cocartesian arrows is cocartesian and isomorphisms are cocartesian. Hence the subcategory \(\cat{F}_{\text{coca}}\) of cocartesian arrows in a fibred category \(\cat{F}\) over \(\cat{C}\) is fibred in groupoids. If \(\cat{F}\) is fibred in groupoids there is an associated category fibred in sets \(\bar{\cat{F}}\ra\cat{C}\) defined by identifying all isomorphic objects in all fibres \(\cat{F}(T)\) and identifying arrows accordingly. If \(\cat{F}\) is fibred in sets one defines a functor \(F:\cat{C}\ra\Sets\) by \(F(T):=\cat{F}(T)\) and \(F(f):F(T)\ra F(T')\) is defined by \(F(f)(\xi):=\eta_{\xi,f}\) where \(\phi_{\xi,f}:\xi\ra \eta_{\xi,f}\) is the (in this case) unique cocartesian lifting of \(f\). From a functor \(G:\cat{C}\ra\Sets\) one defines a category fibred in sets, and these two operations are inverse up to natural equivalences. 
\begin{defn}\label{defn.addkof}
An \emph{additive \textup{(}abelian\textup{)} category \(\cat{F}\) over \(\cat{C}\)} is a functor \(p:\cat{F}\ra\cat{C}\) such that:
\begin{enumerate}
\item[(i)] The fibre \(\cat{F}(T)\) is an additive (abelian) category for all objects \(T\) in \(\cat{C}\).
\item[(ii)] For all objects \(\xi_{1}\) and \(\xi_{2}\) in \(\cat{F}\) and arrows \(f:p(\xi_{1})\ra p(\xi_{2})\) in \(\cat{C}\), 
\begin{equation*}
\hm{}{f}{\xi_{1}}{\xi_{2}}:=\{\phi\in\hm{}{\cat{F}}{\xi_{1}}{\xi_{2}}\,\vert\,p(\phi)=f\}
\end{equation*}
is an abelian group, and composition of arrows 
\begin{equation*}
\hm{}{f_{2}}{\xi_{2}}{\xi_{3}}\times\hm{}{f_{1}}{\xi_{1}}{\xi_{2}}\ra\hm{}{f_{2}f_{1}}{\xi_{1}}{\xi_{3}} 
\end{equation*}
is bilinear.
\end{enumerate}
A morphism \(F:\cat{F}_{1}\ra\cat{F}_{2}\) of additive (abelian) categories over \(\cat{C}\) is a \emph{linear} functor \(F\) over \(\cat{C}\), i.e.\ which gives linear maps of \(\Hom\)-groups. If in addition \(F\) is an inclusion of categories then \(\cat{F}_{1}\) is an additive (abelian) subcategory of \(\cat{F}_{2}\) over \(\cat{C}\). A category \(\cat{F}\) over \(\cat{C}\) is \emph{fibred in additive \textup{(}abelian\textup{)} categories}, abbreviated by FAd (FAb), if \(\cat{F}\) is both fibred and additive (abelian) over \(\cat{C}\). Morphisms should be linear and preserve cocartesian arrows. A FAd subcategory is a morphism of FAds which is an inclusion of categories. For \(i=1,2\) let \(\cat{A}_{i}\) be a FAb over \(\cat{C}\) and \(\cat{X}_{i}\sbeq\cat{A}_{i}\) a FAd subcategory such that the fibre categories \(\cat{X}_{i}(T)\) are exact. Then a morphism of FAds \(F:\cat{X}_{1}\ra\cat{X}_{2}\) is \emph{exact} if \(F\) preserves short exact sequences for all the fibre categories.
\end{defn}
Note that in a FAd finite (co)products in the fibres are preserved by base change.
Given a FAd subcategory \(\cat{D}\sbeq\cat{F}\). Two arrows \(\phi_{1}\) and \(\phi_{2}\) in \(\cat{F}\) are \(\cat{D}\)-equivalent if \(p(\phi_{1})=p(\phi_{2})\) and \(\phi_{1}-\phi_{2}\) factors through an object in \(\cat{D}\). Write \(\phi_{1}\sim\phi_{2}\).
Define the quotient category \(\cat{F}/\cat{D}\) over \(\cat{C}\) to have the same objects as \(\cat{F}\) and \(\hm{}{\cat{F}/\cat{D}}{\xi_{1}}{\xi_{2}}:=\hm{}{\cat{F}}{\xi_{1}}{\xi_{2}}/\sim\). The natural map to \(\cat{C}\) makes \(\cat{F}/\cat{D}\) an additive category over \(\cat{C}\) and the natural functor \(\cat{F}\ra\cat{F}/\cat{D}\) is linear over \(\cat{C}\).
\begin{lem}\label{lem.keystable}
If \(\phi_{1}:\xi\ra\xi_{1}\) is cocartesian in \(\cat{F}\) and \(\phi:\xi_{1}\ra\xi_{2}\) is any arrow such that \(\phi\phi_{1}\sim 0\) then \(\phi\sim 0\)\textup{.}
\end{lem}
\begin{proof}
Suppose \(\phi\phi_{1}=\beta\alpha\) with \(\alpha:\xi\ra\delta\) and with \(\delta\) in \(\cat{D}\). If \(p(\beta):T'\ra T_{2}\) then since \(\cat{D}\) is a fibred subcategory there exists an arrow \(\delta\ra\delta_{2}\) which is cocartesian in \(\cat{F}\) and with \(p(\delta_{2})=p(\xi_{2})=T_{2}\). Replacing \(\delta\) with \(\delta_{2}\) we assume \(p(\delta)=T_{2}\). Since \(\phi_{1}\) is cocartesian there exists a unique arrow \(\tau:\xi_{1}\ra\delta\) with \(\tau\phi_{1}=\alpha\). Since \(\phi_{1}\) is cocartesian uniqueness implies that \(\beta\tau=\phi\).
\end{proof}
\begin{lem}\label{lem.stable}
Given a FAd subcategory \(\cat{D}\sbeq\cat{F}\) over \(\cat{C}\)\textup{,} then the quotient category \(\cat{F}/\cat{D}\) is FAd over \(\cat{C}\) and the quotient morphism \(\cat{F}\ra\cat{F}/\cat{D}\) is a morphism of FAds\textup{.}
\end{lem}
\begin{proof}
We first show that if \(\phi_{1}:\xi\ra\xi_{1}\) is cocartesian in \(\cat{F}\) then its image \([\phi_{1}]\) in \(\cat{F}/\cat{D}\) is cocartesian. Given \(\phi_{2}:\xi\ra\xi_{2}\) and \(\theta:\xi_{1}\ra\xi_{2}\) with \(\theta\phi_{1}=\phi_{2}\). Suppose \(\theta':\xi_{1}\ra\xi_{2}\) with \(p(\theta')=p(\theta)\) satisfies \(\theta'\phi_{1}\sim\phi_{2}\). If \(\phi=\theta'-\theta\) then \(\phi\phi_{1}\sim 0\) so by Lemma \ref{lem.keystable} \(\phi\sim 0\). Now we show that \([\theta]\) is independent of the representations of the other maps. Let \(\phi_{i}':\xi\ra\xi_{i}\) with \(\phi_{i}'\sim\phi_{i}\) and suppose (as we may) that \(\theta'\) satisfies \(\theta'\phi_{1}'=\phi_{2}'\) with \(p(\theta')=p(\theta)\). Again let \(\phi=\theta'-\theta\). Then \(0\sim\phi'_{2}-\phi_{2}=\theta'\phi'_{1}-\theta\phi_{1}=\theta'(\phi_{1}'-\phi_{1})+\phi\phi_{1}\sim\phi\phi_{1}\). By Lemma \ref{lem.keystable} \(\phi\sim 0\). Given \(f:T\ra T_{1}\) and \(\xi\) in \(\cat{F}/\cat{D}\) with \(p(\xi)=T\) there exists a cocartesian \(\phi_{1}:\xi\ra \xi_{1}\) in \(\cat{F}\) with \(p(\phi_{1})=f\) and by what we have done \([\phi_{1}]\) is cocartesian in \(\cat{F}/\cat{D}\).
\end{proof}
Note that there are in general more cocartesian arrows in \(\cat{F}/\cat{D}\) than those in the image of cocartesian arrows in \(\cat{F}\). The following lemma characterises the cocartesian arrows in the quotient category:
\begin{lem}\label{lem.stablecc}
If \(\rho\) and \(\theta\) are composable arrows in \(\cat{F}\)  with \(\rho\) cocartesian and \(\theta\) inducing an isomorphism in \(\cat{F}/\cat{D}\)\textup{,} then \([\theta\rho]\) is cocartesian in \(\cat{F}/\cat{D}\)\textup{.}
Conversely\textup{,} suppose \([\phi]:\xi_{1}\ra\xi_{2}\) is cocartesian in \(\cat{F}/\cat{D}\) over \(f:T_{1}\ra T_{2}\)\textup{.} Then for any base change
\(\rho:\xi_{1}\ra \xi_{1}^{\#}\) of \(\xi_{1}\) over \(f\) in \(\cat{F}\)\textup{,} the induced arrow \(\phi^{\#}:\xi_{1}^{\#}\ra\xi_{2}\) gives an isomorphism in \(\cat{F}/\cat{D}(T_{2})\)\textup{.} 
\end{lem}
\begin{proof}
If \(\rho:\xi_{1}\ra\xi_{2}\) is cocartesian and \([\theta]:\xi_{2}\ra \xi_{3}\) is an isomorphism, let \(\phi=\theta\rho\). If \(\tau:\xi_{1}\ra \xi_{4}\), \(p(\xi_{i})=T_{i}\) and there is a map \(f:T_{3}\ra T_{4}\) with \(p(\tau)=f p(\theta) p(\rho)\), then there is a unique arrow \(\mu:\xi_{2}\ra\xi_{4}\) above \(f p(\theta)\) with \(\mu\rho=\tau\). This gives the arrow \([\mu][\theta]^{-1}:\xi_{3}\ra\xi_{4}\). If \(\mu_{i}:\xi_{3}\ra\xi_{4}\) for \(i=1,2\) are two arrows with \([\mu_{i}][\phi]=[\tau]\), then \([\mu_{1}][\theta]=[\mu_{2}][\theta]\) since \([\rho]\) is cocartesian in \(\cat{F}/\cat{D}\) by Lemma \ref{lem.stable}. Since \([\theta]\) is an isomorphism, \([\mu_{1}]=[\mu_{2}]\).

Conversely, since \([\phi]\) is cocartesian there is a unique arrow \([\psi]:\xi_{2}\ra\xi_{1}^{\#}\) in \(\cat{F}/\cat{D}(T_{2})\) with \([\psi\phi]=[\rho]\). By Lemma \ref{lem.stable} \([\rho]\) is cocartesian. It follows that \([\phi^{\#}]=[\psi]^{-1}\). 
\end{proof}
\section{Cohen-Macaulay approximation in fibred categories}
Given a category \(\cat{C}\) and a category \(\cat{A}\) fibred in abelian categories over \(\cat{C}\). Base change by an \(f:T\ra T'\) in \(\cat{C}\) applied to the objects in a complex \(\dots\ra N_{d}\ra N_{d-1}\ra\dots\) in \(\cat{A}(T)\) can by Lemma \ref{lem.keystable} be uniquely extended to a complex and yield a commutative diagram where the vertical arrows are the cocartesian base change arrows:
\begin{equation*}
\xymatrix@C-0pt@R-12pt@H0pt{
\dots \ar[r] & N_{d+1} \ar[r]\ar[d] & N_{d} \ar[r]\ar[d] & N_{d-1} \ar[d]\ar[r] & \dots \\   
\dots \ar[r] & N_{d+1}^{\#} \ar[r] & N_{d}^{\#} \ar[r] & N_{d-1}^{\#} \ar[r] & \dots 
}
\end{equation*}
Similarly base change of a commutative diagram \(\Delta\) in \(\cat{A}(T)\) gives a commutative diagram \(\Delta^{\#}\) and the base change arrows give an arrow of diagrams \(\Delta\ra\Delta^{\#}\).

Let \(\cat{X}\sbeq \cat{A}\) be a FAd subcategory.
Consider the following two conditions on the pair \((\cat{A},\cat{X})\) and an object \(T\) in \(\cat{C}\).
\begin{enumerate}
\item[(BC1)] If \(\alpha:A_{1}\ra A_{2}\) is an epimorphism in \(\cat{A}(T)\) and \(f:T\ra T'\) is an arrow in \(\cat{C}\) then any base change of \(\alpha\) by \(f\) is an epimorphism in \(\cat{A}(T')\).
\item[(BC2)] Let \(\xi:0\ra A\ra B\ra M\ra 0\) be an exact sequence in \(\cat{A}(T)\) with \(M\) in \(\cat{X}(T)\) and \(f:T\ra T'\) is an arrow in \(\cat{C}\). Then any base change of \(\xi\) by \(f\) is an exact sequence in \(\cat{A}(T')\).
\end{enumerate}
The first condition would be satisfied if base change had a right adjoint. The second condition mimics flatness for all objects in \(\cat{X}(T)\). 

The following is an elementary, but essential technical consequence of BC1.
\begin{lem}\label{lem.bs}
Let \(\cat{A}\) be a category fibred in abelian categories over \(\cat{C}\) which satisfies \textup{BC1} for \(T\) in \(\cat{C}\)\textup{.} Let \(c:\dots\ra L_{n}\ra L_{n-1}\ra\dots\) be an acyclic complex in \(\cat{A}(T)\) which remains exact after a base change \(\dots\ra L_{n}^{\#}\ra L_{n-1}^{\#}\ra\dots\) of \(c\) by \(f:T\ra T'\)\textup{.} Then  base change of \(K_{n}:=\ker\{d_{n-1}:L_{n-1}\ra L_{n-2}\}\) by \(f\) is isomorphic to \(\ker d_{n-1}^{\#}\) for all \(n\)\textup{.}
\end{lem}
\begin{proof}
Let \(Q_{n}=\ker d_{n-1}^{\#}\). Since the composition \(K_{n}^{\#}\ra L_{n-1}^{\#}\ra L_{n-2}^{\#}\) by Lemma \ref{lem.keystable} is zero (as \(K_{n}\ra K_{n}^{\#}\) is cocartesian), there is a factorisation \(\rho: K_{n}^{\#}\ra Q_{n}\) of \(K_{n}^{\#}\ra L_{n-1}^{\#}\). On the other hand the composition \(L_{n+1}^{\#}\ra L_{n}^{\#}\ra K_{n}^{\#}\) is zero too, hence there is an arrow from \(\coker d_{n+1}^{\#}\cong Q_{n}\) to \(K_{n}^{\#}\) which is a section of \(\rho\). By assumption \(L_{n}^{\#}\ra K_{n}^{\#}\) is an epimorphism. It follows that \(Q_{n}\cong K_{n}^{\#}\).
\end{proof}
\begin{defn}\label{defn.f}
Given FAd subcategories \(\cat{D}\sbeq\cat{X}\sbeq\cat{A}\).
Let \(\Xf(T)\) denote the additive subcategory of \(\cat{A}(T)\) with objects \(N\) which have a finite \(\cat{X}\)-resolution \(M_{*}\ra N\) which is preserved as resolution by any base change. Let \(\Xf\sbeq \cat{A}\) denote the resulting FAd subcategory.
Let \(\Df(T)\) denote the additive subcategory of \(\cat{A}(T)\) with objects \(L\) which have a \(\cat{D}(T)\)-resolution \(D^{*}\ra L\) which is preserved as resolution by any base change. Let \(\Df\sbeq\cat{A}\) denote the resulting FAd subcategory.
\end{defn}
The reasoning in the beginning of this section combined with Lemma \ref{lem.res} gives the following.
\begin{lem}\label{lem.forall}
Let \(\eta:\dots \ra E_{n}\ra E_{n-1}\ra\dots\) and \(\lambda:\dots\ra F_{n}\ra F_{n-1}\ra\dots\) be complexes in \(\cat{A}(T)\) and \(\eta^{\#}\) and \(\lambda^{\#}\) the complexes resulting from base change over \(f:T\ra T'\)\textup{.} If \(\eta\) is homotopic to \(\lambda\) then \(\eta^{\#}\) is homotopic to \(\lambda^{\#}\)\textup{.}

In particular\textup{;} if \(N\) in \(\cat{A}(T)\) has one \(\cat{DX}\)-resolution \textup{(}\(\cat{\hat{D}D}\)-coresolution\textup{)} which is preserved by base change then all \(\cat{DX}\)-resolutions \textup{(}\(\cat{\hat{D}D}\)-coresolutions\textup{)} are preserved by base change\textup{.}
\end{lem}
\begin{thm}\label{thm.cofapprox}
Let \(\cat{A}\) be a category fibred in abelian categories over \(\cat{C}\) and let \(\cat{D}\sbeq\cat{X}\sbeq\cat{A}\) be inclusion morphisms of categories fibred in additive categories\textup{.} Fix an object \(T\) in \(\cat{C}\)\textup{.} Assume \textup{BC1-BC2} for \((\cat{A},\cat{X})\) and \(T\)\textup{,} and \textup{AB1-AB2} for the triple of categories \((\cat{A}(T),\cat{X}(T),\cat{D}(T))\)\textup{.}
Then any object \(N\) in \(\Xf(T)\) admits an \(\cat{X}(T)\)-approximation and a \(\hat{\cat{D}}(T)\)-hull\textup{;}
\begin{equation*}\label{eq.cofapprox}
0\ra L\lra M\lra N\ra 0   \quad \text{and}  \quad 0\ra N\lra L'\lra M'\ra 0
\end{equation*}
with \(M\) and \(M'\) in \(\cat{X}(T)\) and \(L\) and \(L'\) in \(\Df(T)\)\textup{,} which are preserved by any base change\textup{.}
\end{thm}
\begin{proof}
The proof is a variation of the original proof of \cite[1.1]{aus/buc:89}.
For every \(N\) in \(\Xf(T)\) let \(r(N)\) denote the minimal length of an \(\cat{X}(T)\)-resolution \(M_{*}\thr N\) which is preserved by base change. The proof is by induction on \(r(N)\). If \(r(N)=0\) then \(N\) is in \(\cat{X}\) and so is its own \(\cat{X}\)-approximation, while AB2 provides a short exact sequence \(N\ra D\ra M'\) which is a \(\hat{\cat{D}}(T)\)-hull with \(D\) in \(\cat{D}(T)\sbeq\Df(T)\). The approximation is trivially preserved by base change, the hull because of BC2. Assume \(r=r(N)>0\) and let \(0\ra M_{r}\ra\dots\ra M_{0}\thr N\) be an \(\cat{X}(T)\)-resolution of minimal length preserved by base change. Then \(N_{1}=\ker(M_{0}\thr N)\) is in \(\Xf(T)\) by Lemma \ref{lem.bs} and \(r(N_{1})=r-1\). By induction there is a \(\hat{\cat{D}}(T)\)-hull \(N_{1}\ra L\ra M'_{1}\) with \(L\) in \(\Df\) which is preserved by base change. Pushout of \(e:N_{1}\ra M_{0}\ra N\) along \(N_{1}\ra L\) gives an \(\cat{X}(T)\)-approximation \(L\ra M\ra N\) by AB1. In the commutative diagram obtained by a base change;
\begin{equation}
\xymatrix@C-0pt@R-12pt@H0pt{
N_{1}^{\#} \ar[r]\ar[d] & M_{0}^{\#} \ar[r]\ar[d] & N^{\#} \ar@{=}[d]  \\   
L^{\#} \ar[r]\ar[d] & M^{\#} \ar[r]\ar[d] & N^{\#} \\
(M'_{1})^{\#} \ar@{=}[r] & (M'_{1})^{\#}
}
\end{equation}
the upper row (by Lemma \ref{lem.bs}) and the columns (by BC2) are short exact sequences.
It follows that the middle row is a short exact sequence.

By AB2 there is a short exact sequence \(M\ra D\ra M'\) with \(D\) in \(\cat{D}(T)\) and \(M'\) in \(\cat{X}(T)\). Pushout of \(M\ra D\ra M'\) along \(M\ra N\) gives a short exact sequence \(h:N\ra L'\ra M'\). Since the induced sequence \(L\ra D\ra L'\) is short exact, \(L'\) is contained in \(\hat{\cat{D}}(T)\). Applying a base change we obtain the following commutative diagram:
\begin{equation}
\xymatrix@C-0pt@R-12pt@H0pt{
L^{\#} \ar[r]\ar@{=}[d] & M^{\#} \ar[r]\ar[d] & N^{\#} \ar[d]  \\   
L^{\#} \ar[r] & D^{\#} \ar[r]\ar[d] & (L')^{\#} \ar[d] \\
& (M')^{\#} \ar@{=}[r] & (M')^{\#}
}
\end{equation}
The upper row and (by BC2) the two columns are short exact sequences. It follows that the middle row is a short exact sequence and hence that \(L'\) is contained in \(\Df\).
\end{proof}
Sequences as in Theorem \ref{thm.cofapprox} preserved by any base change will be called an \emph{\(\cat{X}\)-approximation} and a \emph{\(\Df\)-hull} of \(N\) respectively.

Lemma \ref{lem.stable} makes the following definition reasonable. Three categories fibred in additive categories (FAds) \(\cat{A}_{i}\), \(i=1,2,3\), an inclusion of FAds \(\cat{A}_{1}\sbeq \cat{A}_{2}\), and a morphism of FAds \(F:\cat{A}_{2}\ra \cat{A}_{3}\) equivalent to the quotient morphism \(\cat{A}_{2}\ra\cat{A}_{2}/\cat{A}_{1}\) is called a \emph{short exact sequence of categories fibred in additive categories} and is denoted by \(0\ra \cat{A}_{1}\ra\cat{A}_{2}\ra\cat{A}_{3}\ra 0\).
\begin{thm}\label{thm.cofmain}
Let \(\cat{A}\) be a category fibred in abelian categories over \(\cat{C}\) and let \(\cat{D}\sbeq\cat{X}\sbeq\cat{A}\) be inclusion morphisms of categories fibred in additive categories\textup{.} Assume \textup{BC1-BC2} for the pair \((\cat{A},\cat{X})\) and \textup{AB1-AB3} for the triple of categories \((\cat{A}(T),\cat{X}(T),\cat{D}(T))\)\textup{,} for all objects \(T\) in \(\cat{C}\)\textup{.} 
Then\textup{:} 
\begin{enumerate}
\item[(i)] The \(\cat{X}\)-approximation induces a morphism of categories fibred in additive categories \(j^{!}:\Xf/\cat{D}\ra\cat{X}/\cat{D}\) which is a right adjoint to the full and faithful inclusion morphism \(j_{!}:\cat{X}/\cat{D}\ra\Xf/\cat{D}\)\textup{.}
\item[(ii)] The \(\Df\)-hull induces a morphism of categories fibred in additive categories \(i^{*}:\Xf/\cat{D}\ra\Df/\cat{D}\) which is a left adjoint to the full and faithful inclusion morphism \(i_{*}:\Df/\cat{D}\ra \Xf/\cat{D}\)\textup{.}
\item[(iii)] Together these maps give the following commutative diagram of short exact sequences of categories fibred in additive categories\textup{:}
\begin{equation*}
\xymatrix@C-0pt@R-8pt@H-30pt{
0\ar[r] & \, \Df/\cat{D} \ar[r]^{i_{*}}\ar[d]_{\id} & \Xf/\cat{D}\ar[r]^{j^{!}}\ar@{=}[d] & \cat{X}/\cat{D} \ar[r] & 0 \\
0 & \Df/\cat{D}\ar[l] & \Xf/\cat{D}\ar[l]_{i^{*}} & \,  \cat{X}/\cat{D}  \ar[l]_{j_{!}}\ar[u]_{\id} & 0 \ar[l]
}
\end{equation*}
\end{enumerate}
\end{thm}
\begin{proof}
In each fibre most of these statements are true by the arguments in the proof of \cite[2.8]{aus/buc:89} since we have Theorem \ref{thm.cofapprox}. The general cases are reduced to fibre cases by applying base change. First we have to establish the functors. Note that the quotient categories involved are FAds over \(\cat{C}\) by Lemma \ref{lem.stable}. Let \(p:\cat{A}\ra\cat{C}\) denote the fibration. For each \(N_{i}\) in \(\Xf\) put \(T_{i}=p(N_{i})\) and choose a \(\Df\)-hull \(\iota_{i}:N_{i}\ra L_{i}\ra M_{i}\) which exists by Theorem \ref{thm.cofapprox} and such that \(\iota_{i}=\id\) if \(N_{i}\) is in \(\Df\). For each arrow \(\psi:N_{1}\ra N_{2}\) choose an arrow \(\lambda_{21}:L_{1}\ra L_{2}\) commuting with \(\psi\). This arrow is obtained as a composition of a base change \(L_{1}\ra L_{1}^{\#}\) over \(p(\psi):T_{1}\ra T_{2}\) with an extension \(L_{1}^{\#}\ra L_{2}\) of \(N_{1}^{\#}\ra L_{2}\) obtained since \(\xt{1}{\cat{A}(T_{2})}{M_{1}^{\#}}{L_{2}}=0\) by \cite[2.5]{aus/buc:89}. If composable it follows from \cite[2.8]{aus/buc:89} that \(\lambda_{32}\lambda_{21}\sim\lambda_{31}\).

There is a unique arrow \(\phi:N_{1}^{\#}\ra N_{2}\) induced by \(\psi\).
If \(\lambda_{21}':L_{1}\ra L_{2}\) is an extension of \(\psi':N_{1}\ra N_{2}\) with \(p(\psi')=p(\psi)\) such that \(\delta_{1}:=\psi-\psi'\) is equivalent to \(0\), we have by Lemma \ref{lem.keystable} that \(\delta:N_{1}^{\#}\ra N_{2}\) induced from \(\delta_{1}\) by base change factors through an object \(D\) in \(\cat{D}(T_{2})\). It follows that \(\delta\) factors through \(N_{1}^{\#}\ra L_{1}^{\#}\). Let \(\tau\) denote the composition \(L_{1}^{\#}\ra N_{2}\ra L_{2}\) (so \(\tau\sim 0\)).
Let \(\eta\) be a base change over \(p(\psi)\) of the difference of the two extensions; \(\eta=(\lambda_{21}-\lambda_{21}')^{\#}\). One calculates that \((\eta-\tau)\iota_{1}^{\#}=0\), hence \(\eta-\tau\) is induced by an arrow \(M_{1}^{\#}\ra L_{2}\) which lifts to an arrow \(M_{1}^{\#}\ra D^{0}\) where \(D^{*}\thr L_{2}\) is a finite \(\cat{D}\)-resolution of \(L_{2}\) (since \(\xt{1}{\cat{A}(T_{2})}{M_{1}^{\#}}{\hat{\cat{D}}(T_{2})}=0\)). 
Hence \(\eta-\tau\sim 0\), so \(\eta\sim 0\) and \(\lambda_{21}\sim\lambda_{21}'\). We have shown that \(i^{*}:\Xf/\cat{D}\ra \Df/\cat{D}\) is a well-defined functor. To show that \(i^{*}\) preserves cocartesian arrows we apply Lemma \ref{lem.stablecc}: If \([\psi]:N_{1}\ra N_{2}\) is cocartesian in \(\Xf/\cat{D}\) then the induced map \([\phi]:N_{1}^{\#}\ra N_{2}\) is an isomorphism and by \cite[2.8]{aus/buc:89} so is any extension \(L_{1}^{\#}\ra L_{2}\) of \([\phi]\). Composed with the base change \(L_{1}\ra L_{1}^{\#}\) we get a cocartesian arrow in \(\Df/\cat{D}\) by Lemma \ref{lem.stablecc}.

A similar argument gives that the morphism \(j^{!}: \Xf/\cat{D}\ra \cat{X}/\cat{D}\) induced by (choices of) \(\cat{X}\)-approximation also is well-defined as a map of fibred categories.

To prove adjointness for the pair \((j_{!},j^{!})\) consider the chosen \(\cat{X}\)-approximation \(L\ra M\xra{\pi} N\) of \(N\) in \(\Xf(T)\). Given \(\phi_{1}:M_{1}\ra N\) with \(M_{1}\) in \(\cat{X}(T_{1})\) and \(f=p(\phi_{1})\). Let \(\phi: M_{1}^{\#}\ra N\) be induced by a base change of \(M_{1}\) by \(f\). Since \(\xt{1}{\cat{A}(T)}{M_{1}^{\#}}{L}=0\), \(\phi\) can be lifted to an arrow \(\psi:M_{1}^{\#}\ra M\). Composing \(\psi\) with the base change \(M_{1}\ra M_{1}^{\#}\) gives a lifting of \(\phi_{1}\) which shows surjectivity of the adjointness map \([\pi\circ -]\). To prove injectivity consider for \(i=2,3\) arrows \(\psi_{i}:M_{1}\ra M\) in \(\cat{X}\) with \(\pi\psi_{2}=\pi\psi_{3}\). Since \(p(\pi)=\id\) we have \(p(\psi_{2})=p(\psi_{3})=f\) and we can define \(\psi_{1}=\psi_{2}-\psi_{3}\) with \(\pi\psi_{1}\sim 0\). Base change by \(f\) induces a \(\psi:M_{1}^{\#}\ra M\) from \(\psi_{1}\). Lemma \ref{lem.keystable} gives \(\pi\psi\sim 0\). The argument in \cite[2.8]{aus/buc:89} implies that \(\psi\) and hence \(\psi_{1}\) factors through an object in \(\cat{D}(T)\). Analogous arguing gives the adjointness of the pair \((i^{*},i_{*})\).

The commutativity of the diagram in \textup{(iii)} follows by definition. For \(i^{*}j_{!}=0=j^{!}i_{*}\) see \cite[2.8]{aus/buc:89}. We prove exactness in the upper row. Given \(\phi:N_{1}\ra N_{2}\) in \(\Xf\) with \(f=p(\phi):T_{1}\ra T_{2}\) such that \(j^{!}[\phi]=0\). If \(\pi_{i}:M_{i}\ra N_{i}\) are the chosen \(\cat{X}\)-approximations, \(j^{!}[\phi]\) is represented by a lifting \(\psi:M_{1}\ra M_{2}\) and the assumption is that \(\psi\) factors through an object \(D\) of \(\cat{D}\). We claim that \(\phi\) factors through an object in \(\Df\). By base change it's sufficient to prove the special case \(f=\id_{T_{2}}\). If \(M\) is any object in \(\cat{X}(T)\) we have that the composition \(\xt{1}{\cat{A}(T)}{M}{N_{1}}\cong \xt{1}{\cat{A}(T)}{M}{M_{1}} \xra{\psi_{*}}\xt{1}{\cat{A}(T)}{M}{M_{2}}\cong \xt{1}{\cat{A}(T)}{M}{N_{2}}\) is \(\phi_{*}\) which hence equals \(0\). If \(e:N_{1}\ra L_{1}'\ra M_{1}'\) is a \(\Df\)-hull of \(N_{1}\), the connecting takes \(\phi\) to \(\phi_{*}e\in \xt{1}{\cat{A}(T)}{M_{1}'}{N_{2}}\), i.e.\ to \(0\), and so there exists a \(\delta:L_{1}'\ra N_{2}\) which induces \(\phi\). Exactness in the lower row is analogous.
\end{proof}
\begin{prop}\label{prop.cof}
Let \(\cat{A}\) be a category fibred in abelian categories over \(\cat{C}\) and let \(\cat{D}\sbeq\cat{X}\sbeq\cat{A}\) be inclusion morphisms of categories fibred in additive categories\textup{.} Fix an object \(T\) in \(\cat{C}\)\textup{.} Assume \textup{BC1-BC2} for \((\cat{A},\cat{X})\) and \(T\), and \textup{AB1-AB3} for the triple of categories \((\cat{A}(T),\cat{X}(T),\cat{D}(T))\)\textup{.} Then\textup{:}
\begin{enumerate}
\item[(i)] \(\Df(T)=\Xf(T)\cap \cat{X}(T)^{\perp}\) and \(\cat{D}(T)=\cat{X}(T)\cap \Df(T)\)\textup{.}
\item[(ii)] \(\Xf(T)\) and \(\Df(T)\) are closed under extensions\textup{.}
\item[(iii)] Exact sequences \(\dots\ra N_{n}\xra{d_{n}} N_{n-1}\ra \dots\) with objects \(N_{i}\) and kernels \(\ker d_{i}\) in \(\Xf(T)\) remain exact after base change\textup{.}
\end{enumerate}
If in addition \textup{AB4}\textup{,} then\textup{:}
\begin{enumerate}
\item[(iv)] Epimorphisms in \(\Xf(T)\) are admissible\textup{.}
\item[(v)] \(\Add \Xf(T) =\Xf(T)\) and \(\Add \Df(T) = \Df(T)\)\textup{.}
\end{enumerate}
\end{prop}
\begin{proof}
For (i) the proofs of \cite[3.6-7]{aus/buc:89} work with the fl-s too by Theorem \ref{thm.cofapprox}. By (i) it's sufficient to prove (ii) for \(\Xf(T)\). Let \(e: N_{1}\ra N_{2}\ra N_{3}\) be a short exact sequence in \(\hat{\cat{X}}(T)\).
If \(N_{1}\) and \(N_{3}\) are in \(\Xf(T)\), Theorem \ref{thm.cofapprox} and Lemma \ref{lem.forall} imply that \({}^{-}C^{*}(N_{1})\) and \({}^{-}C^{*}(N_{3})\) in Lemma \ref{lem.forall} are preserved as resolutions by base change. Together with BC2 this implies that base change of the short exact sequence of resolutions in Lemma \ref{lem.xres} (i) gives a short exact sequence of \(\cat{DX}\)-resolutions. 
For (iii) the long exact sequence is broken into short exact sequences \(\xi_{i}\) with objects in \(\Xf(T)\). By Lemma \ref{lem.bs} it is sufficient to prove the claim for short exact sequences. By Lemma \ref{lem.xres} the short exact sequence of \(\cat{DX}\)-resolutions in Lemma \ref{lem.xres} (i) is preserved by base change. It follows that the short exact sequences \(\xi_{i}\) remain exact after base change.

For (iv); if \(N_{2}\) and \(N_{3}\) in \(e\) are in \(\Xf(T)\) then \(N_{1}\) is in \(\hat{\cat{X}}(T)\) by AB4; see \cite[3.5]{aus/buc:89}. The argument proceeds as for extensions. 

By (i) it's sufficient to prove (v) for \(\Xf\). If \(N_{1}\amalg  N_{2}\) is an object in \(\Xf(T)\), then \(N_{i}\) is in \(\hat{\cat{X}}(T)\) for \(i=1,2\) by \cite[3.4]{aus/buc:89}. By Lemma \ref{lem.forall} \({}^{-}C^{*}(N_{1})\amalg  {}^{-}C^{*}(N_{2})\) is preserved by base change as resolution of \(N_{1}\amalg  N_{2}\). It follows that the resolution \({}^{-}C^{*}(N_{i})\) is preserved by base change for \(i=1,2\). 
\end{proof}
\begin{cor}\label{cor.Cbs}
Assume \textup{BC1-BC2} and \textup{AB1-AB3} as in \textup{Proposition \ref{prop.cof}.}
If \(N\) is in \(\Xf(T)\) then any \(\cat{DX}\)-resolution \({}^{-}C^{*}(N)\thr N\) and any \(\hat{\cat{D}}\cat{D}\)-coresolution \(N\rightarrowtail {}^{+}C^{*}(N)\) as in \textup{Section \ref{subsec.cplx}} is preserved by base change\textup{.}
\end{cor}
\begin{proof}
By Theorem \ref{thm.cofapprox} there exist a \(\cat{DX}\)-resolution and a \(\cat{\hat{D}D}\)-coresolution in \(\Xf\). 
The result follows from Proposition \ref{prop.cof} (iii) and Lemma \ref{lem.forall}. 
\end{proof}
\begin{defn}\label{defn.cof}
Let \(\cat{A}\) be a category fibred in abelian categories over \(\cat{C}\) and let \(\cat{D}\sbeq\cat{A}\) be an inclusion morphism of a category fibred in additive categories. For \(T\) in \(\cat{C}\) let \(\cDf(T)\) denote the full subcategory of \(\cat{A}(T)\) of objects \(K\) with a finite coresolution \(K\ra L^{*}\) with objects \(L^{i}\) in \(\Df(T)\) for \(i\geq 0\).
\end{defn}
\begin{lem}\label{lem.cof2} 
With these notions we have\textup{:}
\begin{enumerate}
\item[(i)] Epimorphisms in \(\Df(T)\) are admissible if and only if\, \(\Df(T)=\cDf(T)\)\textup{.}
\end{enumerate}
Assume \textup{BC1-BC2} and \textup{AB1-AB4} for \((\cat{A}(T),\cat{X}(T),\cat{D}(T))\)\textup{.} Then\textup{:}
\begin{enumerate}
\item[(ii)] \(\Df(T)=\hat{\cat{D}}(T)\cap\cDf(T)\)\textup{.}
\item[(iii)] Epimorphisms in \(\Df(T)\) are admissible if epimorphisms in \(\hat{\cat{D}}(T)\) are admissible\textup{.} 
\end{enumerate}
\end{lem}
\begin{proof}
(i) is trivially true. In (ii) \(\Df(T)\sbeq\hat{\cat{D}}(T)\cap\cDf(T)\) is obvious. For the other inclusion, suppose \(K\) is an object in \(\hat{\cat{D}}(T)\cap\cDf(T)\setminus \Df(T)\) with a \(\Df(T)\)-coresolution \(K\ra L^{*}\) of length \(n>0\). Since monomorphisms are admissible in \(\hat{\cat{D}}(T)=\hat{\cat{X}}(T)\cap \cat{X}(T)^{\perp}\), all \(K^{i}=\ker(L^{i}\ra L^{i+1})\) are contained in \(\hat{\cat{D}}(T)\), and we can assume \(n=1\). But then \(K\) has to be in \(\Xf\cap\hat{\cat{D}}=\Df(T)\) by Proposition \ref{prop.cof}. Since (i) is true ``without the fl'' by \cite[4.1]{aus/buc:89}, (iii) follows immediately from (i) and (ii).
\end{proof}
\section{Cohen-Macaulay approximation of flat families}\label{sec.ft}
We define fibred categories of Cohen-Macaulay maps with flat modules and show that they allow Cohen-Macaulay approximation in the finite type case and the local, algebraic case.  
\subsection{The finite type case}\label{subsec.ft}
Let \(h:S\ra T\) be a ring homomorphism of noetherian rings. We say that \(h\) is a \emph{Cohen-Macaulay \textup{(}CM\textup{)} map} if it is of finite type, faithfully flat and all fibres are Cohen-Macaulay (cf.\ \cite[6.8.1]{EGAIV2}). In particular \(h\) is equidimensional (\cite[15.4.1]{EGAIV3}). B.\ Conrad has defined the \emph{dualising module} \(\omega_{h}\) for any CM \(h\); see \cite[Section 3.5]{con:00}.
Suppose \(h\) has pure relative dimension \(n\). For some \(N\geq n\) there is a surjective \(S\)-algebra map \(P\ra T\) where \(P=S[t_{1},\dots,t_{N}]\). Let \(\omega_{P/S}:=\bigwedge^{N}\varOmega_{P/S}\). Then there is an isomorphism \(\omega_{h}\cong\xt{N-n}{P}{T}{\omega_{P/S}}\) which is natural in the factorisation \(S\ra P\ra T\); see \cite[3.5.3-6]{con:00}. By (local) duality theory and Corollary \ref{cor.xtdef} \(\omega_{h}\) is \(S\)-flat
(or see \cite[Cor.\ 3.5.2]{con:00}).
If \(S\) is a field, we have that \(\omega_{h}\) is a canonical module of \(T\) as in Example \ref{ex.MCMapprox}; cf.\ \cite[3.3.7 and 16]{bru/her:98}. 

Let \(\cat{CM}\) be the category with objects the CM maps and morphisms \((g,f):h_{1}\ra h_{2}\) pairs of ring homomorphisms \(g:S_{1}\ra S_{2}\) and \(f:T_{1}\ra T_{2}\) such that \(h_{2}g=fh_{1}\) and such that the induced map \(f\ot 1:T_{1}\ot S_{2}\ra T_{2}\) is an isomorphism:
\begin{equation*}
\xymatrix@C-0pt@R-8pt@H-30pt{
T_{1}\ar[r]^{f} & T_{2} & T_{1}\ot_{S_{1}}S_{2} \ar[l]_(0.6){\simeq} \\
S_{1}\ar[u]^{h_{1}}\ar[r]^{g} & S_{2}\ar[u]_{h_{2}}
}
\end{equation*}
Let \(\cat{NR}\) denote the category of noetherian rings. The forgetful functor \(p:\cat{CM}\ra\cat{NR}\); \((g,f)\mapsto g\), makes \(\cat{CM}\) fibred in groupoids over \(\cat{NR}\). The essential part is that \(\cat{CM}\) should allow base change, i.e.\ given \(g:S_{1}\ra S_{2}\) and \(h_{1}:S_{1}\ra T_{1}\) as above there should exist a \(T_{2}\), an \(h_{2}:S_{2}\ra T_{2}\) and an \(f\) such that \((g,f)\) is a morphism \(h_{1}\ra h_{2}\) in \(\cat{CM}\). This follows from \cite[15.4.3]{EGAIV3}.

Let \(\cat{mod}\) be the category of pairs \((h:S\ra T,N)\) with \(h\) in \(\cat{CM}\) and \(N\) a finite \(T\)-module. A morphism \((h_{1},N_{1})\ra (h_{2},N_{2})\) is a morphism \((g,f):h_{1}\ra h_{2}\) in \(\cat{CM}\) and an \(f\)-linear map \(\alpha:N_{1}\ra N_{2}\). Then \(\alpha\) is cocartesian with respect to the forgetful functor \(F:\cat{mod}\ra \cat{CM}\) if \(1\ot \alpha: T_{2}\ot N_{1}\ra N_{2}\) is an isomorphism. It follows that \(\cat{mod}\) is fibred in abelian categories over \(\cat{CM}\). 
Adding the property that \(N\) is \(S\)-flat gives the full subcategory \(\modf\). Moreover, let \(\cat{MCM}\) be the full subcategory of \(\modf\) where the fibre \(N_{s}=N\ot_{S}k(s)\) is a maximal Cohen-Macaulay \(T_{s}\)-module for all \(s\in \Spec S\). The inclusions \(\cat{MCM}\sbeq\modf\sbeq\cat{mod}\) are inclusion morphisms of categories fibred in additive categories (FAds) over \(\cat{CM}\). For \(\cat{MCM}\) this follows from \cite[15.4.3]{EGAIV3}.
If \(h\) is a CM map let \(\cat{mod}_{h}\), \(\cat{MCM}_{h}\), \dots denote the fibre categories of \(\cat{mod}\), \(\cat{MCM}\), \dots over \(h\). An object in \(\cat{MCM}_{h}\) is called an (\(h\)-)family of maximal Cohen-Macaulay modules.

Given a morphism \(h_{1}\ra h_{2}\) in \(\cat{CM}\). By \cite[Thm.\ 3.6.1]{con:00} there is a natural isomorphism with base change \(T_{2}\ot\omega_{h_{1}}\cong\omega_{h_{2}}\) which is compatible with localisation of \(T_{1}\) and is functorial with respect to composition \(h_{1}\ra h_{2}\ra h_{3}\). It follows that \(h\mapsto (h,\omega_{h})\) defines a morphism \(\omega:\cat{CM}\ra\cat{MCM}\) of fibred categories over \(\cat{NR}\) which is a section of the forgetful \(F:\cat{MCM}\ra\cat{CM}\). Let \(\cat{D}\) be the full subcategory of \(\cat{MCM}\) over \(\cat{CM}\) with the objects \((h,D)\) where \(D\) is an object in \(\Add\{\omega_{h}\}\). The inclusion \(\cat{D}\sbeq\cat{MCM}\) is an inclusion of FAds over \(\cat{CM}\).

If \(\cat{U}\) denotes any of these FAds over \(\cat{CM}\), let \(\ul{\cat{U}}\) denote the quotient (`stable') category \(\cat{U}/\cat{D}\). With this notation we have the following.
\begin{thm}\label{thm.flatCMapprox}
The pair \((\cat{mod},\cat{MCM})\) over \(\cat{CM}\) satisfies \textup{BC1-BC2} and the triple of fibre categories \((\cat{mod}_{h},\cat{MCM}_{h},\cat{D}_{h})\) satisfies \textup{AB1-AB4} for all objects \(h\) in \(\cat{CM}\)\textup{.}  
Moreover\textup{:}
\begin{enumerate}
\item[(i)] The fibred categories \(\MCMf\) and \(\Df\) equal \(\modf\) and \(\hat{\cat{D}}\cap\modf\) respectively\textup{.}
\item[(ii)] For any object \((h,N)\) in \(\modf\), \(N\) admits an \(\cat{MCM}\)-approximation and a \(\Df\)-hull which in particular are preserved by any base change\textup{.}
\item[(iii)] The \(\cat{MCM}\)-approximation induces a morphism of categories fibred in additive categories \(j^{!}:\umodf\ra\ul{\cat{MCM}}\) which is a right adjoint to the full and faithful inclusion morphism \(j_{!}:\ul{\cat{MCM}}\ra\umodf\)\textup{.}
\item[(iv)] The \(\Df\)-hull induces a morphism of categories fibred in additive categories \(i^{*}:\umodf\ra{\uDf}\) which is a left adjoint to the full and faithful inclusion morphism \(i_{*}:\uDf\ra \umodf\)\textup{.}
\item[(v)] Together these maps give the following commutative diagram of short exact sequences of categories fibred in additive categories\textup{:}
\begin{equation*}
\xymatrix@C-0pt@R-8pt@H-30pt{
0\ar[r] & \,\uDf \ar[r]^(0.45){i_{*}}\ar[d]_(0.45){\id} & \umodf\ar[r]^{j^{!}}\ar@{=}[d] & \ul{\cat{MCM}}\ar[r] & 0 \\
0 & \,\uDf\ar[l] & \umodf\ar[l]_{i^{*}} & \ul{\cat{MCM}}\ar[l]_(0.45){j_{!}}\ar[u]_{\id} & 0 \ar[l]
}
\end{equation*}
\end{enumerate}
\end{thm}
\begin{proof}
Since base change is given by the tensor product BC1 and BC2 follow. In particular, a short exact sequence \(e:N_{1}\ra N_{2}\ra N_{3}\) in \(\cat{mod}_{h}\) with \(N_{3}\) and either \(N_{1}\) or \(N_{2}\) in \(\cat{MCM}_{h}\) gives short exact sequences of MCMs after base change to each fibre \(T_{s}\), i.e.\ AB1 and AB4.

For AB2, suppose \(M\) is in \(\cat{MCM}_{h}\). Then \(M^{\vee}=\hm{}{T}{M}{\omega}\) is in \(\cat{MCM}_{h}\) too by Corollary \ref{cor.xtdef}. Since \(M^{\vee}\) is finite there is a short exact sequence
\(M^{\vee}\la T^{r}\la M_{1}\). By AB1 \(M_{1}\) is in \(\cat{MCM}_{h}\). Applying \(\hm{}{T}{-}{\omega_{h}}\) gives the desired short exact sequence since Corollary \ref{cor.xtdef} implies that \(\xt{1}{T}{M^{\vee}}{\omega_{h}}=0\) and that the natural map \(M\ra M^{\vee\vee}\) is an isomorphism. AB3 also follows from Corollary \ref{cor.xtdef}.

Any \(N\) in \(\MCMf_{h}\) has by definition a finite \(\cat{MCM}\)-resolution (say of length \(n\)) preserved by base change. Since objects in \(\cat{MCM}\) are \(S\)-flat it follows by induction on \(n\) that \(\tor{S}{1}{N}{k(s)}=0\) for all \(s\in\Spec S\). Hence \(N\) is \(S\)-flat. Conversely, if \(N\) is in \(\modf_{h}\) it follows that a sufficiently high syzygy of \(N\) is in \(\cat{MCM}_{h}\), i.e.\ \(N\) is in \(\MCMf_{h}\). 
With Proposition \ref{prop.cof} this gives \(\Df_{h}=\modf_{h}\cap\cat{MCM}_{h}^{\perp}\). By induction on the length of the resolution \(\hat{\cat{D}}_{h}\sbeq \cat{MCM}_{h}^{\perp}\) and so \(\modf_{h}\cap\hat{\cat{D}}_{h}\sbeq\Df_{h}\). The opposite inclusion is clear by the first part of (i). Now (ii)-(v) follow directly from Theorem \ref{thm.cofmain}.
\end{proof}
\begin{cor}\label{cor.flatCMapprox}
Let \(h:S\ra T\) be an object in \(\cat{CM}\)\textup{.} Then\textup{:}
\begin{enumerate}
\item[(i)] \(\cat{D}_{h}=\cat{MCM}_{h}^{\perp}\cap\cat{MCM}_{h}\quad\text{and}\quad \Df_{h}=\cat{MCM}_{h}^{\perp}\cap\modf_{h}\)
\item[(ii)] The kernel of a surjective map in \(\Df_{h}\) is contained in \(\Df_{h}\)\textup{.}
\end{enumerate}
\end{cor}
\begin{proof}
(ii): Note that if \(N_{1}\ra N_{2}\ra N_{3}\) is a short exact sequence with \(N_{2}\) and \(N_{3}\) in \(\Df_{h}\), in particular \(S\)-flat, then \(N_{1}\) has to be \(S\)-flat too. To show that \(N_{1}\) is in \(\hat{\cat{D}}_{h}\) we use the criterion in \cite[4.6]{aus/buc:89} to show that \(\check{\cat{D}}_{h}=\hat{\cat{D}}_{h}\): 
Suppose that \(M\) is in \(\cat{MCM}_{h}\). Assume that \(M\) satisfies \(\injdim{\cat{MCM}_{h}}{M}= n<\infty\) which by \cite[4.3]{aus/buc:89} is equivalent to the existence of a coresolution of \(M\) of length \(n\) in \(\cat{D}_{h}\). The fibre at any \(s\in\Spec S\) gives a  \(\cat{D}_{T_{s}}\)-coresolution of the MCM \(T_{s}\)-module \(M_{s}\). Since \(\check{\cat{D}}_{T_{s}}=\hat{\cat{D}}_{T_{s}}\) (\cite[6.3]{aus/buc:89}) it follows by \cite[4.6]{aus/buc:89} that \(M_{s}\) is contained in \(\cat{D}_{T_{s}}\). Since \(\cat{D}_{T_{s}}=\cat{MCM}_{T_{s}}\cap\cat{MCM}_{T_{s}}^{\perp}\) by \cite[3.7]{aus/buc:89} it follows from Corollary \ref{cor.xtdef} that \(\injdim{\cat{MCM}_{h}}{M}=0\) and so \(M\) is in \(\cat{MCM}_{h}\cap\cat{MCM}_{h}^{\perp}\). But by Theorem \ref{thm.flatCMapprox} we can invoke \cite[3.7]{aus/buc:89} again which gives \(\cat{MCM}_{h}\cap\cat{MCM}_{h}^{\perp}=\cat{D}_{h}\) so \(M\) is in \(\cat{D}_{h}\). By \cite[4.6(d)]{aus/buc:89} \(\check{\cat{D}}_{h}=\hat{\cat{D}}_{h}\) follows and \(N_{1}\) is in \(\hat{\cat{D}}_{h}\).
\end{proof}
\begin{rem}
Let \(\vL\) be any noetherian ring. By abuse of notation let \(\cat{CM}\ra\QR\) denote the category fibred in groupoids obtained by restriction to the category of noetherian \(\vL\)-algebras \(\QR\). Fix a Cohen-Macaulay map \(\vL\ra T^{\mspace{-1mu}\mathit{o}}\). There is a section \(s:\QR\ra\cat{CM}\) defined by \(s:S\mapsto (S\ra T^{\mspace{-1mu}\mathit{o}}\ot S)\). Let \(\cat{T^{o}}\) denote the resulting fibred subcategory of \(\cat{CM}\). We restrict \(\cat{mod}\), \(\cat{MCM}\) and \(\cat{D}\) to \(\cat{T^{o}}\) and obtain categories fibred in abelian and additive categories over \(\cat{T^{o}}\) respectively. These restricted fibred categories satisfy the axioms AB1-AB4 and BC1-BC2 and we obtain restricted versions of Theorem \ref{thm.flatCMapprox} and Corollary \ref{cor.flatCMapprox}.
\end{rem}
Let \(\cat{P}\) denote the fibred subcategory of \(\cat{MCM}\) over \(\cat{CM}\) of pairs \((h,P)\) with \(h:S\ra T\) in \(\cat{CM}\) and \(P\) a finite projective \(T\)-module. Let \(\Pf\) denote the full subcategory of \(\modf\) of pairs \((h,Q)\) such that \(Q\) has a finite projective dimension. The inclusions of categories fibred in additive categories \(\cat{P}\sbeq\Pf\sbeq\modf\) are closed under extensions over \(\cat{CM}\).
\begin{lem}\label{lem.proj}
There is an exact equivalence \(\Df\simeq\Pf\) of categories fibred in additive categories defined by the functor \((h,L)\mapsto(h,\hm{}{T}{\omega_{h}}{L})\) with a quasi-inverse \((h,Q)\mapsto (h,Q\ot_{T}\omega_{h})\)\textup{.} It induces an equivalence of fibred quotient categories \(\Df/\cat{D}\simeq\Pf/\cat{P}\)\textup{.}
\end{lem}
\begin{proof}
By base change we reduce to the absolute case where the equivalence is well-known; cf.\ \cite[I 4.10.16]{has:00}. The functor \(\hm{}{-}{\omega_{-}}{-}\) takes a map \(\phi:L_{1}\ra L_{2}\) over \(h_{1}\ra h_{2}\) to the map \(\hm{}{T_{1}}{\omega_{h_{1}}}{L_{1}}\ra\hm{}{T_{2}}{\omega_{h_{2}}}{L_{2}}\); \(\alpha\mapsto(\phi\alpha)^{\#}\), the unique map which pulls back to \(\phi\alpha\) by the cocartesian map \(\omega_{h_{1}}\ra\omega_{h_{2}}\). It is a well-defined functor since the dualising module is functorial.
If \(L\) is in \(\Df_{h}\) then \(\xt{i}{T_{s}}{\omega_{T_{s}}}{L_{s}}=0\) for all \(i\neq 0\) and \(s\in \Spec S\). By Corollary \ref{cor.xtdef} the functor is exact and \(\hm{}{T}{\omega_{h}}{L}\) is \(S\)-flat. In particular \(\nd{}{T}{\omega_{h}}\cong T\). If \(D\) is in \(\cat{D}_{h}\) then \(\hm{}{T}{\omega_{h}}{D}\) is projective as a direct summand of a free module. If \(D^{*}\thr L\) is a finite \(\cat{D}_{h}\)-resolution of \(L\), then \(\hm{}{T}{\omega_{h}}{D^{*}}\) gives a projective resolution of \(\hm{}{T}{\omega_{h}}{L}\) since \(\Df_{h}\sbeq\cat{MCM}_{h}^{\perp}\) (Corollary \ref{cor.flatCMapprox}). The natural map \(\ev:\hm{}{T}{\omega_{h}}{L}\ot\omega_{h}\ra L\) commutes with base change to the fibres where it is an isomorphism (\cite[9.6.5]{bru/her:98}) and Nakayama's lemma and \(S\)-flatness implies that \(\ev\) is an isomorphism too. Let \(P^{*}\ra Q\) be a \(\cat{P}\)-resolution of \(Q\) in \(\Pf_{h}\) of length \(n\). Define covariant functors \(G^{i}:\cat{mod}_{S}\ra\cat{mod}_{T}\) by \(G^{i}(V):=\cH^{i-n}(P^{*}\ot_{T}\omega_{h}\ot_{S}V)\). Since \(P^{*}\ot_{T}\omega_{h}\) is \(S\)-flat, \(\{G^{i}\}\) defines a cohomological \(\delta\)-functor. Since \(P^{*}\ot k(s)\) gives a resolution of \(Q\ot_{S} k(s)\) for all \(s\in\Spec S\), \cite[9.6.5]{bru/her:98} and Proposition \ref{prop.nakayama} implies that \(G^{n}(S)=Q\ot_{T}\omega_{h}\) is \(S\)-flat and \(P^{*}\ot_{T}\omega_{h}\thr Q\ot_{T}\omega_{h}\) is a \(\cat{D}\)-resolution. Moreover, the natural map \(Q\ra\hm{}{T}{\omega_{h}}{Q\ot\omega_{h}}\) is an isomorphism by Nakayama's lemma again. 
\end{proof}
\begin{ex}\label{ex.extiso}
Assume \(A\) is a Cohen-Macaulay ring with a canonical module \(\omega_{A}\).
Given \(L_{i}\in\hat{\cat{D}}_{A}\) and put \(Q_{i}=\hm{}{A}{\omega_{A}}{L_{i}}\) for \(i=1,2\). If \(I\) is an injective resolution of \(L_{2}\) and \(P\) is a projective resolution of \(Q_{1}\) then both spectral sequences of \(\hm{}{A}{P}{\hm{}{A}{\omega_{A}}{I}}\) collapse at page \(2\) (use \cite[I 4.10.19]{has:00}) to give canonical isomorphisms 
\begin{equation}
\xt{*}{A}{L_{1}}{L_{2}}\cong\xt{*}{A}{Q_{1}}{Q_{2}}
\end{equation}
\end{ex}
\begin{rem}\label{rem.Buch}
In his unpublished manuscript Buchweitz gave a construction of Cohen-Macaulay approximation for finite \(A\)-modules if \(A\) is a not necessarily commutative Gorenstein ring; see \cite{buc:86}. The \(\cat{MCM}\)-approximation and the \(\Df\)-hull in Theorem \ref{thm.flatCMapprox} can be given essentially by the same construction. Let \(N\) be a finite \(T\)-module which is \(S\)-flat where \(h:S\ra T\) is a finite type Cohen-Macaulay map. Let \(P=P(N)\ra N\) be a projective resolution of \(N\) (i.e.\ by finite projective \(T\)-modules). Then \(P^{\vee}=\hm{}{T}{P}{\omega_{h}}\) is a bounded below complex with bounded cohomology because \(\xt{i}{T}{N}{\omega_{h}}=0\) for \(i\) greater than the relative dimension \(d\) of \(h\) by Corollary \ref{cor.xtdef} (\(\Injdim \omega_{T_{s}}=\dim T_{s}\leq d\)). Then we can choose a projective resolution \(f:P(P^{\vee})\ra P^{\vee}\) of \(P^{\vee}\) which is bounded above. Let \(C=C(f)\) be the mapping cone of \(f\). The modules in \(C\) are direct sums of projective modules and modules in \(\cat{D}_{h}\) and the (co)kernels in the acyclic \(C\) are modules in \(\cat{MCM}_{h}\). By Corollary \ref{cor.xtdef} it follows that \(C^{\vee}\) is acyclic too. There is a composition of natural maps \(P\cong P^{\vee\vee}\ra P(P^{\vee})^{\vee}:=G\) which hence is a quasi-isomorphism. But then \(G\) is just the representing complex of \(N\) and the \(\cat{MCM}\)-approximation and the \(\hat{\cat{D}}\)-hull is obtained as in Section \ref{subsec.cplx}. In the case of coherent rings with a cotilting module (a concept introduced by Y.\ Miyashita; see \cite[p.\ 142]{miy:86}) J.-i.\ Miyachi implicitly gave the same construction in \cite[3.2]{miy:96}.

Suppose \(h\) has pure relative dimension and \(N_{s}\) is a non-zero Cohen-Macaulay \(T_{s}\)-module with \(n=\dim T_{s}-\dim N_{s}\) constant for all \(s\in\Spec S\), i.e.\ \(N\) is a family of Cohen-Macaulay modules of codepth \(n\). Put \(N^{\vee}:=\xt{n}{T}{N}{\omega_{h}}\). Then \(P^{\vee}=P(N)^{\vee}\) is quasi-isomorphic to \(\cH^{n}(P^{\vee})=N^{\vee}\) by duality theory and Corollary \ref{cor.xtdef}. So if \((F,d)\ra N^{\vee}\) is a projective resolution of \(N^{\vee}\), the representing complex is given as \(G=F^{\vee}\). If \(\syz{T}{i}N^{\vee}\) denotes the \(i^{\text{th}}\) syzygy module \(\im d_{i}\) and \(d_{i}^{\vee}\) denotes \(\hm{}{T}{d_{i}}{\omega_{h}}\) then the commutative diagram \eqref{eq.pullpush} with short exact sequences is given as
\begin{equation*}
\xymatrix@C-0pt@R-9pt@H-30pt{
\im (d_{n}^{\vee}) \ar[r]\ar@{=}[d] & \hm{}{T}{\syz{T}{n}N^{\vee}}{\omega_{h}} \ar[r]\ar[d]\ar@{}[dr]|{\Box} & N^{\vee\vee} \ar[d]  \\   
\im (d_{n}^{\vee}) \ar[r] & \hm{}{T}{F_{n}}{\omega_{h}} \ar[r]\ar[d] & \coker(d_{n}^{\vee}) \ar[d]  \\
& \hm{}{T}{\syz{T}{n{+}1}N^{\vee}}{\omega_{h}} \ar@{=}[r] & \hm{}{T}{\syz{T}{n{+}1}N^{\vee}}{\omega_{h}}
}
\end{equation*}
with the \(\cat{MCM}_{h}\)-approximation and \(\hat{\cat{D}}_{h}\)-hull of \(N\) given by the upper horizontal and right vertical sequence, respectively, since \(N^{\vee\vee}\cong N\) by duality theory and Corollary \ref{cor.xtdef}. 
Let \(C\) denote the mapping cone of a comparison map \(P(N)\ra G\). The homology of the truncated short exact sequence of complexes \(G\ra C\ra P(N)[-1]\) gives the \(\cat{MCM}_{h}\)-approximation \(\coker d^{G}_{i+2}\ra\coker d^{C}_{i+2}\ra\syz{T}{i}N\) for \(i\geq 0\) (with \(i=0\) being the upper horizontal sequence in the diagram). In the absolute case (with \(N\) Cohen-Macaulay) this latter construction of the MCM approximation of \(\syz{}{i}N\) was given by J.\ Herzog and A.\ Martsinkovsky in \cite[1.1]{her/mar:93}. 
\end{rem}
\subsection{Local cases}\label{subsec.loc}
We formulate local variants of the approximation theorem. 

Fix a field \(k\). Let \(\cat{H}\) denote the category of noetherian, henselian, local rings \(S\) with residue field \(S/\fr{m}_{S}\cong k\) and with local ring homomorphisms. A map \(h:S\ra T\) in \(\cat{H}\) is \emph{algebraic} (or \(T\) is an algebraic \(S\)-algebra) if there is a finite type \(S\)-algebra \(\tilde{T}\) and a maximal ideal \(\fr{m}\) in \(\tilde{T}\) with \(\tilde{T}/\fr{m}\cong k\) such that \(h\) is given by henselisation of \(\tilde{T}\) in \(\fr{m}\). Fix an algebraic \(k\)-algebra \(A\) which is supposed to be Cohen-Macaulay. Objects in the category \(\cat{hCM}\) are algebraic and flat \(S\)-algebras \(T\) with \(T\ot_{S}k\cong A\). A morphism \(h_{1}\ra h_{2}\) is a pair of commuting maps \(f:T_{1}\ra T_{2}\) and \(g:S_{1}\ra S_{2}\) in \(\cat{H}\) as for the finite type case, giving a cocartesian square. Base change exists for the forgetful functor \(\cat{hCM}\ra\cat{H}\) and is given by the henselisation of the tensor product \(T=T_{1}\ot_{S_{1}}S_{2}\) in the maximal ideal \(\fr{m}_{T_{1}}T+\fr{m}_{S_{2}}T\). We denote it by \(T_{1}\hot_{S_{1}}S_{2}\). It follows that \(\cat{hCM}\ra\cat{H}\) is fibred in groupoids. The objects in \(\cat{hCM}\) will be called henselian Cohen-Macaulay (hCM) maps.

If \(\tilde{h}:\tilde{S}\ra\tilde{T}\) is a finite type CM map and \(t\) a \(k\)-point in \(\Spec \tilde{T}\) with image \(s\) in \(\Spec \tilde{S}\) then localisation for the \'etale topology at \(t\) and \(s\) gives a hCM map \(h:S=\tilde{S}^{\text{h}}\ra \tilde{T}^{\text{h}}=T\). Conversely, every hCM map is obtained this way which follows from \cite[18.6.6 and 18.6.10]{EGAIV4} and \cite[15.4.3 and 12.1.1]{EGAIV3}. We will call such an \(\tilde{h}\) a (finite type) representative of \(h\). The dualising module \(\omega_{\tilde{h}}\) induces an \(S\)-flat finite \(T\)-module \(\omega_{h}\) called the dualising module for \(h\). Two representatives of \(h\) factor through a common \'etale neighbourhood contained in \(\cat{CM}\) and since the dualising module is functorial for \(\cat{CM}\) the dualising module \(\omega_{h}\) is functorial too.

Let \(\cat{mod}\) denote the category of pairs \((h:S\ra T,\mc{N})\) with \(h\) in \(\cat{hCM}\) and \(\mc{N}\) a finite \(T\)-module. Morphisms are defined as for the finite type case and the forgetful functor \(\cat{mod}\ra\cat{hCM}\) makes \(\cat{mod}\) fibred in abelian categories over \(\cat{hCM}\). Let \(\modf\) denote the full subcategory of objects \((h,\mc{N})\) in \(\cat{mod}\) with \(\mc{N}\) \(S\)-flat and let \(\cat{MCM}\) denote the full subcategory of objects \((h,\mc{N})\) in \(\modf\) where the closed fibre \(\mc{N}\ot_{S}k\) is a maximal Cohen-Macaulay \(A\)-module. Let \(\cat{D}\) denote the full subcategory of \(\cat{MCM}\) of objects \((h,D)\) with \(D\) in \(\Add\{\omega_{h}\}\). 
All three are FAd subcategories of \(\cat{mod}\) over \(\cat{hCM}\). Any finite \(T\)-module \(\mc{N}\) has finite presentation hence it is induced from a finite module over a representative of \(T\). If \(\mc{N}\) is contained in one of the subcategories the representative can be assumed to belong to the corresponding finite type subcategory ({\sl loc.\ sit.}). Similarly all maps in these fibred categories over \(\cat{H}\) are induced from maps in the corresponding fibred categories of finite type objects.

Let \(\cat{L}\) (\(\cat{C}\)) denote the category of noetherian, (complete) local rings with residue field \(k\) and local ring homomorphisms. Let \(\cat{lCM}\) (\(\cat{cCM}\)) denote the category of local Cohen-Macaulay (lCM) maps (respectively complete Cohen-Macaulay or cCM maps) defined analogously as above with (completion of) \emph{essentially of finite type} replacing algebraic. Similar arguing as above makes \(\cat{lCM}\) (\(\cat{cCM}\)) fibred in groupoids over \(\cat{L}\) (\(\cat{C}\)). The definitions of the module categories apply in the local and the complete case too and we use the same notation in all three cases. Again objects and maps are induced from the finite type case.

Either arguing with representatives or applying the proofs for Theorem \ref{thm.flatCMapprox} and Corollary \ref{cor.flatCMapprox} (with only minor adjustments) we obtain the following.
\begin{cor}\label{cor.locCMapprox}
Let \(\cat{xCM}\) denote either \(\cat{hCM}\)\textup{,} \(\cat{lCM}\) or \(\cat{cCM}\)\textup{.} The pair \((\cat{mod},\cat{MCM})\) of fibred categories over \(\cat{xCM}\) satisfies \textup{BC1-BC2} and the triple of fibre categories \((\cat{mod}_{h},\cat{MCM}_{h},\cat{D}_{h})\) satisfies \textup{AB1-AB4} for all objects \(h\) in \(\cat{xCM}\)\textup{.} 

Moreover\textup{,} the statements \textup{(i)-(v)} in \textup{Theorem \ref{thm.flatCMapprox}} and \textup{(i)-(ii) in Corollary \ref{cor.flatCMapprox}} are valid over \(\cat{xCM}\) too\textup{.} 
\end{cor}
\section{Minimal approximations and semi-continuity of invariants}
We show that the Cohen-Macaulay approximation and the \(\hat{\cat{D}}\)-hull in Corollary \ref{cor.locCMapprox} can be chosen to be minimal. Upper semi-continuous invariants on \(\cat{MCM}_{A}\) or \(\cat{FID}_{A}\) extend to upper semi-continuous invariants on \(\cat{mod}_{A}\). An example is given by the \(\omega_{A}\)-ranks in the representing \(\cat{D}\)-complex.
\begin{lem}\label{lem.Aapprox}
Let \(S\ra T\) be a homomorphism of noetherian rings and \(\fr{a}\) an ideal in \(S\) such that \(I=\fr{a}T\) is contained in the Jacobson radical of \(T\)\textup{.} Let \(M\) and \(N\) be finite \(T\)-modules\textup{.} Let \(T_{n}=T/I^{n+1}\)\textup{,} \(M_{n}=T_{n}\ot M\) and \(N_{n}=T_{n}\ot N\)\textup{.} Suppose there exists a tower of surjections \(\{\phi_{n}:M_{n}\ra N_{n}\}\)\textup{.} Fix any non-negative integer \(n_{0}\)\textup{.} Then there exists a \(T\)-linear surjection \(\psi:M\ra N\) such that \(T_{n_{0}}\ot\psi=\phi_{n_{0}}\)\textup{.} If the \(\phi_{n}\) are isomorphisms and \(N\) is \(S\)-flat then \(\psi\) is an isomorphism\textup{.}
\end{lem}
\begin{proof}
Let \(\hat{T}=\limproj T_{n}\), \(\hat{M}=\limproj M_{n}\) and \(\hat{N}=\limproj N_{n}\). We have 
\begin{equation}
\limproj\hm{}{T_{n}}{M_{n}}{N_{n}}\cong\hm{}{\hat{T}}{\hat{M}}{\hat{N}}\cong\hat{T}\ot\hm{}{T}{M}{N}.
\end{equation}
Hence \(\limproj \phi_{n}=\Sigma r^{(i)}\ot\beta^{(i)}\) with \(r^{(i)}\in \hat{T}\) and \(\beta^{(i)}\in\hm{}{T}{M}{N}\). Let \(r^{(i)}_{n_{0}}\) be the image of \(r^{(i)}\) under \(\hat{T}\ra T_{n_{0}}\) and choose liftings \(t^{(i)}\) in \(T\) of \(r^{(i)}_{n_{0}}\). Put \(\psi=\Sigma t^{(i)}\beta^{(i)}\). Then \(T_{n_{0}}\ot\psi=\phi_{n_{0}}\). Since \(T_{n_{0}}\ot\coker\psi=\coker\phi_{n_{0}}=0\), Nakayama's lemma implies \(\coker\psi=0\). Since \(T_{n_{0}}\ot_{T}N\) equals \(N\ot_{S}S/\fr{a}^{n_{0}+1}\), \(S\)-flatness of \(N\) implies \(T_{n_{0}}\ot\ker\psi=\ker\phi_{n_{0}}\) and if \(\ker\phi_{n_{0}}=0\) then Nakayama's lemma implies \(\ker\psi=0\).
\end{proof}
\begin{prop}\label{prop.minapprox}
Let \(\cat{xCM}\) denote either \(\cat{hCM}\)\textup{,} \(\cat{lCM}\) or \(\cat{cCM}\)\textup{,} let \(h:S\ra T\) be an object in \(\cat{xCM}\) and let \(\xi:\mc{L}\ra\mc{M}\xra{\pi}\mc{N}\) be an \(\cat{MCM}_{h}\)-approximation of \(\mc{N}\)\textup{.} Then the following statements are equivalent\textup{.} 
\begin{enumerate}
\item[(i)] The sequence \(\xi\) is a right minimal \(\cat{MCM}_{h}\)-approximation\textup{.}
\item[(ii)] There are no surjections \(\mc{M}\ra\omega_{h}\) which induces a surjection \(\mc{L}\ra\omega_{h}\)\textup{.}
\item[(iii)] There are no common \(\omega_{h}\)-summand in \(\mc{L}\ra\mc{M}\)\textup{.}
\item[(iv)] The closed fibre \(\xi\ot_{S}k\) is a right minimal \(\cat{MCM}_{A}\)-approximation\textup{.}
\item[(v)] The completion \((\xi\ot_{S}k)\hat{\,}\) of the closed fibre is a right minimal \(\cat{MCM}_{\hat{A}}\)-approximation\textup{.}
\end{enumerate}
The analogous statements \textnormal{(i')-(v')} for a \(\Df_{h}\)-hull \(\xi':\mc{N}\ra\mc{L}'\ra\mc{M}'\) are equivalent\textup{.} In particular\textup{:}
\begin{enumerate}
\item[(i')] The sequence \(\xi'\) is a left minimal \(\Df_{h}\)-hull\textup{.}
\item[(ii')] There are no surjections \(\mc{M}'\ra\omega_{h}\) which induces a surjection \(\mc{L}'\ra\omega_{h}\)\textup{.}
\end{enumerate}
\end{prop}
\begin{proof}  
Suppose there is a surjection \(\mc{M}\ra \omega_{h}\) such that the composition \(\mc{L}\ra \omega_{h}\) is surjective too. Then the kernels of these maps give a new \(\cat{MCM}\)-approximation of \(\mc{N}\) by Corollary \ref{cor.locCMapprox}. Since a surjection \(\mc{L}\ra\omega_{h}\) has to split by Corollary \ref{cor.locCMapprox}, \(\omega_{h}\) is a common summand in \(\mc{L}\ra\mc{M}\) and \(\pi\) cannot be right minimal. 

Let the closed fibre \(\xi\ot_{S}k\) of the sequence \(\xi\) be denoted by \(L\ra M\xra{p} N\).
Assume there is a non-surjective endomorphism \(\theta:\mc{M}\ra\mc{M}\) with \(\pi\theta=\pi\).  Then \(\theta_{0}=\theta\ot_{S}k\) gives a non-surjective endomorphism of \(M\) with \(p\theta_{0}=p\). It follows that the completion \(\hat{L}\ra\hat{M}\ra\hat{N}\) is not a right minimal Cohen-Macaulay approximation of \(\hat{N}\). By \cite[1.12.8]{has:00} there is a common \(\omega_{\hat{A}}\)-summand in \(\hat{L}\ra\hat{M}\). Let \(\phi:\hat{M}\ra\omega_{\hat{A}}\) denote the projection. By Lemma \ref{lem.Aapprox} there exists a surjection \(\psi:M\ra\omega_{A}\). The induced map \(L\ra\omega_{A}\) is surjective too. The map \(\psi\) lifts to a surjection \(\mc{M}\ra\omega_{h}\) (with \(\mc{L}\ra\omega_{h}\) surjective) since the canonical map \(\hm{}{T}{\mc{M}}{\omega_{h}}\ra\hm{}{A}{M}{\omega_{A}}\) is surjective by Corollary \ref{cor.xtdef}. 
The \(\Df\)-case is analogous.
\end{proof}
\begin{cor}\label{cor.minapprox}
Let \(\cat{xCM}\) denote either \(\cat{hCM}\)\textup{,} \(\cat{lCM}\) or \(\cat{cCM}\)\textup{.}
For any object \((h,\mc{N})\) in the fibred category \(\modf\) over \(\cat{xCM}\), \(\mc{N}\) admits a right minimal \(\cat{MCM}\)-approximation and a left minimal \(\Df\)-hull which remain minimal after base change and which in particular are unique up to non-canonical isomorphism\textup{.}
\end{cor}
\begin{proof}
The existence of a right minimal \(\xi\) follows immediately from criterion (iii) in Proposition \ref{prop.minapprox}.  Moreover \(\xi\) is right minimal if and only if the closed fibre \(\xi\ot_{S}k\) is right minimal. Since any base change \(\xi_{1}\) of \(\xi\) has the same closed fibre as \(\xi\), \(\xi_{1}\) is right minimal if \(\xi\) is.
\end{proof}
\begin{rem}
In \cite[2.3]{has/shi:97} M.\ Hashimoto and A.\ Shida gave essentially the analog of Proposition \ref{prop.minapprox} in the absolute case of a Cohen-Macaulay Zariski local ring with a canonical module. They attributed the complete case to Y.\ Yoshino. Note Miyachi's proof in the cotilting semi-perfect case; see \cite[3.4]{miy:96} (cf.\ \cite[1.12.8]{has:00}). In \cite[3.1]{sim/str:02} A.-M.\ Simon and J.\ R.\ Strooker give a short independent proof. The proof of Proposition \ref{prop.minapprox} also works in the Zariski local case in full generality and is different from these (but depends on the complete case).
\end{rem}
Since minimal choices of \(\cat{MCM}_{A}\)-approximations and \(\cat{\hat{D}}_{A}\)-hulls exist and are unique up to isomorphism any invariant defined for MCM modules or for FID modules is extended to all finite \(A\)-modules. Upper semi-continuity of the invariants is also extended as we explain now. First some notation.

Let \(h:S\ra T\) be ring homomorphism and \(\mc{N}\) a finite \(T\)-module. If \(t\in\Spec T\) has image \(s\in\Spec S\), let \(\mc{N}_{\fr{p}_{t}}\) denote the localisation of \(\mc{N}\) at the prime ideal \(t\), let \(h_{\fr{p}_{t}}:S_{\fr{p}_{s}}\ra T_{\fr{p}_{t}}\) denote the ring homomorphism obtained by localising, and put \(\mc{N}(t)=\mc{N}_{\fr{p}_{t}}\ot_{S_{\fr{p}_{s}}}k(s)\) which is a \(T(t)\)-module; indeed \(\mc{N}(t)\cong T(t)\ot_{T_{\fr{p}_{t}}}\mc{N}_{\fr{p}_{t}}\). 
If \(h\) is a finite type Cohen-Macaulay map, \(h_{\fr{p}_{t}}\) is in \(\cat{lCM}\). 

Suppose \(\mu\) is an invariant on \(\cat{MCM}_{A}\) where \(A\) is a Cohen-Macaulay local ring with canonical module. Let \({}_{\cat{MCM}}\mu\) denote the induced invariant on \(\cat{mod}_{A}\) defined by \({}_{\cat{MCM}}\mu(N) = \mu(M)\) where \(L\ra M\ra N\) is the minimal Cohen-Macaulay approximation of \(N\). Similarly \({}_{\cat{FID}}\mu(N)= \mu(L)\) for an invariant \(\mu\) defined on \(\cat{FID}_{A}\). Use the minimal hull \(N\ra L'\ra M'\) to define \({}_{\cat{FID}}\mu'\) and \({}_{\cat{MCM}}\mu'\).

The following theorem is a major application of what we have done so far. 
\begin{thm}\label{thm.semicont}
Let \(\mu\) be an additive non-negative numerical invariant defined for maximal Cohen-Macaulay modules or for finite modules of finite injective dimension on a Cohen-Macaulay local ring with canonical module\textup{.} Assume \(\mu\) is upper semi-continuous for finite type flat families \((h:S\ra T,\mc{M})\) in \(\cat{MCM}\) \textup{(}or in \(\hat{\cat{D}}\)\textup{).} Then the induced invariants \({}_{\cat{MCM}}\mu\) and \({}_{\cat{MCM}}\mu'\) \textup{(}or \({}_{\cat{FID}}\mu\) and \({}_{\cat{FID}}\mu'\)\textup{)} are upper semi-continuous in finite type flat families \((h,\mc{N})\) in \(\modf\)\textup{.} 
\end{thm}
\begin{proof}
Given \((h:S\ra T,\mc{N})\) in \(\modf\) and \(t\in\Spec T\). By Corollary \ref{cor.minapprox} there exist open affines \(U=\Spec S_{1}\sbeq \Spec S\), \(V=\Spec T_{1}\sbeq \Spec T\) with \(t\in V\), \(h_{1}:S_{1}\ra T_{1}\) induced from \(h\), and an \(\cat{MCM}_{h_{1}}\)-approximation \(\xi: \mc{L}\ra\mc{M}\ra\mc{N}_{\vert V}\) such that the localisation \(\xi_{\fr{p}_{t}}\) is minimal. By Corollary \ref{cor.minapprox} \(\xi(t)\) is minimal too. Put \(n=\mu(\mc{M}(t))\). Since \(\mu\) is upper semi-continuous there is an open \(V_{n}\sbeq V\) containing \(t\) such that \(\mu(\mc{M}(t'))\leq n\) for all \(t'\in V_{n}\). If \(L\ra M\ra \mc{N}(t')\) is the minimal MCM approximation of \(\mc{N}(t')\), \(M\) is a direct summand of \(\mc{M}(t')\) by Proposition \ref{prop.minapprox}, and hence \({}_{\cat{MCM}}\mu(\mc{N}(t'))=\mu(M)\leq\mu(\mc{M}(t'))\leq n\).
\end{proof}
\begin{ex}\label{ex.semicont1}
The Betti numbers \(\beta_{i}(M)=\dim\tor{A}{i}{M}{A/\fr{m}_{A}}\) are well-known upper semi-continuous invariants of finite modules over local rings. By Theorem \ref{thm.semicont} the induced invariants \({}_{\cat{MCM}}\beta_{i}\), \({}_{\cat{MCM}}\beta_{i}'\), \({}_{\cat{FID}}\beta_{i}\) and \({}_{\cat{FID}}\beta_{i}'\) are upper semi-continuous too.
\end{ex}
We now consider some invariants defined in terms of Cohen-Macaulay approximation.
If \(h:S\ra T\) is in one of the categories of local Cohen-Macaulay maps, a map \(\partial:D\ra D'\) of objects in \(\cat{D}_{h}\) is said to be \emph{minimal} if \(k\ot_{T}\partial=0\). Any module \(D\) in \(\cat{D}_{h}\) is isomorphic to some \(\omega_{h}^{\oplus n}\) and \(\nd{}{T}{\omega_{h}}\cong T\). Hence if \(\partial\) is not minimal then there is a surjection \(D'\ra\omega_{h}\) inducing a surjection \(D\ra\omega_{h}\). By Corollary \ref{cor.locCMapprox} the \(\omega_{h}\) splits off from \(\partial\). Hence any \(\cat{D}_{h}\)-complex is homotopy equivalent to one with all differentials being minimal, which is called a minimal \(\cat{D}\)-complex. 

For any module \(\mc{N}\) in \(\modf_{h}\) over \(\cat{xCM}\) we choose a minimal \(\cat{MCM}\)-approximation \(\mc{L}\ra\mc{M}\ra\mc{N}\) and a minimal \(\Df\)-hull \(\mc{N}\ra\mc{L}'\ra\mc{M}'\) which exist by Corollary \ref{cor.minapprox}. Spliced with a
minimal \(\cat{D}\)-resolution of \(\mc{L}\) and a minimal \(\cat{D}\)-coresolution of \(\mc{M}'\) we obtain complexes \({}^{-}C^{*}(\mc{N})\), \({}^{+}C^{*}(\mc{N})\) and \(D^{*}(\mc{N})\), as defined in Section \ref{subsec.cplx}, where no differential has any \(\omega_{h}\)-summand. We call such choices of these complexes for minimal. They are unique:
\begin{lem}\label{lem.uniqueC}
Suppose \(h\) is in \(\cat{hCM}\)\textup{,} \(\cat{lCM}\) or \(\cat{cCM}\) and \((h,\mc{N})\) is in \(\modf_{h}\)\textup{.} Then 
minimal choices of \({}^{-}C^{*}(\mc{N})\), \({}^{+}C^{*}(\mc{N})\) and \(D^{*}(\mc{N})\) exist and are unique up to non-canonical isomorphisms\textup{.}
\end{lem}
\begin{proof}
Minimal choices \({}^{+}C^{*}_{1}\) and \({}^{+}C^{*}_{2}\) of coresolutions for \(\mc{N}\) are by Lemma \ref{lem.res} homotopic through chain maps \(\alpha\) and \(\beta\) starting with an isomorphism \(\mc{L}_{1}'\cong\mc{L}_{2}'\). If \(\rho_{i}\) are homotopies with \(\beta\alpha-\id=\partial\rho_{1}+\rho_{1}\partial\) and \(\alpha\beta-\id=\partial\rho_{2}+\rho_{2}\partial\) then tensoring down by \(k\ot_{T}-\) makes the right hand side of these identities equal to zero by the minimality of the complexes. Hence \(\beta\alpha\) and \(\alpha\beta\) are surjective endomorphisms, i.e.\ isomorphisms. The same argument applies to \({}^{-}C^{*}\) and \(D^{*}\).
\end{proof}
For each module \(\mc{N}\) in \(\modf_{h}\) we fix a minimal \(\cat{D}\)-complex \((D^{*}(\mc{N}),\partial^{*})\) representing \(\mc{N}\). Let \(\mc{L}'=\coker \partial^{-1}\) and \(\mc{M}=\ker \partial^{0}\).
Put \(\syz{\omega_{h}}{i}\mc{L}':=\coker\{\partial^{-i{-}1}:D^{{-}i{-}1}\ra D^{-i}\}\) for \(i\geq 0\) and \(\syz{\omega_{h}}{-i}\mc{M}:=\ker\{\partial^{i}:D^{i}\ra D^{i{+}1}\}\) for \(i\geq 0\). For any finite \(T\)-module \(\mc{N}\) let the \emph{\(\omega_{h}\)-rank} of \(\mc{N}\), denoted \(\rk{\omega_{h}}{\mc{N}}\), be the largest number \(n\) with \(\omega_{h}^{{\oplus}n}{\oplus}\mc{N}'\cong \mc{N}\) for some \(T\)-module \(\mc{N}'\). Since \(\nd{}{T}{\omega_{h}}\cong T\) is a local ring, this is a well-behaved invariant; cf.\ \cite[Section 1.1]{sim/str:02}. 
\begin{defn}\label{defn.c}
Suppose \(h\) is in \(\cat{hCM}\)\textup{,} \(\cat{lCM}\) or \(\cat{cCM}\) and \(\mc{N}\) is in \(\modf_{h}\)\textup{.} Define the numbers\textup{:}
\begin{enumerate}
\item[(i)] \(d^{i}_{T}(\mc{N}):=\rk{\omega_{h}}{D^{i}(\mc{N})}\) for all \(i\)
\item[(ii)] \(\nu^{T}_{i}(\mc{N}):=\rk{\omega_{h}}{\syz{\omega_{h}}{i}\mc{L}'}\) for \(i\geq 0\)
\item[(iii)] \(\gamma_{T}(\mc{N}):=\rk{\omega_{h}}{\mc{M}}\)
\end{enumerate}
\end{defn}
The definition gives well-defined invariants of \(\mc{N}\) by Lemma \ref{lem.uniqueC}. In particular we see that \(\nu^{T}_{0}(\mc{N})\) equals \(\rk{\omega_{h}}{\mc{L}'}\). We also notice that \(\rk{\omega_{h}}{\syz{\omega_{h}}{-i}\mc{M}}=0\) for all \(i>0\) by Proposition \ref{prop.minapprox}. 

The same notation is used for the absolute counterparts of these invariants. 
\begin{rem}
In the absolute case with \(A\) a Gorenstein local ring, the \(\gamma_{A}\)-invariant is called Auslander's \(\delta\) invariant and has been studied in particular by S.\ Ding. One has that \(\gamma_{A}(A/\fr{m}_{A}^{n})\leq 1\) with equality for \(n\gg0\). The smallest number \(n\) with \(\gamma_{A}(A/\fr{m}_{A}^{n})=1\) is called the \emph{index} of \(A\) and Ding has given results and conjectures concerning this invariant; cf.\ \cite{din:92}. See also \cite{has/shi:97}.

Let \((A,\fr{m},k)\) be a noetherian local ring and \(N\) a finite \(A\)-module. Simon and Strooker introduced the \emph{reduced Bass numbers} 
\begin{equation}
\nu_{A}^{i}(N)=\dim_{k}\im\{\xt{i}{A}{k}{N}\ra\cH^{i}_{\fr{m}}(N)\}
\end{equation}
and in the case \(A\) has a canonical module they showed in \cite[2.6 and 3.10]{sim/str:02} that 
\begin{equation}
\nu^{A}_{0}(N)=\nu_{A}^{\dim A}(N)\quad\text{and}\quad \nu^{A}_{i}(N)=\nu_{A}^{\dim A{-}i}(L')\,\,\,\text{for } i\geq 0
\end{equation}
where \(L'\) is the minimal \(\cat{D}_{A}\)-hull of \(N\). The Canonical Element Conjecture and the Monomial Conjecture are equivalent to \(\nu_{A}^{\dim A}(\fr{b})\neq 0\) (or equivalently \(\gamma_{A}(A/\fr{b})=0\)) for certain ideals \(\fr{b}\) in any Gorenstein local ring \(A\); see \cite[6.4 and 6.6]{sim/str:02}.  
\end{rem}
Let \(\cat{P}\) and \(\Pf\) denote the FAds of finite projective modules, respectively modules of finite projective dimension over \(\cat{xCM}\) defined as for the finite type case.
\begin{lem}\label{lem.fri} With notation as above\textup{:}
\begin{enumerate}
\item[(i)] The functor \(F=\hm{}{-}{\omega_{-}}{-}\) gives exact equivalences of categories fibred in additive categories \(\cat{D}\simeq \cat{P}\) and \(\Df\simeq\Pf\) over \(\cat{xCM}\)\textup{.}
\item[(ii)] In particular \(d^{-i}(\mc{N})=\beta_{i}(\hm{}{T}{\omega_{h}}{\mc{L}'})\) for all \(i\geq 0\)\textup{.}
\item[(iii)] \(\hm{}{T}{D^{*\geq 0}(\mc{N})}{\omega_{h}}\) gives a minimal free resolution of \(\hm{}{T}{\mc{M}}{\omega_{h}}\)\textup{.} In particular \(d^{i}(\mc{N})=\beta_{i}(\hm{}{T}{\mc{M}}{\omega_{h}})\) for all \(i\geq 0\)\textup{.}
\end{enumerate}
\end{lem}
\begin{proof}
(i) and (ii) are the local variants of Lemma \ref{lem.proj}. Breaking \(\mc{M}\hra D^{*\geq 0}(\mc{N})\) into short exact sequences, (iii) follows from Corollary \ref{cor.locCMapprox}.
\end{proof}

\begin{cor}\label{cor.semicont} 
Let \(h:S\ra T\) be a finite type Cohen-Macaulay map and suppose \(\mc{N}\) is a \(T\)-module in \(\modf_{h}\)\textup{.} Then \(d^{i}(\mc{N}(t))\)  are upper semi-continuous functions in \(t\in\Spec T\) for all \(i\)\textup{.} 
\end{cor}
\begin{proof}
This follows from Theorem \ref{thm.semicont} and Lemma \ref{lem.fri}.
\end{proof}
\begin{rem}\label{rem.semicontnot}
The invariants \(\nu_{0}\) and \(\gamma\) are not semi-continuous either way, see Example \ref{ex.semicont}. Moreover, let \(\mc{L}\) be in \(\hat{\cat{D}}_{h}\) with \(L=\mc{L}\ot_{S}k\).  If \(\rho_{0}:\omega_{A}\ra L\) is a direct summand, then \(\rho_{0}\) lifts to a map \(\rho:\omega_{h}\ra\mc{L}\), but no lifting \(\rho\) of \(\rho_{0}\) has to split, even if \(A\) is a regular ring. 
\end{rem}
\begin{rem}\label{rem.semicont}
One can also define functions of the base. E.g.\ if \(\phi:\Spec T\ra \Spec S\) denotes the map induced from \(h\) and \(\mu(\mc{N}(t))\) is an upper semi-continuous function of \(t\in \Spec T\), then \(\phi_{*}\mu(\mc{N})\) defined by
\begin{equation}
\phi_{*}\mu(\mc{N})(s)=\sup_{t\in\phi^{-1}(s)}\!\!\mu(\mc{N}(t))
\end{equation}
is an upper semi-continuous function in \(s\in \Spec S\) since \(\phi\) is an open map.
\end{rem}
\section{The fundamental module of a Cohen-Macaulay map}\label{sec.fund}
\begin{ex}\label{ex.Fmod}
Let \((A,\fr{m}_{A},k)\) be a Cohen-Macaulay local ring with canonical module \(\omega_{A}\). Let \((G_{*},d_{*})\ra k\) be a minimal \(A\)-free resolution of \(k\) and put \(\syz{}{i}=\syz{A}{i}k=\coker d_{i+1}\). Suppose \(\dim A=d\geq 2\). There are connecting isomorphisms \(\xt{1}{A}{\syz{}{d{-}1}}{\omega_{A}}\cong\xt{2}{A}{\syz{}{d{-}2}}{\omega_{A}}\cong\dots\cong\xt{d}{A}{k}{\omega_{A}}\) which is isomorphic to \(k\) by duality theory. To \(1\in k\) there is hence a non-split short exact sequence 
\begin{equation}\label{eq.Fmod}
\theta:\,\,0\ra\omega_{A}\lra E_{A}\xra{\,\,\pi\,\,} \syz{A}{d{-}1}k\ra 0
\end{equation}
with \(E_{A}\) uniquely defined up to non-canonical isomorphism. We call \(E_{A}\) for the \emph{fundamental module} of \(A\). We claim that \(E_{A}\) is a maximal Cohen-Macaulay module which implies that \eqref{eq.Fmod} is the minimal MCM approximation of \(\syz{A}{d{-}1}k\). If we apply \(\hm{}{A}{-}{\omega_{A}}\) to \eqref{eq.Fmod} we obtain the exact sequence
\begin{equation}\label{eq.Fmod2}
\begin{aligned}
0\ra\, &\hm{}{A}{\syz{A}{d{-}1}k}{\omega_{A}}\lra\hm{}{A}{E_{A}}{\omega_{A}}\lra\nd{}{A}{\omega_{A}}\xra{\,\,\partial\,\,}\\
&\xt{1}{A}{\syz{A}{d{-}1}k}{\omega_{A}}\lra\xt{1}{A}{E_{A}}{\omega_{A}}\ra 0
\end{aligned}
\end{equation}
We have \(\partial(\id)=\theta\) so \(\partial\) is surjective and \(\xt{1}{A}{E_{A}}{\omega_{A}}=0\). By duality theory (e.g.\ \cite[3.5.11]{bru/her:98}) this excludes the possibility \(\depth E_{A}=d-1\) and we conclude that \(E_{A}\) is a maximal Cohen-Macaulay module. If \(N\) is a Cohen-Macaulay module of codimension \(c\) we denote \(\xt{c}{A}{N}{\omega_{A}}\) by \(N^{\vee}\). Since \(\nd{}{A}{\omega_{A}}\cong A\) we get from \eqref{eq.Fmod2} a short exact sequence:
\begin{equation}\label{eq.Fmod3}
0\ra \hm{}{A}{\syz{A}{d{-}1}k}{\omega_{A}}\lra E_{A}^{\vee}\lra\fr{m}_{A}\ra 0
\end{equation}
Since \(\xt{i}{A}{k}{\omega_{A}}=0\) for \(i\neq d\), \(0\ra G_{0}^{\vee}\ra\dots G_{d-2}^{\vee}\ra\hm{}{A}{\syz{A}{d{-}1}k}{\omega_{A}}\ra 0\) is a finite \(\omega_{A}\)-resolution and so \eqref{eq.Fmod3} gives the minimal MCM approximation of the maximal ideal. Auslander introduced the fundamental module in the case \(d=2\); see \cite{aus:86}.
\end{ex}

We can make a relative version of the fundamental module in much the same way. Let \({}^{(2)\!}\Delta:\cat{CM}\ra\cat{CM}\) be the morphism of fibred categories over \(\cat{NR}\) defined by taking the CM map \(h:S\ra T\) to the composition \(h^{(2)}\) of \(h\) with \(\iota=1\ot\id_{T}:T\ra T\ot_{S}T\) and taking a morphism \((g,f):h_{1}\ra h_{2}\) to the composition of two cocartesian squares \((g,f^{\ot2})\) as follows:
\begin{equation}
\xymatrix@-0pt@C+6pt@R-4pt@H-0pt{
S_{1}\ar[r]^{h_{1}}\ar[d]_{g} & T_{1}\ar[d]^{f}\ar@{}[dr]|-{\mapsto} & S_{1}\ar[r]^{h_{1}}\ar[d]_{g} & T_{1}\ar[d]^{f}\ar[r]^(.4){\iota} & T_{1}\ot_{S_{1}}T_{1}\ar[d]^{f^{\ot2}} \\
S_{2}\ar[r]^{h_{2}} & T_{2} & S_{2}\ar[r]^{h_{2}} & T_{2}\ar[r]^(.4){\iota} & T_{2}\ot_{S_{2}} T_{2}
}
\end{equation}
There is also a functor \(\Delta:\cat{CM}\ra\cat{CM}\) defined by mapping \((g,f)\) to the rightmost cocartesian square \((f,f^{\ot2})\), but it doesn't commute with the forgetful functor \(\cat{CM}\ra\cat{NR}\).
Let \({}^{d}\cat{CM}\) denote the full subcategory of CM maps of pure relative dimension \(d\). Then \({}^{d}\cat{CM}\) is a fibred subcategory of \(\cat{CM}\) over \(\cat{NR}\) and \({}^{(2)\!}\Delta\) and \(\Delta\) restricts to a morphism \({}^{(2)\!}\Delta:{}^{d}\cat{CM}\ra{}^{2d}\cat{CM}\) over \(\cat{NR}\) and a functor \(\Delta:{}^{d}\cat{CM}\ra{}^{d}\cat{CM}\).

Let \(h:S\ra T\) be a finite type CM map of pure relative dimension \(d\geq 2\). Consider \(P\) in \(\cat{P}_{h}\) (see Lemma \ref{lem.proj}) as a \(T^{\ot2}\)-module by pullback along the multiplication map \(\mu:T^{\ot2}\ra T\). By Corollary \ref{cor.xtdef} \(E=\xt{d}{T^{\ot2}}{P}{\omega_{\iota}}\) is flat and finite as \(T\)-module, i.e.\ \(T\)-projective. Let \(P^{*}\) denote \(\hm{}{T}{P}{T}\). By Corollary \ref{cor.xtdef} 
\begin{equation}
\nd{}{T}{E} \cong \xt{d}{T^{\ot 2}}{P}{\omega_{\iota}}\ot E^{*}\cong \xt{d}{T^{\ot2}}{P}{\omega_{\iota}\ot E^{*}}\,.
\end{equation}
Combined with the connecting isomorphisms the identity in \(\nd{}{T}{E}\) corresponds to a canonical extension of \(T^{\ot 2}\)-modules:
\begin{equation}\label{eq.Fxt}
0\ra \omega_{\iota}\ot_{T}\xt{d}{T^{\ot2}}{P}{\omega_{\iota}}^{*}\lra E_{h}(P)\lra \syz{T^{\ot2}}{d{-}1}P\ra 0\,.
\end{equation}
Let \({}^{d}\cat{P}\) and \({}^{d}\cat{MCM}\) denote the restriction of \(\cat{P}\) and \(\cat{MCM}\) to fibred categories over \({}^{d}\cat{CM}\).
\begin{prop}\label{prop.Fmod}
Let \(d\geq 2\)\textup{.} The association \((h,P)\mapsto E_{h}(P)\) in \eqref{eq.Fxt} induces 
\begin{enumerate}
\item[(i)] a functor \(E:{}^{d}\cat{P}\ra {}^{d}\cat{MCM}/{}^{d}\cat{P}\) which preserves cocartesian maps and lifts the functor \(\Delta:{}^{d}\cat{CM}\ra{}^{d}\cat{CM}\)\textup{,} and

\item[(ii)] a morphism \({}^{(2)\!}E:{}^{d}\cat{P}\ra {}^{2d}\cat{MCM}/{}^{2d}\cat{P}\) of fibred categories over \(\cat{NR}\) which lifts \({}^{(2)\!}\Delta:{}^{d}\cat{CM}\ra{}^{2d}\cat{CM}\)\textup{.}
\end{enumerate}
\end{prop}
\begin{proof}
As an extension of \(T\)-flat modules, \(E_{h}(P)\) is \(T\)-flat. Applying \(\hm{}{T^{\ot2}}{-}{\omega_{\iota}}\) to \eqref{eq.Fxt} with \(E=\xt{d}{T^{\ot2}}{P}{\omega_{\iota}}\) and \(\syz{}{d{-}1}=\syz{T^{\ot2}}{d{-}1}P\) gives an exact sequence
\begin{equation}\label{eq.Fmod4}
\begin{aligned}
0\ra\, &\hm{}{T^{\ot2}}{\syz{}{d{-}1}}{\omega_{\iota}}\ra\hm{}{T^{\ot2}}{E_{h}(P)}{\omega_{\iota}}\ra\nd{}{T^{\ot2}}{\omega_{\iota}}\ot_{T}E\xra{\partial}\\
&\xt{1}{T^{\ot2}}{\syz{}{d{-}1}}{\omega_{\iota}}\ra\xt{1}{T^{\ot2}}{E_{h}(P)}{\omega_{\iota}}\ra 0
\end{aligned}
\end{equation}
by Corollary \ref{cor.xtdef} and duality theory. In particular there is a canonical isomorphism \(\hm{}{T^{\ot2}}{\omega_{\iota}\ot_{T}E^{*}}{\omega_{\iota}}\cong\nd{}{T^{\ot2}}{\omega_{\iota}}\ot_{T}E\). We have that \(\nd{}{T^{\ot2}}{\omega_{\iota}}\) is canonically isomorphic to \(T^{\ot2}\) and \(\partial(t\ot\xi)=\mu(t)\syz{T^{\ot2}}{d{-}1}(\xi)\in \xt{1}{T^{\ot2}}{\syz{}{d{-}1}}{\omega_{\iota}}\) where \(\syz{T^{\ot2}}{d{-}1}\) is the composition of the connecting isomorphisms. So \(\partial\) is surjective and \(\xt{1}{T^{\ot2}}{E_{h}(P)}{\omega_{\iota}}=0\) by Corollary \ref{cor.xtdef}. It follows that all fibres of \(E_{h}(P)\) are MCM modules and so \(E_{h}(P)\) is in \(\cat{MCM}_{\iota}\) and \(\eqref{eq.Fxt}\) is an \(\cat{MCM}_{\iota}\)-approximation of \(\syz{T^{\ot2}}{d{-}1}P\). Let \(\mc{I}_{h}\) denote the kernel of \(\mu:T^{\ot2}\ra T\) and \((-)^{\vee}=\hm{}{T^{\ot2}}{-}{\omega_{\iota}}\). From \eqref{eq.Fmod4} we get another \(\cat{MCM}_{\iota}\)-approximation
\begin{equation}\label{eq.Fmod5}
0\ra \hm{}{T^{\ot2}}{\syz{T^{\ot 2}}{d{-}1}P}{\omega_{\iota}}\lra E_{h}(P)^{\vee}\lra\mc{I}_{h}\ot_{T}\xt{d}{T^{\ot2}}{P}{\omega_{\iota}}\ra 0\,.
\end{equation}
Dualising \eqref{eq.Fmod5} induces \eqref{eq.Fxt} since \(E_{h}(P)\cong E_{h}(P)^{\vee\vee}\) and \(\hm{}{T^{\ot2}}{\mc{I}_{h}}{\omega_{\iota}}\cong\omega_{\iota}\).
By Theorem \ref{thm.flatCMapprox} the image of \(E_{h}(P)^{\vee}\) in \(\cat{MCM}_{\iota}/\cat{D}_{\iota}\) is functorial in the \(T^{\ot2}\)-module \(\mc{I}_{h}\ot E\) which again is contravariantly functorial in \(P\). Since \((-)^{\vee}\) induces an equivalence 
\begin{equation}
\xymatrix@-0pt@C+6pt@R-4pt@H-0pt{
\vee:\cat{MCM}_{\iota}/\cat{P}_{\iota}\ar@{<->}[r]^(0.47){\simeq} & \cat{MCM}_{\iota}^{\text{op}}/\cat{D}_{\iota}^{\text{op}}:\vee
}
\end{equation}
we conclude that \(E_{h}(P)\) is functorial in \(\cat{MCM}/\cat{P}\) by our functorial choice of extension. This gives (i) and (ii).
\end{proof}
\begin{cor}\label{cor.Fmod}
For any Cohen-Macaulay map \(h:S\ra T\) of pure relative dimension \(d\geq 2\) there is a finite \(T^{\ot2}\)-module \(E_{h}=E_{h}(T)\) which is faithfully flat along \(\iota:T\ra T^{\ot2}\) with all fibres being maximal Cohen-Macaulay modules\textup{.} The association \(h\mapsto E_{h}\) defines a functor \({}^{d}\cat{CM}\ra {}^{d}\cat{MCM}/{}^{d}\cat{P}\) lifting \(\Delta:{}^{d}\cat{CM}\ra{}^{d}\cat{CM}\)\textup{.} In particular \(E_{h}\) gives \(\cat{MCM}_{\iota}\)-approximations
\begin{equation*}
0\ra\omega_{\iota}\lra  E_{h}\lra \syz{T^{\ot2}}{d{-}1}T\ra 0\quad\text{and}
\end{equation*}
\begin{equation*}
0\ra \hm{}{T^{\ot2}}{\syz{T^{\ot2}}{d{-}1}T}{\omega_{\iota}}\lra E_{h}^{\vee}\lra \mc{I}_{h}\ra 0
\end{equation*}
where \(\mc{I}_{h}\) is the kernel of the multiplication map \(T^{\ot2}\ra T\)\textup{.}
\end{cor}
\begin{proof}
This follows from Proposition \ref{prop.Fmod} and \eqref{eq.Fmod5} once we have proved the natural isomorphism \(\xt{d}{T^{\ot2}}{T}{\omega_{\iota}}\cong T\).
Choose a surjection of \(S\)-algebras \(P\ra T\) with \(P=S[t_{1},\dots,t_{N}]\). Recall that \(\omega_{\iota}\) can be given as \(\xt{N-d}{P\ot T}{T^{\ot2}}{\omega_{P\ot T/T}}\) where \(\omega_{P\ot T/T}=\bigwedge^{N}\varOmega_{P\ot T/T}\). There is a change of rings spectral sequence
\begin{equation}
\cE^{p,q}_{2}=\xt{q}{T^{\ot2}}{T}{\xt{p}{P\ot T}{T^{\ot2}}{\omega_{P\ot T/T}}}\,\Ra\,\xt{p+q}{P\ot T}{T}{\omega_{P\ot T/T}}
\end{equation}
which by Corollary \ref{cor.xtdef} and duality theory collapses to the canonical isomorphism 
\begin{equation}
\xt{d}{T^{\ot2}}{T}{\xt{N-d}{P\ot T}{T^{\ot2}}{\omega_{P\ot T/T}}}\cong\xt{N}{P\ot T}{T}{\omega_{P\ot T/T}}\,.
\end{equation}
By \cite[3.5.6]{con:00} \(\xt{N}{P\ot T}{T}{\omega_{P\ot T/T}}\) is canonically isomorphic to \(\omega_{T/T}=T\) as \(T^{\ot2}\)-module.
\end{proof}
We call the module \(E_{h}\) given in Corollary \ref{cor.Fmod} for the \emph{fundamental module} of the Cohen-Macaulay map \(h\).
\begin{ex}\label{ex.semicont}
Let \(k\) be an algebraically closed field and \(A\) a finite type \(k\)-algebra which is Cohen-Macaulay of pure dimension \(2\). Then the fundamental module \(E=E_{h}\) of \(h:k\ra A\) is the maximal Cohen-Macaulay approximation of \(I=\ker\{A^{\ot 2}\ra A\}\) in \(\modf_{\iota}\);
\begin{equation}\label{eq.semicont}
0\ra\omega_{h}\ot A\lra E\lra I\ra 0
\end{equation}
where \(\iota=1\ot\id:A\ra A^{\ot 2}\) and \(\omega_{h}\cong\omega_{A}\).
Let \(t\) in \(\Spec A^{\ot 2}\) be a \(k\)-point, and \(t_{i}\in \Spec A\) be the image of \(t\) by the \(i^{\text{th}}\) projection. Let \(A_{i}\) denote \(A\) localised at \(t_{i}\). Let \(\fr{m}_{i}\) be the maximal ideal in \(A_{i}\). Localising gives a local Cohen-Macaulay map \(\iota_{\fr{p}_{t}}:A_{2}\ra (A^{\ot2})_{\fr{p}_{t}}\) and a module \(E_{\fr{p}_{t}}\) in \(\cat{MCM}_{\iota_{\fr{p}_{t}}}\). Let \(E(t)\) denote base change of \(E_{\fr{p}_{t}}\) to \(k(t_{2})\). If \(t_{1}=t_{2}\) then \(I(t)\cong \fr{p}_{1}\) and \(E(t)\) equals the fundamental module \(E_{A_{1}}\) of \eqref{eq.Fmod}. If \(t_{1}=t_{2}\) is singular, then \(\rk{\omega}{E(t)}=0\) while if \(t_{1}=t_{2}\) is regular then \(E(t)\cong A_{1}^{{\oplus}2}\). If \(t_{1}\neq t_{2}\) then \(I(t)\cong A_{1}\cong E_{A_{1}}\) and \(E(t)\cong A_{1}^{\oplus 2}\). This shows that \(\gamma(I)(t)\) is \emph{not} upper semi-continuous as the \(d^{i}\)-invariants are. 

In particular, if \(A\) equals \(k[x,y,z]/(x^{n{+}1}{-}yz)\) with a \(2\)-dimensional rational double point at \(\fr{m}_{0}=(x,y,z)\), similar considerations give the following table of invariants (note that \(\nu_{1}=d^{-1}\)):
\begin{center}
\setlength{\extrarowheight}{2,5pt}
\begin{tabularx}{280pt}[t]{| X | l | c | c | c | c |}
\hline
\(\iota:A\ra A^{\ot 2}\) & \(\gamma\) & \(\nu_{1}\) & \(d^{0}\) & \(\nu_{0}\) & \(I(t)\)
\\[0.5ex]\hline\hline
\(t_{1}=t_{2}=0\) singular point       & \(0\) & \(1\) & \(4\) & \(1\) & \(\fr{m}_{0}A_{\fr{m}_{0}}\) \\ \hline
\(t_{1}=t_{2}\) non-singular point  & \(2\) & \(1\) & \(2\) & \(0\) & \(\fr{m}_{1}A_{1}\) \\ \hline
\(t_{1}\neq t_{2}\)                        & \(1\) & \(0\) & \(1\) & \(1\) & \(A_{1}\) \\ \hline
\end{tabularx}
\end{center}
\end{ex}
\section{Deformation functors and cohomology}\label{sec.def}
We extend the Cohen-Macaulay approximation over henselian local rings to deformations and obtain maps between the associated deformation functors. We also introduce the appropriate Andr{\'e}-Quillen cohomology and links the various cohomologies in a fundamental long-exact sequence.

Fix an object \(\xi=(h:S\ra T,\mc{N})\) in \(\modf\) over \(\cat{H}\). A \emph{deformation} of \(\xi\) is a cocartesian morphism \(\alpha_{1}:\xi_{1}\ra \xi\) in \(\modf\). A \emph{map of deformations} \(\alpha_{1}\ra\alpha_{2}\) is a cocartesian morphism \(\phi:\xi_{1}\ra \xi_{2}\) in \(\modf\) such that \(\alpha_{2}\phi = \alpha_{1}\). Deformations and maps of deformations are objects and arrows in the comma category \(\cat{Def}_{\xi}:=\modf_{\text{coca}}/\xi\) which is fibred in groupoids over the comma category \(\cat{H}/S\); see Lemma \ref{lem.gpoid} and the proceeding comments. Let the \emph{deformation functor} \(\df{}{\xi}:\cat{H}/S\ra \Sets\) be the functor corresponding to the associated groupoid of sets \(\ol{\cat{Def}}_{\xi}\). The comma category \(\cat{Def}_{h}:=\cat{hCM}/h\) of deformations of \(h:S\ra T\) is also fibred in groupoids over \(\cat{H}/S\) and we have an obvious factorisation \(\cat{Def}_{\xi}\ra\cat{hCM}/h\ra \cat{H}/S\) which makes \(\cat{Def}_{\xi}\) fibred in groupoids over \(\cat{hCM}/h\). To ease readability (and by abuse of notation) we put \(\df{}{(T,\mc{N})}(S_{1})=\df{}{\xi}(S_{1}{\ra} S)\) and \(\df{}{T}(S_{1})=\df{}{h}(S_{1}{\ra} S)\). We also write a \emph{deformation of \((T,\mc{N})\)} meaning a deformation of \(\xi\) and likewise in similar situations.

For each object \(\xi_{v}=(h_{v}\co S_{v}\ra T_{v},\mc{N}_{v})\) in \(\modf\) with \(h_{v}\) in \(\cat{hCM}\) we fix a minimal \(\cat{MCM}\)-approximation and a minimal \(\Df\)-hull
\begin{equation}\label{eq.ny} 
\pi_{v}\co 0\ra\mc{L}_{v}\ra \mc{M}_{v}\xra{\pi_{v}}\mc{N}_{v}\ra0\text{ and }
\iota_{v}\co 0\ra\mc{N}_{v}\xra{\iota_{v}}\mc{L}'_{v}\ra \mc{M}_{v}'\ra0
\end{equation}
which exist by Corollaries \ref{cor.locCMapprox} and \ref{cor.minapprox}. For each deformation \(\xi_{v}\ra \xi\) we choose extensions to commutative diagrams of deformations
\begin{equation}\label{eq.right2}
\xymatrix@C-0pt@R-8pt@H-30pt{
\mc{L}_{v}\ar[r]\ar@{.>}[d]_(0.4){\lambda} & \mc{M}_{v}\ar[r]^{\pi_{v}}\ar@{.>}[d]_(0.4){\mu} & \mc{N}_{v}\ar[d]_(0.4){\nu} \\   
\mc{L}\ar[r] & \mc{M} \ar[r]^{\pi} & \mc{N}
}
\qquad\text{and}\qquad
\xymatrix@C-0pt@R-8pt@H-30pt{
\mc{N}_{v}\ar[r]^{\iota_{v}}\ar[d]_(0.4){\nu} & \mc{L}_{v}'\ar[r]\ar@{.>}[d]_(0.4){\lambda'} & \mc{M}_{v}'\ar@{.>}[d]_(0.4){\mu'} \\   
\mc{N}\ar[r]^{\iota} & \mc{L}' \ar[r] & \mc{M}'
}
\end{equation}
as follows: By Corollary \ref{cor.minapprox} a base change of \(\pi_{v}\) by \(S_{v}\ra S\) gives a minimal \(\cat{MCM}\)-approximation \(\mc{M}_{v}\ot_{S_{v}}S\ra\mc{N}_{v}\ot_{S_{v}}S\xra{\simeq} \mc{N}\). By minimality it is isomorphic to \(\pi\). Choose an isomorphism. Let \(\mu\) be the composition \(\mc{M}_{v}\ra \mc{M}_{v}\ot_{S_{v}}S\xra{\simeq}\mc{M}\). It is cocartesian. Do similarly for the \(\Df\)-hull. Let these choices be fixed.
\begin{defn}\label{defn.defmap}
There are four maps 
\begin{equation*}
\sigma_{X}\co \df{}{(h,\mc{N})}\lra\df{}{(h,X)}\text{ of functors }\cat{H}/S\lra\Sets
\end{equation*}
where \(X\) can be \(\mc{M},\mc{L},\mc{L}'\) and \(\mc{M}'\) given by \([(h_{v}\ra h, \nu)]\mapsto [(h_{v}\ra h, x)]\) for \(x\) equal to \(\mu,\lambda,\lambda'\) and \(\mu'\) in \eqref{eq.right2} respectively.
\end{defn}
The following lemma implies that these maps are well defined and independent of choices.
\begin{lem}\label{lem.defCMapprox2}
Given two deformations 
\begin{equation*}
((f_{j},g_{j}),\nu_{j})\co (h_{v_{j}}\co S_{v_{j}}\ra T_{v_{j}},\mc{N}_{v_{j}})\ra(h_{{j}}\co S_{{j}}\ra T_{{j}},\mc{N}_{{j}}),\quad j=1,2,
\end{equation*}
in \(\modf\) over \(\cat{hCM}\)\textup{.} Consider the minimal \(\cat{MCM}\)-approximations \(\pi_{v_{j}}\) and \(\pi_{j}\)\textup{(}respectively the \(\Df\)-hulls \(\iota_{v_{j}}\) and \(\iota_{j}\)\textup{)} defined in \eqref{eq.ny} and the corresponding maps of short exact sequences \(\pi_{v_{j}}\ra\pi_{j}\) \textup{(}respectively \(\iota_{v_{j}}\ra \iota_{j}\)\textup{)} which extends \(\nu_{j}\) defined in \eqref{eq.right2}. Given 
\begin{itemize}
\item a map \((f,g)\co h_{1}\ra h_{2}\) and an \(f\)-linear map \(\alpha\co \mc{N}_{1}\ra\mc{N}_{2}\)\textup{,} 
\item maps of short exact sequences \(\pi_{1}\ra\pi_{2}\) and \(\iota_{1}\ra\iota_{2}\) which extends \(\alpha\)\textup{,}
\item a map \((\tilde{f},\tilde{g})\co h_{v_{1}}\ra h_{v_{2}}\) lifting \((f,g)\)\textup{,} and
\item an \(\tilde{f}\)-linear map \(\tilde{\alpha}\co \mc{N}_{v_{1}}\ra\mc{N}_{v_{2}}\) which lifts \(\alpha\)\textup{.}
\end{itemize}
In particular the following two diagrams of solid arrows are commutative\textup{:}
\begin{equation*}\label{eq.C}
\xymatrix@C-20pt@R-8pt@H-30pt{
& \mc{L}_{v_{2}}\ar[rr]\ar[dd]^(0.3){\lambda_{2}}|!{[dl];[dr]}\hole && \mc{M}_{v_{2}}\ar[rr]^(0.40){\pi_{v_{2}}}\ar[dd]^(0.3){\mu_{2}}|!{[dl];[dr]}\hole
&& \mc{N}_{v_{2}}\ar[dd]^(0.3){\nu_{2}} \\ 
\mc{L}_{v_{1}}\ar@{.>}[ur]\ar[rr]^(0.2){}\ar[dd]^(0.3){\lambda_{1}} && \mc{M}_{v_{1}}\ar@{.>}[ur]^(0.4){\gamma}\ar[rr]^(0.4){\pi_{v_{1}}}\ar[dd]^(0.3){\mu_{1}} 
&& \mc{N}_{v_{1}}\ar[ur]^(0.45){\tilde\alpha}\ar[dd]^(0.3){\nu_{1}} &  \\  
& \mc{L}_{2}\ar[rr]|!{[ur];[dr]}\hole && \mc{M}_{2} \ar[rr]^(0.35){\pi_{2}}|!{[ur];[dr]}\hole && \mc{N}_{2} \\
\mc{L}_{1}\ar[ur]\ar[rr] && \mc{M}_{1}\ar[ur]^(0.45){\beta} \ar[rr]^(0.45){\pi_{1}} && \mc{N}_{1}\ar[ur]_(0.47){\alpha} & 
}
\quad
\xymatrix@C-20pt@R-10pt@H-30pt{
& \mc{N}_{v_{2}}\ar[rr]^(0.40){\iota_{v_{2}}}\ar[dd]^(0.3){\nu_{2}}|!{[dl];[dr]}\hole && \mc{L}'_{v_{2}}\ar[rr]\ar[dd]^(0.3){\lambda'_{2}}|!{[dl];[dr]}\hole 
&& \mc{M}'_{v_{2}}\ar[dd]^(0.3){\mu'_{2}} \\ 
\mc{N}_{v_{1}}\ar[ur]^(0.45){\tilde\alpha}\ar[rr]^(0.35){\iota_{v_{1}}}\ar[dd]^(0.3){\nu_{1}} && \mc{L}'_{v_{1}}\ar@{.>}[ur]^(0.4){\gamma'}\ar[rr]^(0.2){}\ar[dd]^(0.3){\lambda'_{1}} 
&& \mc{M}'_{v_{1}}\ar@{.>}[ur]\ar[dd]^(0.3){\mu'_{1}} &  \\  
& \mc{N}_{2}\ar[rr]^(0.35){\iota_{2}}|!{[ur];[dr]}\hole && \mc{L}'_{2} \ar[rr]^(0.2){}|!{[ur];[dr]}\hole && \mc{M}'_{2} \\
\mc{N}_{1}\ar[ur]^(0.53){\alpha}\ar[rr]^(0.45){\iota_{1}} && \mc{L}'_{1}\ar[ur]_(0.57){\beta'} \ar[rr] && \mc{M}'_{1}\ar[ur] & 
}
\end{equation*}
Then there exist \(f\)-linear maps \(\gamma\co \mc{M}_{v_{1}}\ra\mc{M}_{v_{2}}\) and \(\gamma'\co \mc{L}'_{v_{1}}\ra\mc{L}'_{v_{2}}\) such that the induced diagrams are commutative\textup{.} If \(\tilde\alpha\) is cocartesian\textup{,} so are \(\gamma\) and \(\gamma'\)\textup{.}
\end{lem}
\begin{proof}
Consider the \(\cat{MCM}\)-approximation case. By applying base changes to the front diagram, we can reduce the problem to the case \(h_{v_{1}}\ra h_{1}\) equals \(h_{v_{2}}\ra h_{2}\). Then, by Corollary \ref{cor.locCMapprox}, there is a lifting \(\gamma_{1}\co \mc{M}_{v_{1}}\ra\mc{M}_{v_{2}}\) of \(\tilde\alpha\). We would like to adjust \(\gamma_{1}\) so that it lifts \(\beta\) too. We have that \(\mu_{2}\gamma_{1}-\beta\mu_{1}\) factors through \(\mc{L}_{2}\) by a map \(\tau\co \mc{M}_{v_{1}}\ra\mc{L}_{2}\). It induces a unique map \(\bar\tau\co \mc{M}_{1}\ra\mc{L}_{2}\) since \(\mu_{1}\) is cocartesian. If \(\mc{D}_{*}\thr\mc{L}_{v_{2}}\) is a finite \(\cat{D}\)-resolution, then base change gives a finite \(\cat{D}\)-resolution \(\mc{D}_{*}\ot_{S_{v_{2}}}S_{2}\thr \mc{L}_{2}\) and \(\bar\tau\) factors through a \(\bar\sigma\co \mc{M}_{1}\ra\mc{D}_{0}\ot_{S_{v_{2}}}S_{2}\) by Corollary \ref{cor.locCMapprox}. By Corollary \ref{cor.xtdef} there is a \(\sigma\co \mc{M}_{v_{1}}\ra\mc{D}_{0}\) lifting \(\bar\sigma\). Subtracting the induced map \(\mc{M}_{v_{1}}\ra\mc{M}_{v_{2}}\) from \(\gamma_{1}\) gives our desired \(\gamma\). If \(\tilde\alpha\) is an isomorphism so is \(\gamma\) by minimality of the approximations \(\pi_{v_{j}}\).
The argument for the \(\Df\)-case is similar.
\end{proof}
\begin{rem}\label{rem.defmap}
There are maps of fibred categories inducing the maps \(\sigma_{X}\) in Definition \ref{defn.defmap}. Two maps \(\alpha,\beta\co  (h_{1},\mc{N}_{1})\ra (h_{2},\mc{N}_{2})\) in \(\cat{Def}_{(h,\mc{N})}\) are stably equivalent if \(h_{1}=h_{2}\) and \(\alpha{-}\beta\) factors through an object in \(\cat{D}\). Let \(\ul{\cat{Def}}_{(h,\mc{N})}\) denote the resulting quotient category which is fibred over \(\cat{hCM}/h\) and over \(\cat{H}/S\). Then Lemma \ref{lem.defCMapprox2} implies that there are well defined maps of categories fibred in groupoids \(\bs{\sigma}_{X}\co \ul{\cat{Def}}_{(h,\mc{N})}\ra\ul{\cat{Def}}_{(h,X)}\) for \(X\) equal to \(\mc{M}\), \(\mc{L}\), \(\mc{L}'\) and \(\mc{M}'\). The associated map of functors is \(\sigma_{X}\). 
Stably isomorphic modules will in general have different deformation functors. E.g.\ let \(N=A\oplus\omega_{A}\). If \(A\) is not Gorenstein, then one can have \(\xt{1}{A}{N}{N}\neq 0\). But in the stable category \(N\) is isomorphic to \(A\) which is infinitesimally rigid. 
\end{rem}
We have the following reformulation. To \((h\co S\ra T,\mc{N})\) in \(\modf\) consider \(\vG=T{\oplus}\mc{N}\) as a graded \(S\)-algebra with \(T\) in degree \(0\) and \(\mc{N}\) in degree \(1\). A deformation of graded algebras \({\vG}_{v}\ra{\vG}\) over \(S_{v}\ra S\) in \(\cat{H}/S\) is equivalent to a deformation \((T_{v},\mc{N}_{v})\) of \((T,\mc{N})\). More generally, given a homogeneous morphism of \(\BB{Z}\)-graded rings \(S\ra T\) and a graded \(T\)-module \(M\), there are \emph{Andr{\'e}-Quillen cohomology groups of graded algebras} \(\gH{0}^{i}(S,T,M)=\cH^{i}{\hm{{\text{gr}}}{T}{L^{\text{gr}}_{T/S}}{M}}\). Here \(L^{\text{gr}}_{T/S}\) is the graded cotangent complex defined as \(\Omega_{P/S}\ot_{P}T\) where \(P=P_{S}(T)\) is a graded simplicial degree-wise free \(S\)-algebra resolution of \(T\) and \(\Omega_{P/S}\) denotes the K{\"a}hler differentials; see \cite[IV]{ill:71} for more details (in a more general situation). See also \cite{kle:79}. 
\begin{defn}\label{def.grobs}
Given graded ring homomorphisms \(h\co S\ra T\) and \(p\co S_{v}\ra S\). Assume \(p\) is surjective. A \emph{lifting} of \(h\) (`of \(T\)') along \(p\) (`to \(S_{v}\)') is a commutative diagram of graded ring homomorphisms
\begin{equation*}\label{eq.lift}
\xymatrix@C-0pt@R-12pt@H-30pt{
T & T_{v}\ar[l]_{q} \\
S\ar[u]^{h} & S_{v}\ar[l]_{p}\ar[u]_{h_{v}}
}
\quad
\textnormal{with \(q\ot S\co T_{v}\ot_{S_{v}}S\cong T\) and \(\tor{S_{v}}{1}{T_{v}}{S}=0\).}
\end{equation*}
Two liftings \(T_{v}\) and \(T'_{v}\) of \(T\) to \(S_{v}\) are equivalent if there is a graded \(S_{v}\)-algebra isomorphism \(T_{v}\cong T'_{v}\) commuting with \(q\) and \(q'\).
\end{defn}
There is an obstruction theory for liftings of graded algebras in terms of graded Andr{\'e}-Quillen cohomology groups. 
\begin{prop}[\cite{ill:71,kle:79}]\label{prop.grobs}
Given graded ring homomorphisms \(S\ra T\) and \(p\co S_{v}\ra S\) with \(p\) surjective and \(I^{2}=0\) for \(I=\ker p\)\textup{.} 
\begin{enumerate}
\item[(i)] There exists an element \(\ob(p,T)\in\gH{0}^{2}(S,T,T\ot_{S} I)\) which is natural in \(p\) such that \(\ob(p,T)=0\) if and only if there exists a lifting of \(T\) to \(S_{v}\)\textup{.}
\item[(ii)] If \(\ob(p,T)=0\) then the set of equivalence classes of liftings of \(T\) to \(S_{v}\) is a torsor for \(\gH{0}^{1}(S,T,T\ot_{S} I)\) which is natural in \(p\)\textup{.}
\end{enumerate}
\end{prop}
The element \(\ob(p,T)\) is called the \emph{obstruction class} of \((p,T)\). If the rings and modules are concentrated in degree \(0\) this equals the ungraded case and the cohomology groups equals the ungraded Andr{\'e}-Quillen cohomology \(\cH^{*}(S,T,T\ot_{S}I)\).
\begin{defn}\label{def.modulobs}
Given a lifting diagram of ungraded ring homomorphisms as in Definition \ref{def.grobs} and a \(T\)-module \(N\). Then a \emph{lifting} of \(N\) to \(T_{v}\) is a \(T_{v}\)-module \(N_{v}\) with \(\tor{S_{v}}{1}{N_{v}}{S}=0\) and a map \(N_{v}\ra N\) inducing an isomorphism \(N_{v}\ot S\cong N\). Two liftings \(N_{v}\) and \(N_{v}'\) of \(N\) to \(T_{v}\) are \emph{equivalent} if there is an isomorphism of \(T_{v}\)-modules \(N_{v}\cong N_{v}'\) commuting with the maps to \(N\). 
\end{defn}
There is also an obstruction theory for liftings of modules in terms of \(\Ext\) groups.
\begin{prop}[{\cite[IV 3.1.5]{ill:71}}]\label{prop.obsmodule}
Given \textup{(}ungraded\textup{)} ring homomorphisms as in \textup{Definition \ref{def.grobs}} with \(I^{2}=0\) and a \(T\)-module \(N\)\textup{.}
\begin{enumerate}
\item[(i)] There exists an element \(\ob(q,N)\in\xt{2}{T}{N}{N\ot_{S}I}\) which is natural in \(q\) such that \(\ob(q,N)=0\) if and only if there exists a lifting of \(N\) to \(T_{v}\)\textup{.}
\item[(ii)] If \(\ob(q,N)=0\) then the set of equivalence classes of liftings of \(N\) to \(T_{v}\) is a torsor for \(\xt{1}{T}{N}{N\ot_{S}I}\) which is natural in \(q\)\textup{.} 
\end{enumerate}
\end{prop}
The element \(\ob(q,N)\) is called the \emph{obstruction class} of \((q,N)\). The following long exact sequence connects these three cohomologies.
\begin{prop}\label{prop.lang}
Suppose \(T\) is an \textup{(}ungraded\textup{)} \(S\)-algebra and \(N\) is a \(T\)-module\textup{.} Let \({\vG}=T{\oplus} N\) be the graded \(S\)-algebra with \(T\) in degree \(0\) and \(N\) in degree \(1\)\textup{.} Let \(J\) be a graded \({\vG}\)-module with graded components \(J=J_{0}{\oplus}J_{1}\) of degree \(0\) and \(1\)\textup{.} Then there is a natural long exact sequence\textup{:}
\begin{align*}
0\ra\, & \hm{}{T}{N}{J_{1}}\lra \gDer{0}_{S}({\vG},J)\lra \Der_{S}(T,J_{0})\xra{\,\,\partial\,\,}\dots \\
\ra\, &\xt{n}{T}{N}{J_{1}}\lra \gH{0}^{n}(S,{\vG},J)\lra \cH^{n}(S,T,J_{0})\xra{\,\,\partial\,\,} \xt{n+1}{T}{N}{J_{1}}\ra\dots
\end{align*}
\end{prop}
\begin{proof}
To the graded ring homomorphisms \(S\ra T\ra {\vG}\) there is a distinguished triangle of transitivity
\begin{equation}\label{eq.triangle}
L^{\text{gr}}_{\vG/T/S}\co \,\,L^{\text{gr}}_{T/S}\ot_{T}{\vG}\lra L^{\text{gr}}_{{\vG}/S}\lra L^{\text{gr}}_{{\vG}/T}\lra L^{\text{gr}}_{T/S}\ot_{T}{\vG}[1]
\end{equation}
in the graded derived category of \({\vG}\); see \cite[IV 2.3.4]{ill:71}. The (standard) simplicial resolution \(P_{T}({\vG})\) equals \(T\) in degree \(0\), the (standard) \(T\)-free resolution \(F_{T}(N)\) of the \(T\)-module \(N\) in degree \(1\), and higher degree terms; see \cite[IV 1.3.2.1]{ill:71}. It follows that \(\hm{{\text{gr}}}{{\vG}}{L^{\text{gr}}_{{\vG}/T}}{J}=\hm{}{T}{F_{T}(N)}{J_{1}}\). Since \(L^{\text{gr}}_{T/S}=L_{T/S}\) is consentrated in degree \(0\), \(\hm{{\text{gr}}}{{\vG}}{L^{\text{gr}}_{T/S}\ot_{T}{\vG}}{J}= \hm{}{T}{L_{T/S}}{J_{0}}\).
\end{proof}
\begin{lem}\label{lem.cohmap}
Let \(S\ra T\) be a finite type Cohen-Macaulay map and let \(\mc{N}\) be an \(S\)-flat finite \(T\)-module\textup{.} Let \(0\ra\mc{L}\ra\mc{M}\xra{\pi}\mc{N}\ra0\) and \(0\ra\mc{N}\xra{\iota}\mc{L}'\ra\mc{M}'\ra0\) be an \(\cat{MCM}\)-approximation and a \(\Df\)-hull of \(\mc{N}\) over \(h\)\textup{.} Let \(X_{i}\) denote \(\mc{N}\)\textup{,} \(\mc{M}\) and \(\mc{L}'\) for \(i=0,1,2\) respectively\textup{,} and put \(\vG_{i}=T{\oplus}X_{i}\)\textup{.} Let \(I\) be any \(S\)-module\textup{.}
Then there are natural maps of short exact sequences of complexes \textup{(}see \eqref{eq.triangle}\textup{)}
\begin{equation*}
\hm{\textnormal{gr}}{\vG_{0}}{L^{\textnormal{gr}}_{\vG_{0}/T/S}}{\vG_{0}\ot I}\xra{\pi^{*}}\hm{\textnormal{gr}}{\vG_{1}}{L^{\textnormal{gr}}_{\vG_{1}/T/S}}{\vG_{0}\ot I}\xla{\pi_{*}}\hm{\textnormal{gr}}{\vG_{1}}{L^{\textnormal{gr}}_{\vG_{1}/T/S}}{\vG_{1}\ot I}
\end{equation*}
and
\begin{equation*}
\hm{\textnormal{gr}}{\vG_{0}}{L^{\textnormal{gr}}_{\vG_{0}/T/S}}{\vG_{0}\ot I}\xra{\iota_{*}}\hm{\textnormal{gr}}{\vG_{0}}{L^{\textnormal{gr}}_{\vG_{0}/T/S}}{\vG_{2}\ot I}\xla{\iota^{*}}\hm{\textnormal{gr}}{\vG_{2}}{L^{\textnormal{gr}}_{\vG_{2}/T/S}}{\vG_{2}\ot I}\textup{.}
\end{equation*}
The induced maps of graded Andr{\'e}-Quillen cohomology 
\begin{align*}
\gH{0}^{n}(\pi_{*}) & \co\gH{0}^{n}(S,\vG_{1},\vG_{1}\ot I)\lra\gH{0}^{n}(S,\vG_{1},\vG_{0}\ot I)\quad\text{and} \\
\gH{0}^{n}(\iota^{*}) & \co\gH{0}^{n}(S,\vG_{2},\vG_{2}\ot I)\lra\gH{0}^{n}(S,\vG_{0},\vG_{2}\ot I)
\end{align*}
are isomorphisms for \(n>0\) and surjections for \(n=0\)\textup{.}
\end{lem}
\begin{proof}
There is a natural map \(L^{\text{gr}}_{\vG_{1}/T/S}\ot_{\vG_{1}}\vG_{0}\ra L^{\text{gr}}_{\vG_{0}/T/S}\) of short exact sequences of complexes (cf.\ \cite[II 2.1.1.6]{ill:71}) induced by the graded \(T\)-algebra map \(\vG_{1}\ra\vG_{0}\). It induces \(\pi^{*}\). Covariance along \(\vG_{1}\ot I\ra\vG_{0}\ot I\) gives \(\pi_{*}\). In each (cohomological) degree the rightmost terms are naturally identified with \(\hm{}{T}{L_{T/S}}{T\ot I}\) as in the proof of Proposition \ref{prop.lang}. By Theorem \ref{thm.flatCMapprox} and Corollary \ref{cor.xtdef} one has \(\xt{n}{T}{\mc{M}}{\mc{L}\ot I}=0\) for \(n>0\) and the \(\gH{0}^{n}(\pi_{*})\)-statement follows. The other case is similar.
\end{proof}
By Lemma \ref{lem.cohmap} (and Theorem \ref{thm.flatCMapprox}) we get induced natural maps for \(n>0\) 
\begin{equation}\label{eq.sigma}
\sigma^{n}_{j}(I)\co \gH{0}^{n}(S,\vG_{0},\vG_{0}\ot_{S} I)\lra\gH{0}^{n}(S,\vG_{j},\vG_{j}\ot_{S} I)\,\, \text{for}\,\, j=1,\,2 \quad\textnormal{and}
\end{equation}
\begin{equation}\label{eq.tau}
\tau^{n}_{j}(I)\co \xt{n}{T}{X_{0}}{X_{0}\ot_{S}I}\lra \xt{n}{T}{X_{j}}{X_{j}\ot_{S}I}\,\, \text{for}\,\, j=1,\,2.
\end{equation}
\begin{ex}\label{ex.lang} By elementary diagram chase Lemma \ref{lem.cohmap} gives the following:
\begin{itemize}
\item[(i)] If \(\pi^{*}\co \xt{n}{T}{\mc{N}}{\mc{N}\ot I}\ra\xt{n}{T}{\mc{M}}{\mc{N}\ot I}\) is an isomorphism for \(n=1\) and injective for \(n=2\) then \(\sigma^{1}_{1}(I)\) is an isomorphism and \(\sigma^{2}_{1}(I)\) is injective.
\item[(ii)] If \(\iota_{*}\co \xt{n}{T}{\mc{N}}{\mc{N}\ot I}\ra\xt{n}{T}{\mc{N}}{\mc{L}'\ot I}\) is an isomorphism for \(n=1\) and injective for \(n=2\) then \(\sigma^{1}_{2}(I)\) is an isomorphism and \(\sigma^{2}_{2}(I)\) is injective.
\end{itemize}
\end{ex}
\section{Maps of deformation functors induced by\\ Cohen-Macaulay approximation}
In order to use Artin's Approximation Theorem \cite{art:69} as extended by D.\ Popescu \cite{pop:86,pop:90} we fix an excellent ring \(\vL\) (see \cite[7.8.2]{EGAIV2}). We consider the category of henselian local \(\vL\)-algebras in \(\cat{H}\), denoted \(\QH\). Fibred categories \(\cat{hCM}\) and \(\modf\) over \(\QH\) and \(\cat{hCM}/h \) and \(\cat{Def}_{\xi}\) over \({\QH}/S\) are defined essentially as in Section \ref{sec.def}. Our previous constructions and results are valid in this context as well. In particular deformation functors \(\Def_{(T,\mc{N})}\co \QH/S\ra\Sets\) are defined and the \(\cat{MCM}\)-approximation and \(\Df\)-hull induce maps of deformation functors as in Definition \ref{defn.defmap}
\begin{defn}\label{def.smooth}
Let \(\QA/k\) denote the subcategory of artin rings in \(\QH/k\).
Let \(F\) and \(G\) be set-valued functors on \(\QH/k\) (or \(\QA/k\)) with \(\#F(k)=1=\#G(k)\). A map \(\phi\co F\ra G\)
is \emph{smooth} (formally smooth) if the natural map of sets \(f_{\phi}\co F(S_{v})\ra F(S)\times_{G(S)}G(S_{v})\) is surjective for all surjections \(\pi\co S_{v}\ra S\) in \(\QH/k\) (respectively \(\QA/k\)).
An element \(\nu\in F(R)\) is \emph{versal} if the induced map \(\hm{}{\QH/k}{R}{-}\ra F\) is smooth and \(R\) is algebraic as \(\vL\)-algebra. If the map is bijective then \(\nu\) is universal. An element \(\nu\in F(R)\) (or a formal element, i.e.\ a tower \(\{\nu_{n}\}\in \limproj F(R/\fr{m}_{R}^{n+1})\)) is \emph{formally versal} if the induced map \(\hm{}{\QH/k}{R}{-}\ra F\) of functors restricted to \(\QA/k\) is formally smooth.
See \cite{art:74}. 
\end{defn}
\begin{thm}\label{thm.defgrade}
Let \(k\) be a field, \(A\) a Cohen-Macaulay local algebraic \(k\)-algebra and \(N\) a finite \(A\)-module\textup{.} Let \(0\ra N\ra L'\ra M'\ra0\) be the minimal \(\hat{\cat{D}}_{A}\)-hull and \(0\ra L\ra M\ra N\ra0\) the minimal \(\cat{MCM}_{A}\)-approximation of \(N\)\textup{.} Consider the map 
\begin{equation*}
\sigma_{L'}\co \df{}{(A,N)}\lra\df{}{(A,L')}\text{ of functors }\QH/k\lra\Sets
\end{equation*}
as in \textup{Definition \ref{defn.defmap}.}
\begin{enumerate}
\item[(i)] If\, \(\hm{}{A}{N}{M'}=0\) then \(\sigma_{L'}\) is injective\textup{.}

\item[(ii)] If\, \(\xt{1}{A}{N}{M'}=0\) then \(\sigma_{L'}\) is formally smooth\textup{.}

\item[(iii)] If\, \(\xt{1}{A}{N}{M'}=0\) and \(\df{}{(A,N)}\) has a versal element then \(\sigma_{L'}\) is smooth\textup{.}
\end{enumerate}
The analogous statements hold for \(\sigma_{L}\co \df{}{(A,N)}\ra\df{}{(A,L)}\) with \(\xt{1}{A}{N}{M}=0\) in \textup{(i)} and \(\xt{2}{A}{N}{M}=0\) in \textup{(ii)} and \textup{(iii).}
\end{thm}
\begin{ex}\label{ex.defgrade}
Note that \(\grade N\geq 1\) implies condition (i) and \(\grade N\geq 2\) implies both condition (i) and (ii).
\end{ex}
\begin{proof}
(i) Given \(S\) in \(\QH/k\) and deformations \(({}^{i\!}h\co S\ra {}^{i\!}T,{}^{i\!}\mc{N})\) of \((A,N)\) to \(S\) for \(i=1,2\). Assume that the images \(({}^{i\!}h,{}^{i\!}\mc{L}')\) under \(\sigma_{L'}\) are isomorphic to \((h\co S\ra T,\mc{L}')\). Pullback of all these modules along the isomorphisms of \(h\) with \({}^{i\!}h\) induce deformations over \(h\). We show that the \({}^{i\!}\mc{N}\) are isomorphic as deformations over \(h\) which implies that \(\sigma_{L'}\) is injective. Let \(S_{n}=S/\fr{m}_{S}^{n+1}\), \(T_{n}=T\ot S_{n}\) etc.. We construct a tower of isomorphisms \(\{\phi_{n}\co {}^{1\!}\mc{N}_{n}\cong{}^{2\!}\mc{N}_{n}\}\) and conclude by Lemma \ref{lem.Aapprox} that the deformations are isomorphic. The case \(n=0\) is trivial. Given \(\phi_{n}\) and use it to identify the \({}^{i\!}\mc{N}_{n}\) and denote them by \(\mc{N}_{n}\). Let \(I=\ker\{S_{n+1}\ra S_{n}\}\). By Proposition \ref{prop.obsmodule} there exists an element \(\theta\) in \(\xt{1}{T_{n}}{\mc{N}_{n}}{\mc{N}_{n}\ot I}\) giving the `difference' of the two deformations of \(\mc{N}_{n}\). But \(\mc{N}_{n}\ot I\cong N\ot I\) and by the edge map isomorphism of \eqref{eq.ss} we get \(\xt{i}{T_{n}}{\mc{N}_{n}}{\mc{N}_{n}\ot I}\cong\xt{i}{A}{N}{N}\ot I\) for all \(i\). If \(i>0\) then \(\xt{i}{A}{L'}{L'}\cong\xt{i}{A}{N}{L'}\) and \(\xt{1}{A}{N}{N}\ra\xt{i}{A}{N}{L'}\) is injective by assumption. 
The obtained injective map \(p\co \xt{1}{A}{N}{N}\ot I\ra \xt{1}{A}{L'}{L'}\ot I\) induces a map of the torsor actions in Proposition \ref{prop.obsmodule} on the liftings of \(\mc{N}_{n}\) and of \(\mc{L}'_{n}\) to \(T_{n+1}\). Since the \(\Df_{h_{n+1}}\)-hulls of the \({}^{i\!}\mc{N}_{n+1}\) are isomorphic as deformations, \(p\) maps \(\theta\) to \(0\) and so \(\theta=0\) and by Proposition \ref{prop.obsmodule} the \({}^{i\!}\mc{N}_{n+1}\) are isomorphic by an isomorphism \(\phi_{n+1}\) compatible with \(\phi_{n}\).

(ii) Let \(S\ra \bar{S}\) in \(\QA/k\) be surjective with kernel \(I\), \(\xi=(h\co S\ra T,\mc{L}')\) a deformation of \((A,L')\) to \(S\) and let \(\bar{\xi}=(\bar{h}\co \bar{S}\ra\bar{T},\bar{\mc{L}}')\) denote the base change of \(\xi\) to \(\bar{S}\). Suppose there is a deformation \((h^{(1)}\co \bar{S}\ra T^{(1)},\bar{\mc{N}})\) of \((A,N)\) which \(\sigma_{L'}\) maps to \(\bar{\xi}\). As above we can assume that \(h^{(1)}=\bar{h}\). By induction on the length of \(S\) we can assume that \(I\cdot\fr{m}_{S}=0\). We find that \(\ob(T\ra \bar{T},\bar{\mc{N}})\) in Proposition \ref{prop.obsmodule} maps to \(\ob(T\ra \bar{T},\bar{\mc{L}}')\) under \(\xt{2}{A}{N}{N}\ot I\ra \xt{2}{A}{L'}{L'}\ot I\) which by the assumption is injective. Since \(\mc{L}'\) lifts \(\bar{\mc{L}'}\) to \(T\) we have \(\ob(T\ra \bar{T},\bar{\mc{L}}')=0\). By Proposition \ref{prop.obsmodule} there exists a lifting \({}^{1\!}\mc{N}\) of \(\bar{\mc{N}}\) to \(T\). If \(\sigma_{L'}({}^{1\!}\mc{N})={}^{1\!}\mc{L}'\) the difference of \({}^{1\!}\mc{L}'\) and \(\mc{L}'\) gives by Proposition \ref{prop.obsmodule} a \(\theta\in \xt{1}{A}{L'}{L'}\ot I\). By assumption \(\xt{1}{A}{N}{N}\ot I\) maps surjectively to \(\xt{1}{A}{L'}{L'}\ot I\) and a lifting of \(\theta\) perturbs \({}^{1\!}\mc{N}\) to a lifting \(\mc{N}\) of \(\bar{\mc{N}}\) with \(\sigma_{L'}(\mc{N})=\mc{L}'\).

(iii) Let \(S\ra \bar{S}\) in \(\QH/k\) be surjective with kernel \(J\), \(\xi=(h\co S\ra T,\mc{L}')\) a deformation of \((A,L')\) to \(S\) and let \(\bar{\xi}=(\bar{h}\co \bar{S}\ra\bar{T},\bar{\mc{L}}')\) denote the base change of \(\xi\) to \(\bar{S}\). Suppose there is a deformation \((h^{(1)}\co \bar{S}\ra T^{(1)},\bar{\mc{N}})\) of \((A,N)\) which \(\sigma_{L'}\) maps to \(\bar{\xi}\). Again we can assume that \(h^{(1)}=\bar{h}\). We will find a deformation \(\mc{N}\) lifting \(\bar{\mc{N}}\) such that \(\sigma_{L'}(h,\mc{N})=(h,\mc{L}')\).

Any \(S\) in \(\QH/k\) is a direct limit of a filtering system of algebraic \(\vL\)-algebras in \(\QH/k\). Since \(\df{}{(A,L')}\) is locally of finite presentation (\(A\) is algebraic and \(L'\) has finite presentation) it is sufficient to prove the lifting property for \(S\) algebraic. Since \(\vL\) is excellent, so is \(S\) by \cite[7.8.3]{EGAIV2} and \cite[18.7.6]{EGAIV4}.

Put \(S_{n}=S/\fr{m}_{S}^{n}J\), \(\mc{L}'_{n}=\mc{L}'\ot_{S}S_{n}\), \(T_{n}=T\ot_{S}S_{n}\) and so on. We proceed by induction on \(n\) to construct a tower \(\{\mc{N}_{n}\}\) of deformations of \(\bar{\mc{N}}\) inducing \(\{\mc{L}'_{n}\}\). Each step is done as in (ii).
If \(({}^{v}T,{}^{v\!}\mc{N})\in\df{}{(A,N)}(R)\) is a versal element, there is a corresponding tower of maps \(\{f_{n}\co R\ra S_{n}\}\) such that \(({}^{v}T,{}^{v\!}\mc{N})\) induces \(\{(T_{n},\mc{N}_{n})\}\). We obtain the algebra map \(f\co R\ra {}^{*\!}S:=\limproj S_{n}\) which induces a deformation \(({}^{*}T,{}^{*\!}\mc{N})\) of \((\bar{T},\bar{\mc{N}})\) to \({}^{*\!}S\). Since \(\limproj {}^{*}T_{n}\cong\limproj T_{n}\) completion in the maximal ideals gives an isomorphism \({}^{*}\hat{T}\cong\hat{T}\). By \cite[2.6]{art:69}, \cite[1.3]{pop:86} and \cite{pop:90} there is an isomorphism \({}^{*}T\cong T\hot_{S}{}^{*\!}S\) whereby \({}^{*}T\) is identified with \(T\hot_{S}{}^{*\!}S\). The tower of isomorphisms \(\{\sigma_{L'}({}^{*\!}\mc{N}_{n})\cong\mc{L}'_{n}\}\) implies by Lemma \ref{lem.Aapprox} that there is an isomorphism of deformations \(\sigma_{L'}({}^{*\!}\mc{N})\cong {}^{*\!}\mc{L}'\) above \(\sigma_{L'}(\bar{\mc{N}})\cong \bar{\mc{L}}'\) where \({}^{*\!}\mc{L}'={}^{*}T\ot_{T}\mc{L}'\). 

To apply Artin's Approximation Theorem we define a functor of \(S\)-algebras \(F\co {}_{S}\cat{H}\ra \Sets\) as follows. If \(\tilde{S}\) is in \({}_{S}\cat{H}\) let \(\tilde{T}\) denote \(T\hot_{S}\tilde{S}\) and let \(\tilde{\mc{L}}'\) denote \(\tilde{T}\ot_{T}\mc{L}'\). Then \(F(\tilde{S})\) is defined as equivalence classes of pairs of maps of finite \(\tilde{T}\)-modules \(\tilde{\xi}=(\tilde{\nu}\co \tilde{\mc{N}}\ra \bar{\mc{N}},\tilde{\iota}\co \tilde{\mc{N}}\ra\tilde{\mc{L}}')\) such that \(\tilde{\mc{N}}\) is \(\tilde{S}\)-flat. A map \(\tilde{S}\ra\tilde{S}'\) gives a map of pairs by base change. Two pairs, \({}^{1}\tilde{\xi}\) and \({}^{2}\tilde{\xi}\), are equivalent if there is an isomorphism \({}^{1\!}\tilde{\mc{N}}\cong{}^{2\!}\tilde{\mc{N}}\) commuting with the \({}^{j}\tilde{\iota}\) and the \({}^{j}\tilde{\nu}\). We show that \(F\) is locally of finite presentation. Suppose \(\tilde{S}=\liminj {}^{i\!}S\) for a filtered injective system of algebras in \({}_{S}\cat{H}\). Put \({}^{i}T=T\hot{}^{i\!}S\). Then \(\liminj {}^{i}T\cong \tilde{T}\) by \cite[7.8.3]{EGAIV2} and \cite[18.6.14]{EGAIV4}.
Given \(\tilde{\xi}\in F(\tilde{S})\) as above. Since \(\tilde{\mc{N}}\) has finite presentation and since the maps \(\tilde{\nu}\) and \(\tilde{\iota}\) can be represented on the finite presentations, there is a finite \({}^{i}T\)-module \({}^{i\!}\mc{N}\) and \({}^{i}T\)-linear maps \({}^{i}\nu\co {}^{i\!}\mc{N}\ra \bar{\mc{N}}\) and \({}^{i}\iota\co {}^{i\!}\mc{N}\ra{}^{i\!}\mc{L}'={}^{i}T\ot_{T}\mc{L}'\) inducing \(\tilde{\xi}\) by base change. We may also assume that \({}^{i\!}\mc{N}\) is \({}^{i\!}S\)-flat. Hence \(\liminj F({}^{i\!}S)\ra F(\tilde{S})\) is surjective, and injectivity is similar. Let \({}^{*}\xi\) denote the element in \(F({}^{*\!}S)\) given by \({}^{*\!}\mc{N}\ra \bar{\mc{N}}\) and the \(\hat{\cat{D}}\)-hull \({}^{*\!}\mc{N}\ra{}^{*\!}\mc{L}'\). By Artin's Approximation Theorem \cite[1.12]{art:69}, \cite[1.3]{pop:86}, \cite{pop:90} there exists a \(\xi=(\nu\co \mc{N}\ra \bar{\mc{N}},\iota\co \mc{N}\ra \mc{L}')\) in \(F(S)\) with \(\xi_{1}={}^{*}\xi_{1}\). In particular \(\mc{N}\ra\bar{\mc{N}}\) is a deformation. Since \(\iota_{0}\co \mc{N}_{0}\ra\mc{L}'_{0}\) equals the injective \(\bar{\mc{N}}\ra\bar{\mc{L}}'\) \cite[5.1-2]{ogu/ber:72} implies that \(\iota\) is injective and \(\coker \iota\) is \(S\)-flat. It follows that \(\iota\) is the \(\hat{\cat{D}}\)-hull of \(\mc{N}\), i.e.\ \(\sigma_{L'}(\mc{N})=\mc{L}'\). The \(L\)-case is analogous. 
\end{proof}
\begin{thm}\label{thm.defgrade2}
With general assumptions as in \textup{Theorem \ref{thm.defgrade}} consider the map 
\begin{equation*}
\sigma_{M}\co \df{}{(A,N)}\lra\df{}{(A,M)}\text{ of functors }\QH/k\lra\Sets
\end{equation*}
as in \textup{Definition \ref{defn.defmap}.}
\begin{enumerate}
\item[(i)] If\, \(\hm{}{A}{L}{N}=0\) then \(\sigma_{M}\) is injective\textup{.}

\item[(ii)] If\, \(\xt{1}{A}{L}{N}=0\) then \(\sigma_{M}\) is formally smooth\textup{.}

\item[(iii)] If\, \(\xt{1}{A}{L}{N}=0\) and \(\df{}{(A,N)}\) has a versal element then \(\sigma_{M}\) is smooth\textup{.}
\end{enumerate}
The analogous statements hold for \(\sigma_{M'}\co \df{}{(A,N)}\ra\df{}{(A,M')}\) with \(\xt{1}{A}{L'}{N}=0\) in \textup{(i)} and \(\xt{2}{A}{L'}{N}=0\) in \textup{(ii)} and \textup{(iii).}
\end{thm}
\begin{proof}
The proof is analogous to the proof of Theorem \ref{thm.defgrade}.
\end{proof}
Fix a field \(k\) and a Cohen-Macaulay local algebraic \(k\)-algebra \(A\). Let \(\vL\ra {\To}\) be obtained by henselisation at a maximal ideal of a finite type Cohen-Macaulay map \(\tilde\vL\ra {\To}\) where \(\tilde\vL\) is assumed to be an excellent ring. In particular \(\vL\) and \({\To}\) are excellent rings (\cite[7.8.3]{EGAIV2}, \cite[18.7.6]{EGAIV4}). Assume \({\To}/\fr{m}_{{\To}}\cong k\) and \({\To}\ot_{\vL}k\cong A\). There is a section \(\QH\ra\cat{hCM}\) defined by \(S\mapsto {\To}\hot S\). Let \(\cat{T^{o}}\) denote the resulting fibred subcategory of \(\cat{hCM}\) and \(\modf_{\vert\cat{T^{o}}}\) the restriction of the fibred category \(\modf\) to \(\cat{T^{o}}\).
Let \(\cat{Def}{}^{{\To}}_{\mc{N}}\) denote the category of deformations in \(\modf_{\vert\cat{T^{o}}}\) of an object \(\xi=(S\ra {\To}\hot S,\mc{N})\) with morphisms being cocartesian maps of deformations. The category \(\cat{Def}{}^{{\To}}_{\mc{N}}\) is fibred over \(\QH/S\). Let \(\df{{\To}}{\mc{N}}\co \QH/S\ra\Sets\) be the associated deformation functor. A special case is given by \(\vL=k\) and \({\To}=A\).
\begin{cor}\label{cor.defgrade}
With general assumptions as in \textup{Theorem \ref{thm.defgrade}} consider the map \(\sigma_{L'}\co \df{{\To}}{N}\ra\df{{\To}}{L'}\) of functors \(\QH/k\ra\Sets\)\textup{.}
\begin{enumerate}
\item[(i)] If\, \(\hm{}{A}{N}{M'}=0\) then \(\sigma_{L'}\) is injective\textup{.}

\item[(ii)] If\, \(\xt{1}{A}{N}{M'}=0\) then \(\sigma_{L'}\) is formally smooth\textup{.}

\item[(iii)] If\, \(\xt{1}{A}{N}{M'}=0\) and \(\df{A}{N}\) has a versal element then \(\sigma_{L'}\) is smooth\textup{.}
\end{enumerate}
The analogous statements hold for \(\sigma_{L}\co \df{{\To}}{N}\ra\df{{\To}}{L}\) with \(\xt{1}{A}{N}{M}=0\) in \textup{(i)} and \(\xt{2}{A}{N}{M}=0\) in \textup{(ii)} and \textup{(iii).}
\end{cor}
\begin{proof}
This is not a formal consequence of Theorem \ref{thm.defgrade}, but the proof is similar.
\end{proof}
\begin{cor}\label{cor.defgrade2}
With general assumptions as in \textup{Theorem \ref{thm.defgrade}} consider the map \(\sigma_{M}\co \df{{\To}}{N}\ra\df{{\To}}{M}\) of functors \(\QH/k\ra\Sets\)\textup{.}
\begin{enumerate}
\item[(i)] If\, \(\hm{}{A}{L}{N}=0\) then \(\sigma_{M}\) is injective\textup{.}

\item[(ii)] If\, \(\xt{1}{A}{L}{N}=0\) then \(\sigma_{M}\) is formally smooth\textup{.}

\item[(iii)] If\, \(\xt{1}{A}{L}{N}=0\) and \(\df{A}{N}\) has a versal element then \(\sigma_{M}\) is smooth\textup{.}
\end{enumerate}
The analogous statements hold for \(\sigma_{M'}\co \df{{\To}}{N}\ra\df{{\To}}{M'}\) with \(\xt{1}{A}{L'}{N}=0\) in \textup{(i)} and \(\xt{2}{A}{L'}{N}=0\) in \textup{(ii)} and \textup{(iii).}
\end{cor}
\begin{prop}\label{prop.defequiv}
With general assumptions as in \textup{Theorem \ref{thm.defgrade}:}
\begin{enumerate}
\item[(i)] If \(Q'=\hm{}{A}{\omega_{A}}{L'}\) and \(Q=\hm{}{A}{\omega_{A}}{L}\) then \(Q'\) and \(Q\) have finite projective dimension and \(\df{}{(A,L')}\cong\df{}{(A,Q')}\) and \(\df{}{(A,L)}\cong\df{}{(A,Q)}\)\textup{.}
\item[(ii)] There are natural maps 
\begin{equation*}
s\co \df{}{(A,M)}\lra\df{}{(A,M')}\quad \text{and}\quad t\co \df{}{(A,L')}\lra\df{}{(A,L)}
\end{equation*}
commuting with the maps 
\(\sigma_{X}\co \df{}{(A,N)}\ra\df{}{(A,X)}\) for \(X\) equal to \(M\) and \(M'\)\textup{,} and to \(L'\) and \(L\)\textup{,} respectively\textup{.} If \(A\) is a Gorenstein ring\textup{,} then \(s\) is an isomorphism\textup{.}
\end{enumerate}
The analogous statements also hold for the deformation functors \(\df{{\To}}{X}\)\textup{.}
\end{prop}
\begin{proof}
(ii) There is a short exact sequence \(0\ra M\ra \omega_{A}^{\oplus n}\ra M'\ra 0\) such that the last map is without a common \(\omega_{A}\)-summand, corresponding (through dualisation) to a short exact sequence \(0\la M^{\vee}\la A^{\oplus n}\la (M')^{\vee}\la 0\) where \(n\) is minimal. The map \(s\) is the composition \(\df{}{(A,M)}\cong\df{}{(A,M^{\vee})}\ra\df{}{(A,(M')^{\vee})}\cong\df{}{(A,M')}\) where the middle map is obtained by taking the syzygy of the deformation.
If \(A\) is a Gorenstein ring then \(\omega_{A}\cong A\) and there is an inverse \(\df{}{(A,M')}\ra\df{}{(A,M)}\) given by the syzygy map. 

Note that the pushout of \(M\ra \omega_{A}^{\oplus n}\) with \(M\ra N\) gives \(N\ra L'\). Consider the induced short exact sequence \(0\ra L\ra\omega_{A}^{\oplus n}\xra{\mu} L'\ra 0\). For a deformation \((h,\mc{L}')\) in \(\df{}{(A,L')}\) with structure map \(\lambda'\co \mc{L}'\ra L'\) there is a lifting of \(\mu\) to a map \(\tilde{\mu}\co \omega_{h}^{\oplus n}\ra \mc{L}'\). If \(\mc{L}\) denotes the kernel of \(\tilde{\mu}\) then there is a cocartesian map \(\lambda\co \mc{L}\ra L\) commuting with \(\omega_{h}^{\oplus n}\ra \omega_{A}^{\oplus n}\). By Lemma \ref{lem.defCMapprox2}, \((h,\lambda')\mapsto (h,\lambda)\) gives a well defined map of deformation functors \(t\co \df{}{(A,L')}\ra\df{}{(A,L)}\). 

Given a deformation \((h,\mc{N})\) in \(\df{}{(A,N)}\), let \(0\ra\mc{L}\ra\mc{M}\ra\mc{N}\ra0\) and \(0\ra\mc{N}\ra\mc{L}'\ra\mc{M}'\ra0\) be the minimal sequences in \eqref{eq.ny}.
There is a commutative diagram of short exact sequences with (co)cartesian square (cf.\ \cite{aus/buc:89})
\begin{equation}
\xymatrix@C-0pt@R-12pt@H-30pt{
&& 0 \ar[d] & 0 \ar[d] \\
0\ar[r] & \mc{L} \ar[r]\ar@{=}[d] & \mc{M} \ar[r]\ar[d]\ar@{}[dr]|{\Box} & \mc{N} \ar[r]\ar[d] & 0  \\   
0\ar[r] & \mc{L} \ar[r] & \omega_{h}^{\oplus n} \ar[r]\ar[d] & \mc{L}' \ar[r]\ar[d] & 0 \\
&& \mc{M}' \ar@{=}[r]\ar[d] & \mc{M}' \ar[d] \\
&& 0 & 0
}
\end{equation}
where \(\omega_{h}^{\oplus n}\ra \mc{L}'\) is given as above. The stated commutativity of maps of deformation functors follows.

(i) follows from Lemma \ref{lem.fri}.
\end{proof}
\begin{cor}\label{cor.defapprox}
Let \(A\) be an Cohen-Macaulay local algebraic \(k\)-algebra with residue field \(k\)\textup{.} Suppose \(\dim A\geq 2\)\textup{.} Then there exists finite \(A\)-modules \(L'\) and \(Q'\) with \(\Injdim L'=\dim A=\pdim Q'\) and universal deformations \(\mc{L}'\in\df{A}{L'}(A)\) and \(\mc{Q}'\in\df{A}{Q'}(A)\)\textup{.} 
\end{cor}
\begin{proof}
Let \(h=1\ot\id\co A\ra A\hot_{k}A=T\) and \(\mc{N}=A\) be the cyclic \(T\)-module defined through the multiplication map. Then \(T\ot_{A}k\cong A\) and \(\mc{N}\ot_{A}k\cong k\) and this gives a deformation \(\mc{N}\ra k\) of the residue field of \(A\) which is universal.
If \(L'\) is the minimal \(\hat{\cat{D}}_{A}\)-hull of the residue field \(k\) then \(\mc{L}'=\sigma_{L'}(\mc{N})\in \df{{\To}}{L'}(A)\) is universal by Corollary \ref{cor.defgrade}. If \(Q'=\hm{}{A}{\omega_{A}}{L'}\) then \(\hm{}{T}{\omega_{T}}{\mc{L}'}\in\df{A}{Q'}(A)\) is universal by Proposition \ref{prop.defequiv}.
\end{proof}
\begin{cor}\label{cor.depth}
With general assumptions as in \textup{Theorem \ref{thm.defgrade}} put \(X=\Spec A\)\textup{.} Let \(Z\) be a closed subscheme of \(X\) such that the complement \(U\) is contained in the regular locus\textup{.} Assume \(\tilde{N}_{\vert U}\) is locally free\textup{,} \(\depth_{Z}N\geq 2\) and \(\cH^{2}_{Z}(\hm{}{A}{L}{N})=0\)\textup{.} Then 
\begin{equation*}
\sigma_{M}\co \df{}{(A,N)}\lra\df{}{(A,M)}\quad\text{and}\quad \sigma_{M}^{{\To}}\co \df{{\To}}{N}\lra\df{{\To}}{M}\quad\text{are formally smooth\textup{.}}
\end{equation*}
\end{cor}
\begin{proof}
We show that \(\xt{1}{A}{L}{N}=0\) and apply Theorem \ref{thm.defgrade2} and Corollary \ref{cor.defgrade2}.
By Th{\'e}or{\`e}me 1.6 in \cite[Expos{\'e} VI]{SGA2} there is a cohomological spectral sequence
\begin{equation}
\cE^{p,q}_{2}=\xt{q}{A}{L}{\cH^{p}_{Z}(N)}\,\Ra\,\xt{p+q}{Z}{X;L}{N}\,.
\end{equation}
Since \(\cH^{i}_{Z}(N)=0\) for \(i=0,1\) the restriction map \(\xt{1}{A}{L}{N}\ra\xt{1}{U}{X;L}{N}\) in the long exact sequence is injective. Since \(U\) is contained in the regular locus, \(\tilde{M}_{\vert U}\) and hence \(\tilde{L}_{\vert U}\) are locally free. It follows that \(\xt{1}{U}{X;L}{N}\) is isomorphic to 
\begin{equation}
\xt{1}{\Q_{U}}{\tilde{L}_{\vert U}}{\tilde{N}_{\vert U}}\cong\cH^{1}(U,\shm{}{\Q_{X}}{\tilde{L}}{\tilde{N}}) \cong \cH^{2}_{Z}(\hm{}{A}{L}{N})
\end{equation}
which is zero by assumption.
\end{proof}
\begin{ex}\label{ex.depth}
The condition \(\cH^{2}_{Z}(\hm{}{A}{L}{N})=0\) is implied by \(\tilde{N}_{\vert U}=0\) or
\(\depth_{Z}(\hm{}{A}{L}{N})\geq 3\).
\end{ex}
The following result extends A.\ Ishii's \cite[3.2]{ish:00} to deformations of the pair.
\begin{prop}\label{prop.Gor}
Let \(k\) be a field and let \(A\) be a Gorenstein local algebraic \(k\)-algebra. Suppose \(0\ra L\ra M\ra N\ra 0\) is the minimal Cohen-Macaulay approximation of a finite \(A\)-module \(N\)\textup{.} If \(\depth N=\dim A -1\) then 
\begin{equation*}
\sigma_{M}\co \df{}{(A,N)}\lra\df{}{(A,M)}\quad\text{and}\quad \sigma_{M}^{{\To}}\co \df{{\To}}{N}\lra\df{{\To}}{M}\quad\text{are smooth\textup{.}}
\end{equation*}
\end{prop}
\begin{proof}
Let \(S_{v}\ra S\) be a surjection in \(\QH/k\) and \((h_{v}\co S_{v}\ra T_{v},\mc{M}_{v})\) an element in \(\df{}{(A,M)}(S_{v})\) which maps to \((h\co S\ra T,\mc{M})\) in \(\df{}{(A,M)}(S)\). Suppose \(\sigma_{M}\) maps \((h',\mc{N})\) in \(\df{}{(A,N)}(S)\) to \((h,\mc{M})\). By assumption \(L\cong A^{\oplus r}\) for some \(r\). We can assume that \(h'=h\) and that the minimal \(\cat{MCM}\)-approximation of \(\mc{N}\) is \(0\ra\mc{L}\xra{\rho}\mc{M}\ra\mc{N}\ra0\) where \(\mc{L}\cong T^{\oplus r}\). Put \(\mc{L}_{v}:=T_{v}^{\oplus r}\) and choose a lifting \(\rho_{v}\co \mc{L}_{v}\ra\mc{M}_{v}\) of \(\rho\). Put \(\mc{N}_{v}:=\coker\rho_{v}\) with its natural map to \(\mc{N}\). Then \(\mc{N}_{v}\) is \(S_{v}\)-flat (\(\rho_{v}\ot S=\rho\)) and \(\sigma_{M}(h_{v},\mc{N}_{v})=(h_{v},\mc{M}_{v})\). 
\end{proof}
\begin{rem}
In the case \(A\) is a Gorenstein domain any MCM \(A\)-module \(M\) is given as an \(\cat{MCM}_{A}\)-approximation since there is a short exact sequence \(0\ra A^{r}\ra M\ra N\ra 0\) with \(N\) a codimension one Cohen-Macaulay module; cf.\ \cite[1.4.3]{bru/her:98}. In particular Proposition \ref{prop.Gor} applies. However, it is not always possible to continue this reduction: If \(A\) is a normal Gorenstein complete local ring any MCM \(A\)-module \(M\) is the \(\cat{MCM}_{A}\)-approximation of a codimension \(2\) Cohen-Macaulay module up to stable isomorphism if and only if \(A\) is a unique factorisation domain; see \cite{yos/iso:00, kat:07}.

Let \(A\) be a Gorenstein normal domain of dimension \(2\) and \(0\ra A^{r{-}1}\ra M\ra I\ra 0\) the minimal MCM approximation of a torsion-free rank \(1\) module \(I\). Let \(U\) denote the regular locus in \(X=\Spec A\). If \(T=A\hot_{k}S\) for \(S\) in \(\kH/k\) there is a natural section \(A\ra T\). Let \(U_{T}\) denote \(U\times_{X}\Spec T\). Consider the subfunctor \(\df{A,\wedge}{M}\sbeq\df{A}{M}\) of deformations \(\mc{M}\) with trivial induced deformation \(\wedge^{r}\mc{M}_{\vert U_{T}}\). Note that \(\cH^{0}(U,\wedge^{r}M)\) is isomorphic to the MCM \(A\)-module \(\bar{I}:=\cH^{0}(U,I)\). It follows from Propostion \ref{prop.Gor} that the resulting map from the (local) quotient functor \(\quot{\bar{I}}{I\sbeq \bar{I}}\ra \df{A,\wedge}{M}\) is also smooth; cf.\ \cite[3.2]{ish:00}. In particular, if \(E_{A}\) is the fundamental module and \(A/\fr{m}_{A}\cong k\) then \(\hm{}{\kH/k}{A}{-}\cong\quot{A}{\fr{m}_{A}\sbeq A}\cong\df{A}{A/\fr{m}_{A}}\) gives a versal family for \(\df{A,\wedge}{E_{A}}\) by the MCM approximation in Example \ref{ex.semicont}; see \cite[3.4]{ish:00}.
\end{rem}
\begin{ex}
Assume \(A/\fr{m}_{A}\cong k\) and let \(M\) denote the minimal MCM approximation of \(k\). It is given as \(M\cong \hm{}{A}{\syz{A}{d}(k^{\vee})}{\omega_{A}}\) where \(d=\dim A\); cf.\ Remark \ref{rem.Buch}. One has \(k^{\vee}=\xt{d}{A}{k}{\omega_{A}}\cong k\). We apply \(\hm{}{A}{-}{\omega_{A}}\) to the short exact sequence \(0\ra \syz{A}{}(\fr{m}_{A})\ra A^{\oplus\beta_{1}}\xra{(\ul{x})}\fr{m}_{A}\ra 0\). Assume \(\dim A = 2\). Since \(\xt{1}{A}{\fr{m}_{A}}{\omega_{A}}\cong k\) we obtain the MCM approximation of \(k\):
\begin{equation}
0\ra\omega_{A}\xra{(\ul{x})^{\text{tr}}}\omega_{A}^{\oplus\beta_{1}}\lra M\lra k\ra 0
\end{equation}
In particular \(\Rk(M)=\beta_{1}-1\) and \(\mu(M)=t(A)\cdot\beta_{1}+1\) where \(t(A)\) is the Cohen-Macaulay type of \(A\). If \(A=A(m)=k[u^{m},u^{m-1}v,\dots,v^{m}]^{\text{h}}\), the vertex of the cone over the rational normal curve of degree \(m\), the indecomposable MCM \(A\)-modules are \(M_{i}=(u^{i},u^{i-1}v,\dots,v^{i})\) for \(i=0,\dots,m{-}1\). One finds that \(M=M_{m-1}^{\oplus m}\) and 
\begin{equation}
\dim_{k}\df{A}{M}(k[\vare])=\dim_{k}\xt{1}{A}{M}{M}= (m-1)\cdot m^{2}
\end{equation}
while \(\dim_{k}\df{A}{k}(k[\vare])=m+1\). Even in the Gorenstein case (\(m=2\)) the tangent map is not surjective and so Proposition \ref{prop.Gor} cannot in general be extended to \(\depth N=\dim A -2\). See \cite{gus/ile:04b} for a detailed description of the strata of the reduced versal deformation space of \(M\) defined by Ishii in \cite{ish:00}. 

If \(\dim A=2\) the \(\cat{MCM}_{A}\)-approximation of \(\fr{m}_{A}\) is a short exact sequence \(0\ra\omega_{A}\ra E_{A}\ra \fr{m}_{A}\ra0\) where \(E_{A}\) is called the \emph{fundamental module}, cf.\ \eqref{eq.Fmod}. Applying \(\hm{}{A}{k}{-}\) to \(0\ra\fr{m}_{A}\ra A\ra k\ra0\) gives an exact sequence
\begin{equation}
0\ra\xt{1}{A}{k}{k}\lra\df{A}{\fr{m}_{A}}(k[\vare])\lra k^{\oplus t(A)}\lra \xt{2}{A}{k}{k}
\end{equation}
since \(\xt{1}{A}{\fr{m}_{A}}{\fr{m}_{A}}\cong\xt{2}{A}{k}{\fr{m}_{A}}\) and \(\dim A = 2\). If \(A=A(m)\) then \(E_{A}\) is isomorphic to \(M_{m-1}^{\oplus 2}\) with \(\dim_{k}\df{A}{E_{A}}(k[\vare])=4(m-1)\). Hence the conclusion in Proposition \ref{prop.Gor} cannot hold in the non-Gorenstein case \(m>2\).
\end{ex}
\section{Existence of versal elements}\label{sec.versal}
We prove existence of a versal element for deformations of a pair (algebra, module) with isolated singularity. The following lemma is used in the proof.
\begin{lem}\label{lem.AQhens}
Let \(h^{\textnormal{ft}}\co S\ra T^{\textnormal{ft}}\) be a finite type homomorphism of noetherian rings\textup{.} Let \(\vG^{\textnormal{ft}}=T^{\textnormal{ft}}{\oplus} \mc{N}\) be the graded \(S\)-algebra with a finite \(T^{\textnormal{ft}}\)-module \(\mc{N}\) in degree \(1\) and let \(I={}^{0\!}I{\oplus}{}^{1\!}I\) be a graded \(\vG^{\textnormal{ft}}\)-module with \(T^{\tn{ft}}\)-module \({}^{i\!}I\) in degree \(i\)\textup{.} Let \(T\) denote the henselisation of \(T^{\textnormal{ft}}\) in a maximal ideal \(\fr{m}\) and put \(\vG=T\ot_{T^{\tn{ft}}}\vG^{\textnormal{ft}}\)\textup{.} 
\begin{enumerate}
\item[(a)] There are natural isomorphisms of graded Andr{\'e}-Quillen cohomology
\begin{equation*}
\gH{0}^{i}(S,\vG,T\ot_{T^{\tn{ft}}}I)\cong\gH{0}^{i}(S, \vG^{\textnormal{ft}},I)\ot_{T^{\textnormal{ft}}}T\,\,\text{ for all } i\,.
\end{equation*}
\end{enumerate}
Suppose in addition that \(h^{\textnormal{ft}}\) is flat\textup{,} \(I\) is finite as \(T^{\textnormal{ft}}\)-module\textup{,} \(S\) is local henselian and \(S/\fr{m}_{S}\cong T^{\textnormal{ft}}/\fr{m}\cong k\)\textup{.} 
Let \(k\ra A^{\textnormal{ft}}\) denote the central fibre of \(h^{\textnormal{ft}}\)\textup{.} Put \(\fr{m}_{0}=\fr{m}A^{\textnormal{ft}}\) and \(N=\mc{N}\ot_{S}k\)\textup{.} Assume that \(V=\Spec A^{\textnormal{ft}} \setminus\{\fr{m}_{0}\}\) is smooth over \(k\) and that \(N\) restricted to \(V\) is locally free\textup{.}
\begin{enumerate}
\item[(b)] For all \(i>0\) the graded Andr{\'e}-Quillen cohomology \(\gH{0}^{i}(S,\vG,T\ot_{T^{\tn{ft}}} I)\) is finite as \(S\)-module and there is a natural \(T^{\textnormal{ft}}_{\fr{m}}\)-isomorphism
\begin{equation*}
\gH{0}^{i}(S,\vG,T\ot_{T^{\tn{ft}}} I)\cong\gH{0}^{i}(S, \vG^{\textnormal{ft}},I)_{\fr{m}}\,.
\end{equation*}
\end{enumerate}
\end{lem}
\begin{proof}
(a) Note that there are natural maps \(\gH{0}^{i}(S,\vG,T\ot_{T^{\tn{ft}}} I)\ra \gH{0}^{i}(S,\vG^{\textnormal{ft}},T\ot_{T^{\tn{ft}}} I)\) and \(\gH{0}^{i}(S, \vG^{\textnormal{ft}},I)\ot_{T^{\textnormal{ft}}}T\ra\gH{0}^{i}(S,\vG^{\textnormal{ft}},T\ot_{T^{\tn{ft}}} I)\). We consider the natural long exact sequence in Proposition \ref{prop.lang} and prove that the corresponding maps for the ungraded Andr{\'e}-Quillen and Ext cohomology are isomorphisms.
The cotangent complex is trivial for Zariski localisation and by the transitivity sequence \(\cH_{i}(T^{\textnormal{ft}},T,-)\ra \cH_{i}(T^{\textnormal{ft}}_{\fr{m}},T,-)\) is an isomorphism for all \(i\).  Andr{\'e}-Quillen homology commutes with direct limits in the second argument \cite[III 35]{and:74} and the cotangent complex is trivial for {\'e}tale extensions \cite[III 3.1.1]{ill:71}. Hence \(\cH_{i}(T^{\textnormal{ft}},T,T)=0\) for all \(i\). By \cite[III 21]{and:74} 
\begin{equation}
\cH^{i}(T^{\textnormal{ft}},T,T\ot_{T^{\tn{ft}}} {}^{0\!}I)\cong \hm{}{T}{\cH_{i}(T^{\textnormal{ft}},T,T)}{T\ot_{T^{\tn{ft}}} {}^{0\!}I}=0 \text{ for all }i\,.
\end{equation}
From the transitivity sequence \(\cH^{i}(S,T,T\ot_{T^{\tn{ft}}} {}^{0\!}I)\cong \cH^{i}(S,T^{\textnormal{ft}},T\ot_{T^{\tn{ft}}} {}^{0\!}I)\) for all \(i\). Moreover \(T^{\textnormal{ft}}\ra T\) is flat and \(T^{\text{ft}}\) is of finite type over the noetherian \(S\) and we obtain 
\begin{equation}
\cH^{i}(S,T^{\textnormal{ft}},{}^{0\!}I)\ot_{T^{\textnormal{ft}}} T\cong \cH^{i}(S,T^{\textnormal{ft}},T\ot_{T^{\tn{ft}}} {}^{0\!}I)\cong\cH^{i}(S,T,T\ot_{T^{\tn{ft}}} {}^{0\!}I)
\end{equation}
 for all \(i\) \cite[IV 58]{and:74}. Since \(T\) is \(T^{\textnormal{ft}}\)-flat and \(\mc{N}\) finite
\begin{equation}
 \xt{i}{T^{\textnormal{ft}}}{\mc{N}}{{}^{1\!}I}\ot_{T^{\textnormal{ft}}}T\cong
 \xt{i}{T^{\textnormal{ft}}}{\mc{N}}{T\ot_{T^{\tn{ft}}} {}^{1\!}I}\cong
 \xt{i}{T}{T\ot_{T^{\tn{ft}}} \mc{N}}{T\ot_{T^{\tn{ft}}}{}^{1\!}I}\,.
\end{equation}

(b) The non-smooth locus of \(h^{\text{ft}}\) is closed, i.e.\ defined by an ideal \(J\sbeq T^{\text{ft}}\). Smooth is equivalent to flat with smooth fibres \cite[17.5.1]{EGAIV4}. Hence \(J_{0}=JA^{\textnormal{ft}}\) defines the non-smooth locus of \(k\ra A^{\text{ft}}\) and \(J_{0}\) is \(\fr{m}_{0}\)-primary. Put \(\bar{T}^{\text{ft}}=T^{\text{ft}}/J\). Since \(A^{\text{ft}}/J_{0}\) has finite length, \(S\ra \bar{T}^{\text{ft}}\) is quasi-finite at \(\fr{m}\) \cite[Err 20]{EGAIII2} by Chevalley's upper semi-continuity theorem \cite[13.1.3]{EGAIV3} and openness of \(\Spec T^{\text{ft}}\ra \Spec S\) \cite[2.4.6]{EGAIV2}. Since \(S\) is henselian it follows that there is a ring \(T'\) such that  \(\bar{T}^{\text{ft}}\) is isomorphic to \(\bar{T}^{\text{ft}}_{\fr{m}}\prod T'\) where \(\bar{T}^{\text{ft}}_{\fr{m}}\) is finite as \(S\)-module \cite[18.5.11]{EGAIV4}. Hence there is a Zariski neighborhood \(U\) of \(\fr{m}\) in \(\Spec T^{\text{ft}}\) such that non-smooth locus \(U\cap V(J)\) is finite over \(S\) and the support of \(\cH^{i}=\cH^{i}(S,T^{\text{ft}},{}^{0\!}I)\) in \(U\) is contained in \(U\cap V(J)\) for all \(i>0\) by \cite[III 3.1.2]{ill:71}. Since the localisation \(\cH^{i}_{\fr{m}}\) equals \(\cH^{i}\) restricted to \(U\), it follows that \(\cH^{i}_{\fr{m}}\) is finite as \(S\)-module. With the isomorphism in (a) we get 
\begin{equation}
\cH^{i}(S,T,T\ot_{T^{\tn{ft}}} {}^{0\!}I)\cong\cH^{i}(S,T^{\text{ft}},{}^{0\!}I)_{\fr{m}}\ot_{T^{\text{ft}}_{\fr{m}}}T\cong \cH^{i}(S,T^{\text{ft}},{}^{0\!}I)_{\fr{m}} \,.
\end{equation}
The locus where \(\mc{N}\) is not locally free is closed, i.e.\ defined by an ideal \(J'\sbeq T^{\text{ft}}\). Locally free is equivalent to flat and locally free fibres. Hence \(J'_{0}=J'A^{\textnormal{ft}}\) defines the singular locus of \(N\) and \(J'_{0}\) is \(\fr{m}_{0}\)-primary. As for the Andr{\'e}-Quillen cohomology we get a Zariski neighborhoud \(U'\) of \(\fr{m}\) such that \(\cE^{i}=\xt{i}{T^{\textnormal{ft}}}{\mc{N}}{{}^{1\!}I}\) restricted to \(U'\) equals \(\cE^{i}_{\fr{m}}\) and is finite as \(S\)-module for all \(i>0\). By (a) with \({}^{0\!}I=0\) and Proposition \ref{prop.lang} we get
\begin{equation}
\xt{i}{T}{T\ot_{T^{\tn{ft}}} \mc{N}}{T\ot_{T^{\tn{ft}}} {}^{1\!}I}\cong \xt{i}{T^{\textnormal{ft}}}{\mc{N}}{{}^{1\!}I}\ot_{T^{\textnormal{ft}}}T\cong \xt{i}{T^{\textnormal{ft}}}{\mc{N}}{{}^{1\!}I}_{\fr{m}} \,.
\end{equation}
By (a) there is a natural map \(\gH{0}^{i}(S, \vG^{\textnormal{ft}},I)_{\fr{m}}\ra\gH{0}^{i}(S,\vG,T\ot_{T^{\tn{ft}}} I)\). We conclude by the natural long exact sequence in Proposition \ref{prop.lang} and the \(5\)-lemma.
\end{proof}
Let \(k\) be a field, \(A\) an algebraic \(k\)-algebra with \(A/\fr{m}_{A}\cong k\) and \(N\) a finite \(A\)-module. Without any Cohen-Macaulay condition on \(A\) we define a deformation \((h\co S\ra T,\mc{N})\) of the pair \((A,N)\) to an \(S\) in \(\QH/k\) as before and obtain the deformation functor \(\df{}{(A,N)}\co \QH/k\ra\Sets\) as equivalence classes of deformations of pairs.

We say that \(A\) is an \emph{isolated singularity over \(k\)} if there is a finite type \(k\)-algebra \(A^{\text{ft}}\) with a maximal ideal \(\fr{m}_{0}\) such that the henselisation \((A^{\text{ft}})^{\text{h}}_{\fr{m}_{0}}\) is isomorphic to \(A\) and which is smooth over \(k\) at all points in \(\Spec A^{\text{ft}} \setminus\{\fr{m}_{0}\}\). We say that the pair \((A,N)\) is an \emph{isolated singularity over \(k\)} if \(A\) is an isolated singularity over \(k\) and if \(N_{\fr{p}}\) is a free \(A_{\fr{p}}\)-module for all prime ideals \(\fr{p}\neq\fr{m}_{A}\). The following theorem is a consequence of results of R.\ Elkik and an argument of H.\ von Essen.
\begin{thm}\label{thm.ExVers}
Let \((A,N)\) be an isolated singularity over the field \(k\) with \(A\) equidimensional\textup{.} Then \(\df{}{(A,N)}\co \QH/k\ra\Sets\) has a versal element\textup{.}
\end{thm}
\begin{proof}
We apply \cite[3.2]{art:74} with the extension to arbitrary excellent coefficient rings given by \cite[1.5]{con/jon:02} to show the existence of a formally versal element for \(\df{}{(A,N)}\). By the finiteness conditions it follows that \(\df{}{(A,N)}\) is locally of finite presentation. The condition (S1) holds in general by Proposition \ref{prop.grobs}. Let \((h\co S\ra T,\mc{N})\) be a deformation to \(S\). 
Let \((T^{\text{ft}}, \fr{m})\) be an \(S\)-flat finite type representative for \(T\) such that \(A^{\text{ft}}=T^{\text{ft}}\ot_{S}k\) has a single non-smooth closed point \(\fr{m}_{0}=\fr{m}A^{\text{ft}}\) and let \(\mc{N}^{\text{ft}}\) be an \(S\)-flat finite \(T^{\text{ft}}\)-module representing \(\mc{N}\). The singular locus of \(N^{\text{ft}}=\mc{N}^{\text{ft}}\ot_{S}k\) is closed and equal to \(V(J_{0})\) for some ideal \(J_{0}\sbeq A^{\text{ft}}\). But \(\fr{m}_{0}\) is isolated in \(V(J_{0})\) so (possibly after inverting some element in \(T^{\text{ft}}\)) we may assume that \(N^{\text{ft}}\) is locally free away from \(\fr{m}_{0}\).
Let \(I\) be a finite \(S\)-module. By Lemma \ref{lem.AQhens}, \(\gH{0}^{1}(S,\vG,\vG\ot_{S} I)\) with \(\vG=T{\oplus}\mc{N}\) is a finite \(S\)-module, so by Proposition \ref{prop.grobs} condition (S2) holds.

For effectivity, there is a deformation functor \(\df{}{(A^{\text{ft}},N^{\text{ft}})}\co \QH/k\ra\Sets\) of base change maps of pairs \((S\ra T^{\text{ft}},\mc{N}^{\text{ft}})\ra (k\ra A^{\text{ft}},N^{\text{ft}})\) where \(T^{\text{ft}}\) is a flat \(S\)-algebra of finite type and \(\mc{N}^{\text{ft}}\) is an \(S\)-flat finite \(T^{\text{ft}}\)-module. Base change is given by the standard tensor product. Similarly there is a \(\df{}{A^{\text{ft}}}\). Restricted to \(\QA/k\) \(\df{}{(A^{\text{ft}},N^{\text{ft}})}\) satisfies (S1) and (S2).  Hence there exists a formally versal formal element \(\{(T_{n}^{\text{ft}},\mc{N}^{\text{ft}}_{n})\}\) in \(\limproj\df{}{(A^{\text{ft}},N^{\text{ft}})}(S_{n})\) where \(S_{n}=S/\fr{m}_{S}^{n+1}\) for some \(S=\hat{S}\) in \(\QH/k\). By \cite[Th\'{e}or\`{e}m 7, p.\ 595]{elk:73} (cf.\ \cite[II 5.1]{art:76}) there exists an element \(S\ra T^{\text{ft}}\) in \(\df{}{A^{\text{ft}}}(S)\) which induces \(\{T_{n}^{\text{ft}}\}\). Let \(T'\) be the henselisation of \(T^{\text{ft}}\) in the maximal ideal \(\fr{m}=(T^{\text{ft}}{\ra} A)^{{-}1}(\fr{m}_{A})\). Then \(S\ra T'\) is a deformation of \(A\). Let \(T^{*}\) be the completion of \(T'\) at the ideal \(\fr{n}=\fr{m}_{S}T'\) and let \(\mc{N}^{*}=\limproj \mc{N}_{n}\). Then \(\mc{N}^{*}\) is an \(S\)-flat finite \(T^{*}\)-module. Let \(J^{*}\sbeq T^{*}\) denote the ideal \(I(\phi)\) where \(\phi\) is a minimal presentation of \(\mc{N}^{*}\). Then \(J^{*}\) defines the locus \(V(J^{*})\) where \(\mc{N}^{*}\) is not locally free. Let \(J=\ker(T'\ra T^{*}{/}J^{*})\). Since \(T^{*}/J^{*}\) is finite as \(S\)-module, \(T'/J\cong T^{*}/J^{*}\). The proof of \cite[2.3]{ess:90} works in this situation too (there is a typo in line 5: it should be a direct sum, not a tensor product) and shows that the completion of \(T'\) in the ideal \(\fr{a}=J\cap\fr{m}_{S}T'\) equals \(T^{*}\).  Since \(\mc{N}^{*}\) is locally free on the complement of \(V(\fr{a}T^{*})\), the conditions in \cite[Th\'{e}or\`{e}m 3]{elk:73} hold. From this result we obtain a finite \(T'\)-module \(\mc{N}\) inducing \(\mc{N}^{*}\). In particular \(\mc{N}\) is \(S\)-flat. 

We claim that the henselisation map \(\df{}{(A^{\text{ft}},N^{\text{ft}})}\ra\df{}{(A,N)}\) is formally smooth. It follows that the element \((T',\mc{N})\) in \(\df{}{(A,N)}(S)\) is formally versal. For the claim, put \(\vG^{\text{ft}}=A^{\text{ft}}{\oplus}N^{\text{ft}}\) and \(\vG=A{\oplus} N\) and let \(\pi\co S_{1}\ra S_{0}=S_{1}/I\) be a small surjection in \(\QA/k\). 
The obstruction \(\ob(\pi,\vG^{\text{ft}}_{0})\in\gH{0}^{2}(k,\vG^{\text{ft}},\vG^{\text{ft}})\ot_{k}I\) for lifting a deformation \(\vG^{\text{ft}}_{0}\) of \(\vG^{\text{ft}}\) along \(\pi\) maps to the corresponding obstruction \(\ob(\pi,\vG_{0})\in\gH{0}^{2}(k,\vG,\vG)\ot_{k}I\). 
The isomorphisms \(\cH^{i}(S, T,T)\cong \cH^{i}(S, T^{\text{ft}},T^{\text{ft}})\ot_{T^{\text{ft}}}T\) for all \(i\) implies isomorphisms \(\gH{0}^{i}(k,\vG^{\text{ft}},\vG^{\text{ft}})\cong \gH{0}^{i}(k,\vG,\vG)\) for \(i=1,2\) as in the beginning of the proof. Smoothness follows by the standard obstruction argument. By \cite[3.2]{art:74} and \cite{con/jon:02} there is an algebraic \(k\)-algebra \(R\) and a formally versal element \((T,\mc{N})\) in \(\df{}{(A,N)}(R)\).

Finally we apply \cite[3.3]{art:74} (for general excellent coefficients) to conclude that the formally versal element \((T,\mc{N})\) is versal. We already have (S1) and (S2). To check \cite[3.3(iii)]{art:74}, let \(S\) be algebraic in \(\QH/k\), \(I\) an ideal in \(S\) and put \(S^{*}=\limproj S_{n}\) where \(S_{n}=S/I^{n{+}1}\).  Let \({}^{i}\xi=({}^{i}T,{}^{i}\mc{N})\) for \(i=1,2\) be two elements in \(\df{}{(A,N)}(S^{*})\) and \(\{\theta_{n}\co {}^{1}\xi_{n}\cong{}^{2}\xi_{n}\}\) be a tower of isomorphisms between the \(S_{n}\)-truncations. There are finite type representatives \({}^{i}\xi^{\text{ft}}=({}^{i}T^{\text{ft}},{}^{i}\mc{N}^{\text{ft}})\) of the \({}^{i}\xi\). By the cohomology argument above one obtains by induction a tower of isomorphisms \(\{\theta_{n}^{\text{ft}}\co {}^{1}\xi_{n}^{\text{ft}}\cong{}^{2}\xi_{n}^{\text{ft}}\}\) inducing \(\{\theta_{n}\}\). Since \(\limproj \hat{S}/I^{n{+}1}\hat{S}\cong S^{*}\) where \(\hat{S}\) is the completion of \(S\) in the maximal ideal, we can apply \cite[Lemme p.\ 600]{elk:73} to conclude that the henselisations of the \({}^{i}T^{\text{ft}}\) in \({}^{i}T^{\text{ft}}I\) are isomorphic by an isomorphism lifting \(\theta_{0}:{}^{1}T_{0}\cong{}^{2}T_{0}\). 
Further henselisation in the maximal ideals gives an isomorphism of deformations \({}^{1}T\cong{}^{2}T\). By Lemma \ref{lem.Aapprox} the isomorphism is extended to an isomorphism of the pairs \(\psi\co {}^{1}\xi\cong{}^{2}\xi\) which lifts \(\theta_{0}\). By \cite[1.3]{ess:90} condition \cite[3.3(ii)]{art:74} is redundant and we conclude that \((T,\mc{N})\) is versal.
\end{proof}
\begin{rem}\label{rem.isomod}
Let \(A\) be an Cohen-Macaulay algebraic \(k\)-algebra and \(N\) a finite \(A\)-module. We say that \(N\) has an \emph{isolated singularity} if \(N_{\fr{p}}\) is a free \(A_{\fr{p}}\)-module for all prime ideals \(\fr{p}\neq\fr{m}_{A}\). Then an easier argument gives that \(\df{{\To}}{N}\) has a versal element. This is the result \cite[2.4]{ess:90} of von Essen, but for a slightly different fibred category of deformations where henselisation is taken along the closed fibre. However it implies the result in our case, essentially by henselisation at \(\fr{m}_{0}\). Corollaries \ref{cor.ExVers} and \ref{cor.2dim} below have obvious analogs for \(\df{{\To}}{N}\).
\end{rem}
\begin{cor}\label{cor.ExVers}
Suppose \(A\) is a Cohen-Macaulay isolated singularity over the field \(k\) and \(N\) is a finite length \(A\)-module\textup{.} Let \(0\ra L\ra M\ra N\ra 0\) and \(0\ra N\ra L'\ra M'\ra0\) be the minimal \(\cat{MCM}_{A}\)-approximation and \(\hat{\cat{D}}_{A}\)-hull of \(N\)\textup{,} respectively\textup{.}
\begin{enumerate}
\item[(i)] \(\df{}{(A,N)}\) has a versal element\textup{.}
\item[(ii)] If \(\dim A\geq 2\) and \(Q'\) denotes \(\hm{}{A}{\omega_{A}}{L'}\) then
\begin{equation*}
\df{}{(A,N)}\cong\df{}{(A,L')}\cong\df{}{(A,Q')}\textup{.}
\end{equation*}
\item[(iii)]  If \(\dim A\geq 3\) and \(Q\) denotes \(\hm{}{A}{\omega_{A}}{L}\) then 
\begin{equation*}
\df{}{(A,N)}\cong\df{}{(A,L)}\cong\df{}{(A,Q)}\textup{.}
\end{equation*}
\end{enumerate} 
\end{cor}
\begin{proof}
This is Theorem \ref{thm.ExVers}, Theorem \ref{thm.defgrade} and Proposition \ref{prop.defequiv}.
\end{proof}
\begin{cor}\label{cor.2dim}
Suppose \(A\) is an algebraic \(k\)-algebra which is a Gorenstein normal domain with \(\dim A=2\) and \(N\) is a finite torsion-free \(A\)-module\textup{.} Let \(0\ra L\ra M\ra N\ra0\) be the minimal \(\cat{MCM}_{A}\)-approximation of \(N\)\textup{.} Assume \(k\) is a perfect field\textup{.} Then \(\df{}{(A,N)}\) and \(\df{}{(A,M)}\) both have versal elements and the map \(\sigma_{M}\co \df{}{(A,N)}\ra\df{}{(A,M)}\) is smooth\textup{.}
\end{cor}
\begin{proof}
Since \(A\) is a domain \(N\) torsion-free implies that \(N\) is a first syzygy. It follows that \(N_{\fr{p}}\) is a MCM \(A_{\fr{p}}\)-module for all primes \(\fr{p}\neq\fr{m}_{A}\) and since \(A\) is \(2\)-dimensional and normal \(N_{\fr{p}}\) is free. As \(k\) is perfect it follows from \cite[6.7.7 and 6.8.6]{EGAIV2} that \((A,N)\) and \((A,M)\) are isolated singularities and hence Theorem \ref{thm.ExVers} applies. Since \(\depth N\geq 1\), \(L\) is projective and by Theorem \ref{thm.defgrade2} (iii) \(\sigma_{M}\) is smooth.
\end{proof}
\section{Deforming maximal Cohen-Macaulay approximations\\ of Cohen-Macaulay modules}\label{sec.defCM}
Let \(h\co S\ra T\) be a homomorphism of local noetherian rings. An \emph{\(h\)-sequence} (or just an \(h\)-regular element if \(n=1\)) is a sequence \(J=(f_{1},\dots,f_{n})\) in \(T\) such that the image \(\bar{J}\) in \(A=T\ot_{S}S/\fr{m}_{S}\) is an \(A\)-sequence. Applying the Koszul complex \(K(J)\) to define a cohomological \(\delta\)-functor, \cite[5.1-2]{ogu/ber:72} implies that an \(h\)-sequence is a transversally \(T\)-regular sequence relative to \(S\) as defined in \cite[19.2.1]{EGAIV4}.
In particular, \(J\) is an \(h\)-sequence if and only if \(J\) is a \(T\)-sequence and \(T/J\) is \(S\)-flat. 
\begin{thm}\label{thm.defMCM}
Let \(q\co \vL\ra {\To}\) denote the henselisation of a finite type Cohen-Macaulay map at a maximal ideal with \({\To}/\fr{m}_{{\To}}=k\) and \({\To}\ot_{\vL}k=A\)\textup{.} Suppose \(J=(f_{1},\dots,f_{n})\) is a \(q\)-sequence\textup{.} Put \({\bTo}={\To}/J\)\textup{,} \(B={\bTo}\ot_{\vL}k\) and let \(\bar{J}\) be the image of \(J\) in \(A\)\textup{.}
Let \(N\) be a maximal Cohen-Macaulay \(B\)-module and \(0\ra L\ra M\ra N\ra 0\) the minimal \(\cat{MCM}_{A}\)-approximation of \(N\)\textup{.} 
If \(\ob{}{}(A/\bar{J}^{2}\ra B,N)=0\) then the composition of maps
\begin{equation*}
\df{{\bTo}}{N}\lra\df{{\To}}{N} \lra\df{{\To}}{M}
\end{equation*}
of functors \(\QH/k\ra\Sets\) is injective\textup{.}
\end{thm}
\begin{ex}
The existence of a splitting \(B\ra A/\bar{J}^{2}\) implies that \(\ob{}{}(A/\bar{J}^{2},N)=0\)
for all \(B\)-modules \(N\) since \(A/\bar{J}^{2}\ot_{B}N\) gives a lifting of \(N\) to \(A/\bar{J}^{2}\).
\end{ex}
Let \(\cat{C}\) be a category. Then \(\Arr\cat{C}\) denotes the category with objects being arrows in \(\cat{C}\) and arrows being commutative diagrams of arrows in \(\cat{C}\). An endo-functor \(F\) on \(\cat{C}\) induces an endo-functor \(\Arr F\) on \(\Arr\cat{C}\). Let \(B\) be a noetherian local ring and \(\cat{P}_{\!B}\) the additive subcategory of projective modules in \(\cat{mod}_{B}\). Let \(\uhm{}{B}{N}{M}\) denote the homomorphisms from \(N\) to \(M\) in the quotient category \(\umod_{B}=\cat{mod}_{B}/\cat{P}_{\!B}\) i.e.\ \(B\)-homomorphisms modulo the ones factoring through an object in \(\cat{P}_{\!B}\). For each \(N\) in \(\cat{mod}_{B}\) we fix a minimal \(B\)-free resolution and use it to define the syzygy modules of \(N\). For each \(i\) the association \(N\mapsto\syz{B}{i}N\) induces an endo-functor on \(\umod_{B}\) defined by A.\ Heller \cite{hel:60}. Define \(\uxt{i}{B}{N}{M}\) as \(\uhm{}{B}{\syz{B}{i}N}{M}\) which turns out to be isomorphic to \(\xt{i}{B}{N}{M}\) for all \(i>0\).
\begin{lem}\label{lem.obssplit}
Let \(A\) be a noetherian local ring and \(J=(f_{1},\dots,f_{n})\) a regular sequence\textup{.} Put \(B=A/J\) and suppose \(N\)\textup{,} \(N_{1}\) and \(N_{2}\) are finite \(B\)-modules\textup{.} Let \(\bar{M}_{i}\) denote \(B\ot_{A}\syz{A}{n}N_{i}\)\textup{.}
\begin{enumerate}
\item[(i)] There is an inclusion \(u_{N}\co N\ra \bar{M}_{N}\) of \(B\)-modules with \(\bar{M}_{N}\cong B\ot_{A}\syz{A}{n}N\) which induces a functor \(u\co \ul{\cat{mod}}_{B}\ra\Arr\ul{\cat{mod}}_{B}\)\textup{.}
\item[(ii)] The functor \(u\) commutes with the \(B\)-syzygy functor\textup{:} 
\begin{equation*}
\Arr\syz{B}{i}(u_{N})=u_{\syz{B}{i}N}
\end{equation*}
\item[(iii)] The endo-functor \(B\ot_{A}\syz{A}{n}(-)\) induces a natural map \(\uxt{i}{B}{N_{1}}{N_{2}}\ra\uxt{i}{B}{\bar{M}_{1}}{\bar{M}_{2}}\) which makes the following diagram commutative for all \(i\)\textup{:}
\begin{equation*}
\xymatrix@-0pt@C-12pt@R-12pt@H-0pt{
\uxt{i}{B}{N_{1}}{N_{2}}\ar[rr]\ar[dr]_{(u_{N_{2}})_{*}} && \uxt{i}{B}{\bar{M}_{1}}{\bar{M}_{2}}\ar[dl]^{(u_{N_{1}})^{*}}\\
& \uxt{i}{B}{N_{1}}{\bar{M}_{2}} &
}
\end{equation*}
\item[(iv)] The inclusion \(u_{N}\co N\hra B\ot_{A}\syz{A}{n}N\) splits \(\Llra\ob(A/J^{2}\ra B,N)=0\).
\end{enumerate}
\end{lem}
\begin{rem}
Lemma \ref{lem.obssplit} (iv) strengthens \cite[3.6]{aus/din/sol:93} (in the commutative case). 
\end{rem}
\begin{proof}
(i) Suppose \(F_{*}\ra N\) is a minimal \(A\)-free resolution of \(N\). Tensoring the short exact sequence \(0\ra\syz{A}{n}N\xra{j} F_{n-1}\ra\syz{A}{n{-}1}N\ra 0\) with \(B\) gives the exact sequence 
\begin{equation}
0\ra\tor{A}{1}{B}{\syz{A}{n{-}1}N}\ra B\ot\syz{A}{n}N\ra \bar{F}_{n-1}\ra B\ot\syz{A}{n{-}1}N\ra 0\,.
\end{equation}
We have \(\tor{A}{1}{B}{\syz{A}{n{-}1}N}\cong \tor{A}{n}{B}{N}\cong N\). Let \(u_{N}\) be the inclusion \(N\cong\ker(B\ot_{A}j)\ra B\ot\syz{A}{n}N\). Then \(N\mapsto u_{N}\) gives a functor of quotient categories.

(ii) Let \(p\co Q\ra N\) be the minimal \(B\)-free cover and \(P_{*}\ra \syz{B}{}N\) the minimal \(A\)-free resolution of the \(B\)-syzygy \(\ker(p)\). Then there is an \(A\)-free resolution \(H_{*}\ra Q\) which is an extension of \(F_{*}\) by \(P_{*}\). Since \(\syz{A}{n}B\cong A\), tensoring the short exact sequence of \(A\)-free resolutions \(0\ra P_{*}\ra H_{*}\ra F_{*}\ra0\) by \(B\) we obtain by (i) a commutative diagram with exact rows 
\begin{equation}
\xymatrix@C-4pt@R-12pt@H-0pt@M4pt{
0\ar[r] & \syz{B}{}N\ar@{_{(}->}[d]_{u_{\Syz N}}\ar[r] & Q\ar@<0.5ex>@{_{(}->}[d]\ar[r] & N\ar@<0.5ex>@{_{(}->}[d]_{u_{N}}\ar[r] & 0 \\
0\ar[r] & B\ot\syz{A}{n}(\syz{B}{}N)\ar[r] & B^{r}{\oplus}Q\ar[r] & B\ot\syz{A}{n}N\ar[r] & 0
}
\end{equation}
which proves the claim.

(iii) By (ii) it is enough to prove this for \(i=0\). The case \(i=0\) follows from the functoriality in (i).

(iv,\(\La\)) For the case \(n=1\) see the proof of \cite[3.2]{aus/din/sol:93}. Assume \(n\geq 2\). We follow the proof of \cite[3.6]{aus/din/sol:93} closely. Let \(A_{1}=A/(f_{1})\). Then \(F_{*}^{(1)}=A_{1}\ot F_{*\geq 1}[1]\) gives a minimal \(A_{1}\)-free resolution of \(A_{1}\ot\syz{A}{}N\). We have \(\ob(A/J^{2}\ra B,N)=0\Ra \ob(A/(f_{1})^{2}\ra A_{1},N)=0\) and hence \(N\) is a direct summand of \(A_{1}\ot\syz{A}{}N\). Let \(G_{*}\ra N\) be a minimal \(A_{1}\)-free resolution of \(N\). Then \(G_{*}\) is a direct summand of \(F_{*}^{(1)}\) and hence \(\syz{A_{1}}{n{-}1}N\) is a direct summand of \(\syz{A_{1}}{n{-}1}(A_{1}\ot\syz{A}{}N)=A_{1}\ot \syz{A}{n}N\). Tensoring this situation with \(B\) gives a commutative diagram:
\begin{equation}\label{eq.split}
\xymatrix@C-4pt@R-8pt@H-30pt{
N\ar[r]_(0.3){u}\ar@{=}[d] & B\ot\syz{A}{n}N\ar[r]_(0.55){\bar{j}}\ar@<0.5ex>@{->>}[d] & \bar{F}_{n-1}\ar[r]\ar@<0.5ex>@{->>}[d] & \dots \ar[r] & \bar{F}_{1}\ar[r]\ar@<0.5ex>@{->>}[d] & \bar{F}_{0} \ar[r] & N \\
N\ar[r]^(0.3){u_{1}} & B\ot\syz{A_{1}}{n{-}1}N\ar[r]^(0.6){\bar{j}_{1}}\ar@<0.5ex>[u] & \bar{G}_{n-2}\ar[r]\ar@<0.5ex>[u] & \dots\ar[r] & \bar{G}_{0}\ar[r]\ar@<0.5ex>[u] & N &
}
\end{equation}
Since \(\ob(A/J^{2}\ra B,N)=0\Ra \ob(A_{1}/(f_{2},\dots,f_{n})^{2}\ra B,N)=0\) the map \(u_{1}\) splits by induction on \(n\). So \(u\) splits. 
The other direction follows from \cite[3.6]{aus/din/sol:93}.
\end{proof}
\begin{prop}\label{prop.obssplit}
Let \(h\co S\ra T\) be a local Cohen-Macaulay map\textup{,} \(J=(f_{1},\dots,f_{n})\) an \(h\)-sequence\textup{,} \(\bar{h}\co S\ra\bar{T}=T/J\) the local Cohen-Macaulay map induced from \(h\)\textup{,} and \((\bar{h},\mc{N})\) an object in \(\cat{MCM}\)\textup{.} Let \(\xi\co  0\ra\mc{L}\ra \mc{M}\xra{\pi} \mc{N}\ra 0\) be the minimal \(\cat{MCM}\)-approximation of \(\mc{N}\) over \(h\)\textup{.} Then tensoring \(\xi\) by \(\bar{T}\) gives a \(4\)-term exact sequence
\begin{equation*}
0\ra \mc{N}\ot J/J^{2}\lra \bar{\mc{L}}\lra \bar{\mc{M}}\xra{\,\,\bar{\pi}\,\,} \mc{N}\ra 0
\end{equation*}
which represents the obstruction class \(\ob(T/J^{2}\ra\bar{T},\mc{N})\in\xt{2}{\bar{T}}{\mc{N}}{\mc{N}\ot J/J^{2}}\)\textup{.}

Moreover\textup{,} \(\ob(T/J^{2}\ra\bar{T},\mc{N})=0 \Llra \ob(T/J^{2}\ra\bar{T},\mc{N}^{\vee})=0 \Llra \bar{\pi}\) splits where \(\mc{N}^{\vee}=\xt{n}{T}{\mc{N}}{\omega_{h}}\)\textup{.}
\end{prop}
\begin{proof}
By \cite[5.1-2]{ogu/ber:72}, \(\tor{T}{i}{\bar{T}}{\mc{M}}=\cH_{i}(K(\ul{f})\ot \mc{M})=0\) for \(i>0\). There is a map from the defining short exact sequence \(0\ra\syz{T}{}\mc{N}\ra F_{0}\ra \mc{N}\ra 0\) to \(\xi\) lifting \(\id_{\mc{N}}\). Tensoring with \(\bar{T}\) gives a map of \(4\)-term exact sequences with outer terms canonically identified. Hence they represent the same element \(\ob(T/J^{2}\ra\bar{T},\mc{N})\) in \(\xt{2}{\bar{T}}{\mc{N}}{\mc{N}\ot J/J^{2}}\).

By the argument in Remark \ref{rem.Buch} we can assume that \(\xi\) is given as \(0\ra\im(d_{n}^{\vee})\ra(\syz{T}{n}\mc{N}^{\vee})^{\vee}\ra \mc{N}^{\vee}{}^{\vee}\ra0\) where \((F_{*},d_{*})\) is a minimal \(T\)-free resolution of \(\mc{N}^{\vee}\). By Lemma \ref{lem.obssplit}, \(\ob(T/J^{2}\ra\bar{T},\mc{N}^{\vee})=0\) if and only if \(u\co \mc{N}^{\vee}\ra \bar{T}\ot\syz{T}{n}\mc{N}^{\vee}\) splits. But applying \(\hm{}{\bar{T}}{-}{\omega_{\bar{h}}}\) to \(u\) gives \(\bar{\pi}\) since \(\mc{N}\cong\xt{n}{T}{\mc{N}^{\vee}}{\omega_{h}}\cong \hm{}{\bar{T}}{\mc{N}^{\vee}}{\omega_{\bar{h}}}\).
\end{proof}
\begin{rem}\label{rem.obssplit}
In the absolute Gorenstein case with \(n=1\) this is given in \cite[4.5]{aus/din/sol:93}.
\end{rem}
\begin{proof}[Proof of Theorem {\ref{thm.defMCM}}]
Given \(S\) in \(\QH/k\) and let \(h\co S\ra T\) and \(\bar{h}\co S\ra \bar{T}=T/JT\) be the induced hCM maps. Let \({}^{i\!}\mc{N}\) be deformations of \(N\) to \(\bar{h}\) for \(i=1,2\) and assume that the minimal \(\cat{MCM}\)-approximation modules \({}^{i\!}\mc{M}\) of \({}^{i\!}\mc{N}\) over \(h\) are isomorphic as deformations of \(M\). We proceed as in the proof of Theorem \ref{thm.defgrade} (i) with \(S_{n}=S/\fr{m}_{S}^{n{+}1}\), \(\mc{N}_{n}=\mc{N}\ot_{S}S_{n}\) etc., construct a tower of isomorphisms \(\{\phi_{n}\co {}^{1\!}\mc{N}\cong{}^{2\!}\mc{N}\}\), and conclude by Lemma \ref{lem.Aapprox} that \({}^{1\!}\mc{N}\) and \({}^{2\!}\mc{N}\) are isomorphic as deformations of \(N\). For the induction step we use that the map of torsor actions along \(\df{\bar{T}}{\mc{N}_{n}}(S_{n+1})\ra\df{T}{\mc{M}_{n}}(S_{n+1})\) is induced by a natural map \(p\co \xt{1}{B}{N}{N}\ra \xt{1}{A}{M}{M}\) which is injective. The map \(p\) is given as follows.

Let \(\pi\co M\ra N\) denote the \(\cat{MCM}_{A}\)-approximation and \(\bar{\pi}\co \bar{M}\ra N\) be the \(B\)-quotient. Then \(\bar{\pi}\) splits by Proposition \ref{prop.obssplit}. Hence \(\bar{\pi}^{*}\co \xt{1}{B}{N}{N} \ra\xt{1}{B}{\bar{M}}{N}\) splits. Since \(J\) is an \(M\)-regular sequence, \(\xt{1}{B}{\bar{M}}{N}\cong\xt{1}{A}{M}{N}\). Since \(\xt{i}{A}{M}{L}=0\) for all \(i>0\), \(\pi_{*}\co \xt{1}{A}{M}{M}\cong \xt{1}{A}{M}{N}\). Summarised:
\begin{equation}\label{eq.xt1}
\xymatrix@C-0pt@R-12pt@H-30pt{
& \xt{1}{A}{M}{N} & \xt{1}{A}{M}{M} \ar[l]_{\cong}^(0.47){\pi_{*}} \\
\xt{1}{B}{N}{N} \ar@<-0.2ex>@{^{(}->}[r]^{\bar{\pi}^{*}} & \xt{1}{B}{\bar{M}}{N} \ar[u]_{\cong} &
}
\end{equation}
\end{proof}
The technique used to prove Theorem \ref{thm.defMCM} also gives the following result.
\begin{thm}\label{thm.defsyz}
Let \(\vL\) and \({\To}\) be henselian and noetherian local rings and \(q\co \vL\ra {\To}\) a local and flat ring homomorphism with \({\To}/\fr{m}_{{\To}}=k\) and \({\To}\ot_{\vL}k=A\)\textup{.} Suppose \(J=(f_{1},\dots,f_{n})\) is a \(q\)-sequence\textup{.} Put \({\bTo}={\To}/J\)\textup{,} \(B={\bTo}\ot_{\vL}k\) and let \(\bar{J}\) be the image of \(J\) in \(A\)\textup{.} Let \(N\) be a finite \(B\)-module and let \(M\) denote the syzygy module \(\syz{A}{n}N\)\textup{.}
If \(\ob(A/\bar{J}^{2}\ra B, N)=0\) then the natural syzygy map \(s\co \df{{\bTo}}{N}\ra\df{{\To}}{M}\) of functors \(\QH/k\ra\Sets\) is injective\textup{.}
\end{thm}
\begin{proof}
We proceed as in the proof of Theorem \ref{thm.defgrade} and \ref{thm.defMCM}. Given \(S\) in \(\QH/k\) and let \(h\co S\ra {\To}{\hot}_{\vL}S\) denote the algebra. Given deformations \({}^{i\!}\mc{N}\) of \(N\) to \(\bar{h}\co S\ra {\bTo}{\hot}_{\vL}S\) for \(i=1,2\). They map to \({}^{i\!}\mc{M}:=\syz{T}{n}({}^{i\!}\mc{N})\) which we suppose are isomorphic as deformations of \(M\) to \(h\). 
Then the natural syzygy map \(s^{1}\co \xt{1}{B}{N}{N}\ra\xt{1}{A}{M}{M}\) induces the map of torsor actions along \(s\) of the inifinitesimal extensions. The composition of \(s^{1}\) with \(\xt{1}{A}{M}{M}\ra\xt{1}{A}{M}{\bar{M}}\cong\xt{1}{B}{\bar{M}}{\bar{M}}\) commutes with the horizontal map in Lemma \ref{lem.obssplit} (iii). But Lemma \ref{lem.obssplit} (iv) implies that \((u_{N})_{*}\) is injective, hence \(s^{1}\) is injective too. Proceeding by induction on \(\fr{m}_{S}^{n+1}\)-truncations of the deformations we construct a tower of isomorphisms and conclude by Lemma \ref{lem.Aapprox}. 
\end{proof}
\begin{rem}
Theorem \ref{thm.defsyz} resembles \cite[Thm.\ 1]{ile:07}. However Theorem \ref{thm.defsyz} makes a sounder statement in a more general setting and has a more transparent proof. Indeed, the various similar results in \cite{ile:07} can be changed and proved accordingly.
\end{rem}
\section{The Kodaira-Spencer map of Cohen-Macaulay approximations}
A modular family of objects is roughly speaking a family where the isomorphism class of the fibre changes non-trivially. The Kodaira-Spencer map makes this idea precise. We consider the Kodaira-Spencer classes and maps for families of pairs (algebra, module) and by invoking the long exact transitivity sequence we relate them to the corresponding notions for the algebra and the module. Then we show that Cohen-Macaulay approximation of modular families under conditions as in Theorem \ref{thm.defgrade}, \ref{thm.defgrade2} and \ref{thm.defMCM} produce new modular families.

The following is a graded version of \cite[II 2.1.5.7]{ill:71}.
\begin{defn}\label{defn.KS}
Let \(\vL\ra S\) and \(S\ra \vG\) be graded ring homomorphisms with \(\vL\) and \(S\) concentrated in degree \(0\). The map \(L^{\text{gr}}_{\vG/S}\ra L_{S/\!\vL}\ot_{S}\vG[1]\) in the corresponding distinguished transitivity triangle of (graded) cotangent complexes (see \eqref{eq.triangle}) is called the \emph{Kodaira-Spencer class} of \(\vL\ra S\ra\vG\).
\end{defn}
Composing the Kodaira-Spencer class with the natural augmentation map 
\begin{equation*}
L_{S/\!\vL}\ot_{S}\vG[1]\lra\Omega_{S/\!\vL}\ot_{S}\vG[1]
\end{equation*}
induces an element \(\kappa(\vG/S/\!\vL)\in\gH{0}^{1}(S,\vG,\Omega_{S/\!\vL}\ot_{S}\vG)\), the \emph{cohomological Kodaira-Spencer class}, which is also given as follows. Let \(\mc{P}=\mc{P}_{S/\!\vL}\) denote \(S\ot_{\vL}S/I^{2}\) where \(I\) is the kernel of the multiplication map \(S\ot_{\vL}S\ra S\). There are two ring homomorphisms \(j_{1}\) and \(j_{2}\) from \(S\) to \(\mc{P}\) defined by \(j_{1}\co s\mapsto s\ot 1\) and \(j_{2}\co s\mapsto 1\ot s\). Let \(d_{S/\!\vL}\co S\ra \Omega_{S/\!\vL}=I/I^{2}\) denote the universal derivation (induced by \(j_{2}-j_{1}\)) and \(\gamma_{S/\!\vL}\co \Omega_{S/\!\vL}\ra\mc{P}_{S/\!\vL}\) the inclusion. The \emph{principal parts of \(\vG\)} is \(\mc{P}\ot_{S}\vG\) which is an \(S\)-algebra via \(j_{1}\ot 1_{\vG}\co S\ra \mc{P}\ot_{S}\vG\). The multiplication map induces the \(S\)-algebra extension representing the Kodaira-Spencer class:
\begin{equation}\label{eq.pp}
\kappa(\vG/S/\!\vL)\co \quad0\ra \Omega_{S/\!\vL}\ot_{S}\vG\xra{\gamma_{S/\!\vL}\ot\id_{\vG}} \mc{P}_{S/\!\vL}\ot_{S}\vG\lra \vG \ra 0
\end{equation}
see \cite[III 1.2.6]{ill:71}. Since \(\mc{P}\ot_{S}\vG\) has a natural \(\mc{P}\)-algebra structure, \eqref{eq.pp} is also a (graded) algebra lifting of \(\vG\) along \(\mc{P}\ra S\) as in Definition \ref{def.grobs}. Let \(\vG\ot_{S}\mc{P}\) denote the \(S\)-algebra defined via the \(j_{1}\) tensor product. The extension \(\vG\ot_{S}\mc{P}\ra \vG\) is a trivial lifting (split by \(\id_{\vG}\ot 1_{\mc{P}} \)) and the `difference' of the \(\mc{P}\)-algebras \(\mc{P}\ot_{S}\vG\) and \(\vG\ot_{S}\mc{P}\), an element in \(\gH{0}^{1}(S,\vG,\Omega_{S/\!\vL}\ot_{S}\vG)\) given by Proposition \ref{prop.grobs} (ii), equals \(\kappa(\vG/S/\!\vL)\); see \cite[III 2.1.5]{ill:71}. 
Moreover, the difference \(1_{\mc{P}}\ot_{j_{2}}\!\id_{\vG}-\id_{\vG}\!\ot _{j_{1}}\!1_{\mc{P}}\) induces \(d_{S/\!\vL}\ot 1_{T}\) (in degree \(0\)) which is mapped to \(\kappa(\vG/S/\!\vL)\) by the connecting homomorphism \(\partial\) in the long exact transitivity sequence of \(\vL\ra S\ra \vG\);
\begin{equation}\label{eq.let}
\gDer{0}_{\vL}(\vG,\Omega_{S/\!\vL}\ot_{S}\vG)\lra\Der_{\vL}(S,\Omega_{S/\!\vL}\ot_{S}T)\xra{\,\partial\,}\gH{0}^{1}(S,\vG,\Omega_{S/\!\vL}\ot_{S}\vG)
\end{equation} 
see \cite[III 1.2.6.5 and 1.2.7]{ill:71}. Hence \(\kappa(\vG/S/\!\vL)=0\) if and only if there is a (graded) derivation \(\mc{D}\) in \(\gDer{0}_{\vL}(\vG,\Omega_{S/\!\vL}\ot_{S}\vG)\) which maps to \(d_{S/\!\vL}\ot 1_{T}\). Then 
\begin{equation}
\sigma:=j_{2}\ot\id_{\vG}-(\gamma_{S/\!\vL}\ot\id_{\vG})\mc{D}\co \vG\ra \mc{P}\ot_{S}\vG
\end{equation}
is an \(S\)-algebra map since \(\sigma_{\vert S}=j_{1}\ot 1_{T}\), splitting \(\mc{P}\ot_{S}\vG\ra\vG\). Conversely, such a splitting \(\sigma\) induces via \(j_{2}\ot\id_{\vG}-\,\sigma\) a derivation lifting \(d_{S/\!\vL}\ot 1_{T}\).

In the special case \(\vG=T{\oplus}\mc{N}\) and \(\mc{N}\) is a \(T\)-module the transitivity sequence of \(\vL\ra T\ra\vG\) in Proposition \ref{prop.lang} gives \(\xt{1}{T}{\mc{N}}{\Omega_{T\!/\!\vL}\ot_{T}\mc{N}}\xra{\simeq}\gH{0}^{1}(T,\vG,\Omega_{T\!/\!\vL}\ot_{T}\vG)\). The Kodaira-Spencer class \(\kappa(\vG/T/\!\vL)\) in the latter corresponds to the (cohomological) \emph{Atiyah class} \(\at_{T\!/\!\vL}(\mc{N})\in\xt{1}{T}{\mc{N}}{\Omega_{T\!/\!\vL}\ot_{T}\mc{N}}\); cf.\ \cite[IV 2.3.6-7]{ill:71}. The class is represented by the short exact sequence of \(T\)-modules
\begin{equation}\label{eq.at}
\at_{T\!/\!\vL}(\mc{N})\co \quad 0\ra\Omega_{T\!/\!\vL}\ot_{T}\mc{N}\lra\mc{P}_{T\!/\!\vL}\ot_{T}\mc{N}\lra\mc{N}\ra 0\,.
\end{equation}

The \emph{Kodaira-Spencer map of \(\vL\ra S\ra\vG\)} 
\begin{equation}\label{eq.KSmap}
g^{\vG}\!\co \Der_{\vL}(S)\lra\gH{0}^{1}(S,\vG,\vG)
\end{equation}
is defined by \(D\mapsto f^{D}_{*}\kappa(\vG/S/\!\vL)\) where \(f^{D}\!\co \Omega_{S/\!\vL}\ra S\) corresponds to \(D\). Pushout of \eqref{eq.pp} by \(f^{D}\ot\id_{\vG}\) gives the corresponding algebra lifting of \(\vG\) along \(S[\vare]\ra S\) given by \(g^{\vG}(D)\); see Proposition \ref{prop.grobs} (ii).
\begin{prop}\label{prop.at}
Let \(\vG\) denote the graded \(S\)-algebra \(T{\oplus}\mc{N}\) where \(\vL\ra S\) and \(S\ra T\) are \textup{(}ungraded\textup{)} ring homomorphisms and \(\mc{N}\) is a \(T\)-module\textup{.} Consider the transitivity sequence of \(S\ra T\xra{i} \vG\) in \textup{Proposition \ref{prop.lang}:}
\begin{equation*}
\dots \xra{\,\partial\,}\xt{1}{T}{\mc{N}}{\Omega_{S/\!\vL}\ot_{S}\mc{N}}\xra{\,u\,}\gH{0}^{1}(S,\vG,\Omega_{S/\!\vL}\ot_{S}\vG)\xra{i^{*}}
\cH^{1}(S,T,\Omega_{S/\!\vL}\ot_{S}T) \xra{\,\partial\,}\dots
\end{equation*}
\begin{enumerate}
\item[(i)] The map \(i^{*}\) takes the Kodaira-Spencer class \(\kappa(\vG/S/\!\vL)\) to \(\kappa(T/S/\!\vL)\)\textup{.}
\item[(ii)] Assume \(\kappa(T/S/\!\vL)=0\) and choose an \(S\)-algebra splitting \(\sigma\co  T\ra \mc{P}_{S/\!\vL}\ot_{S}T\)\textup{.} There is an element 
\begin{equation*}
\kappa(\sigma,\mc{N})\in\xt{1}{T}{\mc{N}}{\Omega_{S/\!\vL}\ot_{S}\mc{N}},
\end{equation*}
appropriately natural in \((T/S/\!\vL,\sigma,\mc{N})\)\textup{,} which \(u\) maps to \(\kappa(\vG/S/\!\vL)\)\textup{.} Moreover\textup{,} \(\kappa(\sigma,\mc{N})\) is represented by the degree \(1\) part of \eqref{eq.pp} 
\begin{equation*}
\kappa(\sigma,\mc{N})\co\quad0\ra \Omega_{S/\!\vL}\ot_{S}\mc{N}\lra\mc{P}_{S/\!\vL}\ot_{S}\mc{N}\lra\mc{N}\ra 0
\end{equation*}
considered as a short exact sequence of \(T\)-modules via \(\sigma\)\textup{.}
\item[(iii)] Let \(D(\sigma)\) in \(\Der_{\vL}(T,\Omega_{S/\!\vL}\ot_{S}T)\) be the derivation induced by \(\sigma\) and \(f^{D(\sigma)}\) in \(\hm{}{T}{\Omega_{T/\!\vL}}{\Omega_{S/\!\vL}\ot_{S}T}\) the corresponding homomorphism\textup{.} Then 
\begin{equation*}
\kappa(\sigma,\mc{N})=(f^{D(\sigma)}\ot\id_{\mc{N}})_{*}\at_{T/\!\vL}(\mc{N}).
\end{equation*}
For each \(Y\in\Der_{\vL}(S)\) let \(X_{\sigma}(Y)\) denote \((f^{Y}\!\ot\id_{T})_{*}D(\sigma)\in\Der_{\vL}(T)\)\textup{.} Then
\begin{equation*}
(f^{Y}\!\ot\id_{\mc{N}})_{*}\kappa(\sigma,\mc{N})=(f^{X_{\sigma}(Y)}\ot\id_{\mc{N}})_{*}\at_{T\!/\!\vL}(\mc{N})\:\:\text{in}\:\: \xt{1}{T}{\mc{N}}{\mc{N}}.
\end{equation*}
\end{enumerate}
\end{prop}
\begin{proof}
(i) The degree zero part of \eqref{eq.pp} gives the image \(i^{*}\kappa(\vG/S/\!\vL)\) represented by the algebra extension
\begin{equation}
\kappa(T/S/\!\vL)\co \quad 0\ra\Omega_{S/\!\vL}\ot_{S} T\lra \mc{P}_{S/\!\vL}\ot_{S}T\lra T\ra 0\,.
\end{equation}

(ii) The degree one part of \eqref{eq.pp} is \emph{a priori} a short exact sequence of \(\mc{P}_{S/\!\vL}\ot_{S}T\)-modules.
The splitting \(\sigma\) makes it to a short exact sequence of \(T\)-modules which defines \(\kappa(\sigma,\mc{N})\). The naturality follows from the naturality of \(\kappa(\vG/S/\!\vL)\).  
 
(iii) We define a \(T\)-algebra homomorphism \(h^{\sigma}\co \mc{P}_{T/\!\vL}\ra\mc{P}_{S/\!\vL}\ot_{S}T\) by \(h^{\sigma}(t_{1}\ot t_{2})=\sigma(t_{1})\cdot1_{\mc{P}}\ot t_{2}\) which we claim makes the following diagram commutative:
\begin{equation}
\xymatrix@C+18pt@R-6pt@H-30pt{
\at_{T\!/\!\vL}(\mc{N})\co\!\!\!\!\!\!\!\!\!\!\!\!\!\!\!\!\!\!\!\!  & 0\lra\Omega_{T\!/\!\vL}\ot_{T}\mc{N}\ar[r]^(0.55){\gamma_{T/\!\vL}\ot\id_{\mc{N}}}\ar@<3.6ex>[d]_(0.46){f^{D(\sigma)}\ot\id_{\mc{N}}} & \mc{P}_{T\!/\!\vL}\ot_{T}\mc{N}\ar[r]\ar@<0.5ex>[d]_(0.46){h^{\sigma}\ot\id_{\mc{N}}} &\mc{N}\lra 0\ar@<-3.1ex>@{=}[d] \\
\kappa(\sigma,\mc{N})\co\!\!\!\!\!\!\!\!\!\!\!\!\!\!\!\!\!\!\!\!\!\!\!\!  & 0\lra \Omega_{S/\!\vL}\ot_{S}\mc{N}\ar[r]_(0.55){\gamma_{S/\!\vL}\ot\id_{\mc{N}}} & \mc{P}_{S/\!\vL}\ot_{S}\mc{N}\ar[r] &\mc{N}\lra 0
}
\end{equation}
For \(t\) in \(T\) we have by the definitions 
\begin{equation}
h^{\sigma}\gamma_{T/\!\vL}d_{T/\!\vL}(t)=h^{\sigma}(1_{T}\ot t-t\ot 1_{T})=1_{\mc{P}}\ot t-\sigma(t)=(\gamma_{S/\!\vL}\ot 1_{T})D(\sigma)(t)
\end{equation}
and the claim follows. Then the pushout of \(\at_{T/\!\vL}(\mc{N})\)  by \(f^{D(\sigma)}\ot\id_{\mc{N}}\) gives \(\kappa(\sigma,\mc{N})\). Since \((f^{Y}\ot\id_{\mc{N}})\circ(f^{D(\sigma)}\ot\id_{\mc{N}})\) equals \(f^{X_{\sigma}(Y)}\ot\id_{\mc{N}}\) the second part of (iii) follows from the first. 
\end{proof}
We name \(\kappa(\sigma,\mc{N})=\kappa(T/S/\!\vL,\sigma,\mc{N})\) the \emph{Kodaira-Spencer class of \((T/S/\!\vL,\sigma,\mc{N})\)}.
Define the \emph{Kodaira-Spencer map of \((T/S/\!\vL,\sigma,\mc{N})\)}
\begin{equation}\label{eq.modKS}
g^{(\sigma,\,\mc{N})}\co \Der_{\vL}(S)\lra\xt{1}{T}{\mc{N}}{\mc{N}}
\end{equation}
by \(g^{(\sigma,\,\mc{N})}(D):=(f^{D}\ot\id)_{*}\kappa(\sigma,\mc{N})\).

In the case \(T=S\ot_{\vL} {\To}\) we always choose the \(S\)-algebra splitting \(S\ot_{\vL} {\To}\ra \mc{P}_{S/\!\vL}\ot_{S}S\ot_{\vL}{\To}\cong\mc{P}_{S/\!\vL}\ot_{\vL}{\To}\) given by \(s\ot t\mapsto j_{1}(s)\ot t\). In particular \(\kappa(T/S/\!\vL)=0\) and we get a canonical Kodaira-Spencer class \(\kappa(\mc{N})\) and a corresponding Kodaira-Spencer map \(g^{\mc{N}}\).
\begin{rem}
As indicated after \eqref{eq.let}, \(\kappa(T/S/\!\vL)=0\) if and only if the canonical map \(\Omega_{S/\!\vL}\ot_{S}T\ra\Omega_{T/\!\vL}\) is splitt injective. By \cite[{\bf{0}} 20.5.7]{EGAIV01} the latter is equivalent to \(T\) being a \emph{formally smooth \(S\)-algebra relative to \(\vL\)} \textup{(}with discrete topology\textup{)}; cf.\ \cite[{\bf{0}} 19.9.1]{EGAIV01}, or, equivalently: All \(\vL\)-split \(S\)-algebra extensions \(0\ra M\ra T'\ra T\ra0\) with \(M^{2}=0\) are trivial; cf.\ \cite[{\bf{0}} 19.9.8.1]{EGAIV01}. In particular, if \(\Spec T\ra \Spec S\) is smooth (\cite[17.3.1]{EGAIV4}) then \(\kappa(T/S/\!\vL)=0\).

By Proposition \ref{prop.at} (iii) we have a commutative diagram
\begin{equation}
\xymatrix@C+15pt@R-6pt@H+3pt{
d_{T/\!\vL}\in\Der_{\vL}(T,\Omega_{T/\!\vL})\ar[r]^(0.41){\partial(\vG/T/\!\vL)}\ar@<2ex>@{->>}[d]_(0.45){f^{D(\sigma)}_{*}} & \xt{1}{T}{\mc{N}}{\Omega_{T/\!\vL}\ot_{T}\mc{N}}\ni\at_{T/\!\vL}(\mc{N})\ar@<-4ex>@{->>}[d]_(0.45){(f^{D(\sigma)}\ot\id_{\mc{N}})_{*}} \\
D(\sigma)\in\Der_{\vL}(T,\Omega_{S/\!\vL}\ot_{S}T)\ar@<-3ex>[u]_{\text{split}}\ar[r]^(0.46){\partial(\vG/T/\!\vL)} & \xt{1}{T}{\mc{N}}{\Omega_{S/\!\vL}\ot_{S}\mc{N}}\ni\kappa(\sigma,\mc{N})\ar@<3ex>[u]_{\text{split}}
}
\end{equation}
where the `down' diagram is pointed. In particular we have \(\partial(D(\sigma))=\kappa(\sigma,\mc{N})\).
\end{rem}
\begin{rem}
The transitivity sequence of \(\vL\ra T\xra{i}\vG\) with \(\vG=T{\oplus}\mc{N}\) and \(J=J_{0}{\oplus}J_{1}\) in Proposition \ref{prop.lang} 
\begin{equation}\label{eq.deg0}
0\ra \hm{}{T}{\mc{N}}{J_{1}}\ra \gDer{0}_{\vL}(\vG,J)\xra{i^{*}} \Der_{\vL}(T,J_{0})\ra \xt{1}{T}{\mc{N}}{J_{1}}\ra\dots
\end{equation}
suggests the following characterisation. An element \(\mc{D}\in\gDer{0}_{\vL}(\vG,J)\) is given by its degree \(0\) restriction \(D:=i^{*}(\mc{D})\in\Der_{\vL}(T,J_{0})\) and its degree \(1\) restriction \(\nabla_{\!D}:=\mc{D}_{\vert\mc{N}}\in\hm{}{\vL}{\mc{N}}{J_{1}}\) which should satisfy the following Leibniz rule: For all \(t\) in \(T\) and \(n\) in \(\mc{N}\) 
\begin{equation}\label{eq.Leib}
\nabla_{\!D}(tn)=t\nabla_{\!D}(n)+D(t)n\,.
\end{equation}
Recall that \(\kappa(\vG/S/\!\vL)=\partial(d_{S/\!\vL}\ot 1_{T})\) in the transitivity sequence \eqref{eq.let}.
Hence \(\kappa(\vG/S/\!\vL)=0\) if and only if there exists a \(D\in\Der_{\vL}(T,\Omega_{S/\!\vL}\ot T)\) which restricts to \(d_{S/\!\vL}\ot 1_{T}\) and a \(\nabla_{\!D} \in\hm{}{\vL}{\mc{N}}{\Omega_{S/\!\vL}\ot \mc{N}}\) satisfying \eqref{eq.Leib}. As a well known special case (\(S=T\)) we get \(\at_{T\!/\!\vL}(\mc{N})=0\) if and only if there exists a \(\nabla\in \hm{}{\vL}{\mc{N}}{\Omega_{T\!/\!\vL}\ot \mc{N}}\) satisfying \eqref{eq.Leib} with \(D=d_{T/\!\vL}\in\Der_{\vL}(T,\Omega_{T\!/\!\vL})\) (i.e.\ \(\nabla\) is a connection), or equivalently, there is a (graded) derivation \(\mc{D}\in \gDer{0}_{\vL}(\vG,\Omega_{T\!/\!\vL}\ot_{T}\vG)\) restricting to \(d_{T/\!\vL}\).
\end{rem}
Recall the maps of cohomology groups \(\sigma^{1}_{j}(I)\) and \(\tau^{1}_{j}(I)\) in \eqref{eq.sigma} and \eqref{eq.tau}.
\begin{prop}\label{prop.KS}
In addition to the assumptions in \textup{Lemma \ref{lem.cohmap}} suppose \(\vL\ra S\) is a ring homomorphism\textup{.} For \(j=1,2\) the following holds\textup{:}
\begin{enumerate}
\item[(i)] The map \(\sigma^{1}_{j}(\Omega_{S/\!\vL})\) takes \(\kappa(\vG_{0}/S/\!\vL)\) to \(\kappa(\vG_{j}/S/\!\vL)\) and the Kodaira-Spencer maps \(g^{\vG_{i}}\co \Der_{\vL}(S)\ra\gH{0}^{1}(S,\vG_{i},\vG_{i})\) commute with \(\sigma^{1}_{j}\)\textup{,} i\textup{.}e\textup{.}\ \(\sigma^{1}_{j} g^{\vG_{0}}=g^{\vG_{j}}\)\textup{.}
\item[(ii)] Assume \(\kappa(T/S/\!\vL)=0\) and choose an \(S\)-algebra splitting \(\sigma\co  T\ra \mc{P}_{S/\!\vL}\ot_{S}T\)\textup{.}
Then \(\tau^{1}_{j}(\Omega_{S/\!\vL})\) maps \(\kappa(\sigma,\mc{N})\) to \(\kappa(\sigma,X_{j})\) and the Kodaira-Spencer maps \(g^{(\sigma,X_{i})}\co \Der_{\vL}(S)\ra\xt{1}{T}{X_{i}}{X_{i}}\) commute with \(\tau^{1}_{j}\)\textup{,} i.e.\ \(\tau^{1}_{j} g^{(\sigma,\mc{N})}\!=g^{(\sigma,X_{j})}\)\textup{.}
\end{enumerate}
\end{prop}
\begin{proof}
(i) Put \(\kappa_{j}=\kappa(\vG_{j}/S/\!\vL)\), \(\Omega=\Omega_{S/\!\vL}\) and let \(\vG(\iota)\co \vG_{0}\ra\vG_{2}\) denote the graded ring homomorphism induced from \(\iota\). Then \(\vG(\iota)\) induces a map of short exact sequences \(\kappa_{0}\ra\kappa_{2}\), hence a map of short exact sequences \(\vG(\iota)_{*}\kappa_{0}\ra\vG(\iota)^{*}\kappa_{2}\), i.e.\ \(\sigma_{2}(\Omega)(\kappa_{0})=(\vG(\iota)^{*})^{-1}\vG(\iota)_{*}\kappa_{0}=\kappa_{2}\). The maps \(\sigma_{2}(\Omega)\) and \(\sigma_{2}(S)\) commute with the covariant action of \(\Der_{\vL}(S)\), hence the second assertion follows from the first. The arguments for the cases \(j=1\) and (ii) are similar.
\end{proof}

There are corresponding \emph{local} Kodaira-Spencer maps given as follows. Suppose \(S\ra T\) is a local ring homomorphism, put \(k=S/\fr{m}_{S}\) and let \(p\co S\ra k\) denote the quotient. By abuse of notation we let \(p\) also denote the quotient map \(\vG\ra \vG\ot_{S}k=:G\). Let \(A\) denote \(T\ot_{S}k\). Assume \(\vG\) is \(S\)-flat. Then the canonical map 
\begin{equation}\label{eq.g0}
\gH{0}^{1}(k,G,M)\lra\gH{0}^{1}(S,\vG,M)
\end{equation}
is an isomorphism for any \(G\)-module \(M\); see \cite[II 2.2.3]{ill:71}. With this identification we define \(g^{\vG}\!(0)(D):=(f^{D}\ot\, p)_{*}\kappa(\vG/S/\!\vL)\) for any \(D\in\Der_{{\vL}}(S,k)\cong\hm{}{S}{\Omega_{S/\!\vL}}{k}\ni f^{D}\) and obtain local Kodaira-Spencer maps of \(\vG\), and (similarly) of \(T\), respectively:  
\begin{align}
g^{\vG}\!(0)\co \,\,&\Der_{\vL}(S,k)\lra\gH{0}^{1}(k,G,G) \\
g^{T}\!(0)\co \,\,&\Der_{\vL}(S,k)\lra\cH^{1}(k,A,A)
\end{align}
Define the \emph{local Kodaira Spencer class} in \(\gH{0}^{1}(k,G,\Omega_{S/\!\vL}\ot_{S}G)\) as \(\bar{\kappa}(\vG/S/\!\vL)=(\id\ot\,p)_{*}\kappa(\vG/S/\!\vL)\). With the identification \eqref{eq.g0} we obtain the following pointed commutative diagram which relates the global and the local Kodaira-Spencer maps:
\begin{equation}
\xymatrix@C+12pt@R-6pt@H-15pt{
\kappa\in\gH{0}^{1}(S,\vG,\Omega_{S/\!\vL}\ot_{S}\vG)\ar[r]^(0.48){(\id\ot\,p)_{*}}\ar@<3ex>[d]_(0.45){(f^{D}\ot\id)_{*}} & \gH{0}^{1}(S,\vG,\Omega_{S/\!\vL}\ot_{S}G)\ni\bar{\kappa}\ar@<-4ex>[d]^(0.45){(f^{p\!D}\ot\id)_{*}} \\
g^{\vG}(D)\in\gH{0}^{1}(S,\vG,\vG)\ar[r]^(0.47){p_{*}} & \gH{0}^{1}(S,\vG,G)\ni g^{\vG}\!(0)(pD) \\
D\in\Der_{{\vL}}(S,S)\ar[r]^(0.47){p_{*}}\ar@<-3ex>[u]^{g^{\vG}} & \Der_{{\vL}}(S,k)\ni pD\ar@<4ex>[u]_{g^{\vG}\!(0)}
}
\end{equation}
Assume \(\vG=T{\oplus}\mc{N}\) and let \(N\) denote the \(A\)-module \(\mc{N}\ot_{S}k\). There is a graded algebra extension representing \(\bar{\kappa}(\vG/S/\!\vL)\) which in degree \(1\) is a short exact sequence 
\begin{equation}\label{eq.lKS}
0\ra\Omega_{S/\!\vL}\ot_{S} N\lra k\ot_{S}\mc{P}_{S/\!\vL}\ot_{S} N\lra N\ra0.
\end{equation}
If \(\bar{\kappa}(T/S/\!\vL)=0\) we choose an \(S\)-algebra splitting \(\sigma\co A\ra k\ot_{S}\mc{P}_{S/\!\vL}\ot_{S} A\), (and the canonical one if \(T=S\ot_{\vL}{\To}\)). The \emph{local Kodaira-Spencer class} \(\bar{\kappa}(\sigma,\mc{N})\) in \(\xt{1}{A}{N}{\Omega_{S/\!\vL}\ot_{S}N}\) is represented by \eqref{eq.lKS} as a short exact sequence of \(A\)-modules. Then we define the \emph{local Kodaira-Spencer map of \((T/S/\!\vL,\sigma,\mc{N})\)}
\begin{equation}
g^{(\sigma,\mc{N})}\!(0)\co \,\,\Der_{\vL}(S,k)\lra\xt{1}{A}{N}{N}
\end{equation}
 by \(g^{(\sigma,\mc{N})}(0)(D):=(f^{D}\ot\id)_{*}\bar{\kappa}(\sigma,\mc{N})\).

We assume that \(\vL\) is an algebraically closed field \(k\) for the rest of this section. Then \(\Der_{k}(S,k(s))\) is canonically isomorphic to the Zariski tangent space at any closed point \(s\in\Spec S\).
\begin{defn}\label{defn.modular} 
The pair \emph{\((h\co S\ra T,\mc{N})\) is locally modular} if the local Kodaira-Spencer map \(g^{\vG}\!(0)\co \Der_{k}(S,k)\ra\gH{0}^{1}(k,G,G)\) is injective. If in addition \(T=(S\ot_{k}A)_{\fr{m}}\) then \emph{\(\mc{N}\) is locally modular} if \(g^{\mc{N}}\!(0)\co \Der_{k}(S,k)\ra\xt{1}{A}{N}{N}\) is injective.

Let \(h^{\text{ft}}\co S\ra T\) be a faithfully flat finite type map of noetherian \(k\)-algebras with a \(k\)-point \(t\in\Spec T\) and \(\mc{N}\) is an \(S\)-flat finite \(T\)-module. We say that \emph{\((h^{\textnormal{ft}},\mc{N})\) is modular at \(t\)} if the localisation of \((h^{\text{ft}},\mc{N})\) at \(t\) is locally modular. 
If \(A\) is a finite type \(k\)-algebra and \(T=S\ot_{k} A\), then \emph{\(\mc{N}\) is modular at \(t\) as \(T\)-module} if its localisation at \(t\) is locally modular. Let \(\nabla(h,\mc{N})\) (\(\nabla_{T}(\mc{N})\) if \(T=S\ot_{k}A\)) denote the set of \(k\)-points \(t\in \Supp\mc{N}\) where \((h,\mc{N})\) (respectively \(\mc{N}\) as \(T\)-module) is modular.
\end{defn}
If \(t\in \Spec T\) maps to \(s\in\Spec S\) let \(T(t)\) denote the \(k(s)\)-algebra \(T_{\fr{p}_{t}}\ot_{S_{\fr{p}_{s}}}k(s)\) and \(\mc{N}(t)\) the \(T(t)\)-module \(\mc{N}_{\fr{p}_{t}}\ot_{S_{\fr{p}_{s}}}k(s)\).
\begin{cor}\label{cor.modY}
Let \(h\co S\ra T\) be a finite type Cohen-Macaulay map of noetherian \(k\)-algebras and let \(\mc{N}\) be an \(S\)-flat finite \(T\)-module\textup{.} Let \(0\ra\mc{N}\xra{\iota}\mc{L}'\ra\mc{M}'\ra0\) and \(0\ra\mc{L}\ra\mc{M}\ra\mc{N}\ra0\) be a \(\Df\)-hull and a \(\cat{MCM}\)-approximation for \(\mc{N}\)\textup{,} respectively\textup{.}
\begin{enumerate}
\item[(i)] Suppose \(\hm{}{T(t)}{\mc{N}(t)}{\mc{M}'(t)}=0\) for all \(t\in \mSpec T\)
\textup{.}
Then 
\begin{equation*}
\nabla(h,\mc{N})=\nabla(h,\mc{L}')\quad\text{and}\quad \nabla_{T}(\mc{N})=\nabla_{T}(\mc{L}')\,\,\text{if}\,\, T=S\ot_{k}A\,.
\end{equation*}
\item[(ii)] Suppose \(\hm{}{T(t)}{\mc{L}(t)}{\mc{N}(t)}=0\) for all \(t\in \mSpec T\)\textup{.}
Then 
\begin{equation*}
\nabla(h,\mc{N})=\nabla(h,\mc{M})\quad\text{and}\quad \nabla_{T}(\mc{N})=\nabla_{T}(\mc{M})\,\,\text{if}\,\, T=S\ot_{k}A\,.
\end{equation*}
\end{enumerate}
\end{cor}
\begin{proof}
(i) Let \(t\) be a \(k\)-point in \(\Spec T\) mapping to \(s\) in \(\Spec S\). By Theorem \ref{thm.flatCMapprox} \(\iota(t)\) is a \(\hat{\cat{D}}_{T(t)}\)-hull for \(\mc{N}(t)\) and by Proposition \ref{prop.minapprox} \(\iota(t)\) is minimal if and only if \(\iota_{\fr{p}_{t}}\) is minimal. In particular, the minimal hull of 
\(\mc{N}_{\fr{p}_{t}}\) is a direct summand of \(\mc{L}'_{\fr{p}_{t}}\). We therefore assume that \(\iota_{\fr{p}_{t}}\) and hence \(\iota(t)\) is minimal.
Put \(S_{\fr{p}_{s}}=S'\), \(T_{\fr{p}_{t}}=T'\), \(\mc{N}_{\fr{p}_{t}}=\mc{N}'\), \(\mc{L}'_{\fr{p}_{t}}=\mc{L}'\), \((\vG_{j})_{\fr{p}_{t}}=\vG_{j}'\) and \(\kappa(\vG_{j}'/S'\!/k)=\kappa_{j}\). Note that \(\gH{0}^{1}(S',\vG_{j}',\vG_{j}'(t))\cong\gH{0}^{1}(k(s),\vG'_{j}(t),\vG'_{j}(t))\) and \(\xt{1}{T'}{\mc{N}'}{\mc{N}'(t)}\cong \xt{1}{T'(t)}{\mc{N}'(t)}{\mc{N}'(t)}\) since \(\vG'_{j}\), \(T'\) and \(\mc{N}'\) are \(S'\)-flat.
Hence the map \(\tau^{1}_{2}(t)\co \xt{1}{T'}{\mc{N}'}{\mc{N}'(t)}\ra \xt{1}{T'}{\mc{L}'}{\mc{L}'(t)}\) in \eqref{eq.tau} is injective by assumption. It implies that the map \(\sigma^{1}_{2}(t)\co \gH{0}^{1}(S',\vG_{0}',\vG_{0}(t))\ra\gH{0}^{1}(S',\vG_{2}',\vG_{2}(t))\) in \eqref{eq.sigma} is injective. The following diagram, which relates local and global maps, is pointed commutative by Proposition \ref{prop.KS}:
\begin{equation}
\xymatrix@C+12pt@R-6pt@H-15pt{
\kappa_{0}\in\gH{0}^{1}(S',\vG_{0}',\Omega_{S'\!/k}\ot_{S}\vG_{0}')\ar[r]^(0.48){\sigma_{2}^{1}(\Omega_{S'\!/k})}\ar@<3ex>[d]_(0.43){(f^{D}\ot\,p)_{*}} & \gH{0}^{1}(S',\vG_{2}',\Omega_{S'\!/k}\ot_{S'}\vG_{2}')\ni\kappa_{2}\ar@<-4ex>[d]^(0.43){(f^{D}\ot\,p)_{*}} \\
g^{\vG'_{0}}(t)(D)\in\gH{0}^{1}(S',\vG_{0}',\vG_{0}'(t))\ar[r]^(0.49){\sigma_{2}^{1}(t)} & \gH{0}^{1}(S',\vG_{2}',\vG_{2}'(t))\ni g^{\vG_{2}'}(t)(D) \\
D\in\Der_{k}(S',k)\ar@{=}[r]\ar@<-3ex>[u]^{g^{\vG'_{0}}(t)} & \Der_{k}(S',k)\ni D\ar@<4ex>[u]_{g^{\vG_{2}'}(t)}
}
\end{equation}
Statement (i) follows and (ii) is similar.
\end{proof}
\begin{ex}\label{ex.modY}
Let \(A\) be a CM finite type \(k\)-algebra and domain of dimension \(\geq 1\). Let \(h\co A\ra T=A^{\ot 2}\) be the base change by \(S=A\) and \(\mc{N}=A\) be the \(A\)-flat \(T\)-module defined by the multiplication map \(T\ra A\). Let \(\Delta\sbeq \Spec T\) denote the closed points on the diagonal and let \(t\) be a closed point in \(\Spec T\) mapping to \(s\) in \(\Spec A\). If \(t\notin\Delta\) then \(\mc{N}(t)=0\). If \(t\in \Delta\) then \(T(t)\cong A_{\fr{p}_{s}}=: A'\) and \(\mc{N}(t)\cong k(s)\). The local Kodaira-Spencer class \(\bar{\kappa}(\mc{N}(t))\in \xt{1}{A'}{k(s)}{\Omega_{A'\!/k}\ot_{A'} k(s)}\) is represented by 
\begin{equation}
\xymatrix@C-0pt@R-12pt@H-30pt{
0\ar[r] & k(s)\ot_{A'}\Omega_{A'\!/k}\ar[r] & k(s)\ot_{A'} \mc{P}_{A'\!/k}\ar[r] & k(s)\ar[r] & 0 \\
0\ar[r] & \fr{m}/\fr{m}^{2}\ar[r]\ar[u]^{\cong}_{\delta} & A'/\fr{m}^{2}\ar[r] \ar[u]^{\cong}_{\chi} & k(s)\ar@{=}[u]\ar[r] & 0
}
\end{equation}
with \(\fr{m}=\fr{p}_{s}A'\); see \eqref{eq.lKS}. Here \(\delta(\bar{x})=1\ot d_{A'}(x)\) and \(\chi\) is induced by \(1\ot j_{2}\) (if \(x, y\in\fr{m}\) then \(1\ot j_{2}(xy)=1\ot([j_{2}(x)-j_{1}(x)][(j_{2}(y)-j_{1}(y)])=0\)). The local Kodaira-Spencer map \(g^{\mc{N}}(t)\) is given by the pushout
\begin{equation}
\phi\in\hm{}{k(s)}{\fr{m}/\fr{m}^{2}}{k(s)}\lra \xt{1}{A'}{k(s)}{k(s)}\ni\phi_{*}\bar{\kappa}(\mc{N}(t)).
\end{equation}
It is an isomorphism. By Corollary \ref{cor.modY} we have \(\nabla_{T}(\mc{N})=\nabla_{T}(\mc{L}')=\Delta\). Put \(\mc{Q}'=\hm{}{T}{\omega_{h}}{\mc{L}'}\). By Proposition \ref{prop.defequiv} also the local Kodaira-Spencer map \(g^{\mc{Q}'}(t)\) is injective for \(t\in \Delta\). Hence \(\nabla_{T}(\mc{Q}')=\Delta\).  Note that \(\mc{L}'(t)\) and \(\mc{Q}'(t)\) are rigid for \(t\notin\Delta\).
\end{ex}
\begin{cor}\label{cor.modCM}
Suppose \(A\) is a finite type Cohen-Macaulay \(k\)-algebra and \(S\) a noetherian \(k\)-algebra\textup{.} Put \(h\co S\ra T=S\ot_{k}A\)\textup{.} Let \(J=(f_{1},\dots,f_{n})\) be an \(A\)-sequence\textup{,} put \(B=A/J\) and let \(\bar{h}\co S\ra\bar{T}=S\ot_{k}B\) be the induced Cohen-Macaulay map\textup{.} Suppose \(\mc{N}\) is an \(S\)-flat finite \(\bar{T}\)-module and let \(0\ra\mc{L}\ra\mc{M}\ra\mc{N}\ra0\) be an \(\cat{MCM}\)-approximation of \(\mc{N}\) over \(h\)\textup{.} Assume \(\ob(T/(JT)^{2}\ra \bar{T}, \mc{N})=0\)\textup{.} Then 
\begin{equation*}
\nabla_{\bar{T}}(\mc{N})=\nabla_{T}(\mc{M})\cap\Supp\bar{T}\,.
\end{equation*}
\end{cor}
\begin{proof}
By Proposition \ref{prop.obsmodule} there is a lifting \(\mc{N}_{1}\ra\mc{N}\) of \(\mc{N}\) to \(T_{1}=T/(JT)^{2}\). It induces liftings \(\mc{N}_{1}(t)\ra\mc{N}(t)\) for all \(k\)-points \(t\) in \(\Supp\bar{T}\). The inclusion \(\bar{\tau}\co \xt{1}{\bar{T}(t)}{\mc{N}(t)}{\mc{N}(t)}\ra\xt{1}{T(t)}{\mc{M}(t)}{\mc{M}(t)}\) in \eqref{eq.xt1} commutes with the local Kodaira-Spencer maps. Proceed as in the proof of Corollary \ref{cor.modY}.
\end{proof}
\providecommand{\bysame}{\leavevmode\hbox to3em{\hrulefill}\thinspace}
\providecommand{\MR}{\relax\ifhmode\unskip\space\fi MR }
\providecommand{\MRhref}[2]{%
  \href{http://www.ams.org/mathscinet-getitem?mr=#1}{#2}
}
\providecommand{\href}[2]{#2}

\end{document}